\documentclass[11pt,openright]{report}
\usepackage[a4paper, twoside, bindingoffset=0.6cm, left=2.7cm,right=2.7cm,
            top=4cm,bottom=3.2cm,footskip=2cm]{geometry}
\usepackage{etex}

\usepackage{amsthm}
\usepackage{mathtools}
\usepackage{amssymb}
\usepackage{alltt}
\usepackage{fix-cm}
\usepackage[dvipsnames]{xcolor}

\usepackage{graphicx}
\usepackage{color}
\usepackage{comment}
\usepackage{url}
\usepackage{enumitem}
\usepackage[hidelinks]{hyperref}
\usepackage{pgf,tikz}
\usetikzlibrary{arrows,spy}
\usepackage{pgfplots}

\definecolor{darkgreen}{cmyk}{1.,0.,1.,0.2}
\definecolor{rahmen}{RGB}{0,73,114}
\definecolor{darkrahmen}{RGB}{10,73,134}
\definecolor{shadecolor}{RGB}{100,100,100}

\usepackage{fancyhdr}
\fancypagestyle{plain}{ %
  \fancyhf{} 
  
}

\usepackage[percent]{overpic}
\usepackage{subcaption}
\usepackage[ ]{titlesec}  %
\titleformat{\chapter}[display]
{ \normalsize \huge \bfseries \itshape \color{darkrahmen}}
{\flushright \normalsize \color{darkrahmen} { \fontsize{60}{60}\selectfont \color{darkrahmen} \mdseries \thechapter }  \vspace{2em}} {10 pt}{\fontsize{30}{30} }
\titleformat*{\section}{\Large\bfseries\sffamily\color{rahmen}}
\titleformat*{\subsection}{\large\bfseries\sffamily\color{rahmen}}
\titleformat*{\subsubsection}{\large\bfseries\sffamily\color{rahmen}}
\titleformat*{\paragraph}{\bfseries\sffamily\color{rahmen} \itshape}

\usepackage{tabu}
\usepackage{microtype}

\usepackage[english]{babel}
\usepackage[utf8]{inputenc}

\usepackage{multicol}
\usepackage{multirow}
\usepackage{arydshln}

\usepackage[sort,nocompress]{cite} 

\newcommand{\vectornorm}[1]{\left|\left|#1\right|\right|}
\newcommand{\rr}[0]{\mathbb{R}}
\newcommand{\zz}[0]{\mathbb{Z}}
\newcommand{\nn}[0]{\mathbb{N}}
\newcommand{\const}[0]{\text{const}}
\newcommand{\cover}[1]{\xRightarrow{\protect\mathmakebox[1.7em]{#1}}}
\newcommand{\longcover}[1]{\xRightarrow{\protect\mathmakebox[2.6em]{#1}}}
\newcommand{\longlongcover}[1]{\xRightarrow{\protect\mathmakebox[3.9em]{#1}}}
\newcommand{\longlongbackcover}[1]{\xRightarrow{\protect\mathmakebox[3.9em]{#1}}}
\newcommand{\longlonggencover}[1]{\xLeftrightarrow{\mathmakebox[3.9em]{#1}}}

\newcommand{\backcover}[1]{\xLeftarrow{\mathmakebox[1.7em]{#1}}}
\newcommand{\longbackcover}[1]{\xLeftarrow{\protect\mathmakebox[2.6em]{#1}}}
\newcommand{\gencover}[1]{\xLeftrightarrow{\mathmakebox[1.7em]{#1}}}
\newcommand{\longgencover}[1]{\xLeftrightarrow{\mathmakebox[2.3em]{#1}}}

\DeclareMathOperator{\sgn}{sgn}
\DeclareMathOperator{\inter}{int}

\DeclareMathOperator{\id}{id}
\DeclareMathOperator{\diam}{diam}

\DeclareMathOperator{\conv}{conv}

\newtheoremstyle{theorC}{\topsep}{\topsep}{\itshape \color{black}}{}{\color{rahmen}\bfseries}{.}{.5em}{}
\theoremstyle{theorC}

\newtheorem{thm}{Theorem}[section]
\newtheorem{cor}[thm]{Corollary}
\newtheorem{lem}[thm]{Lemma}
\newtheorem{prop}[thm]{Proposition}
\newtheorem{thmdefn}[thm]{Definition/Theorem}

\newtheoremstyle{defnC}{\topsep}{\topsep}{\color{black}}{}{\color{rahmen}\bfseries}{.}{.5em}{}
\theoremstyle{defnC}
\newtheorem{rem}[thm]{Remark}

\theoremstyle{defnC}
\newtheorem{defn}[thm]{Definition}

\theoremstyle{remark}

\date{}
\begin{document}
\begin{titlepage}
  \Large
  \centering
   {\scshape \large Faculty of Mathematics and Computer Science \\ of the Jagiellonian University \vspace{47mm}} \\                                                
  {\scshape \Large Aleksander Czechowski}\vspace{9mm} \\
  {\bfseries \sffamily \LARGE \color{darkrahmen} \LARGE Rigorous numerics for a singular perturbation problem \normalfont}\\
         \vspace{6mm}                                                                                                                                 
        {\Large \fontsize{13.5}{16}\selectfont A dissertation prepared under supervision of\\  prof. dr hab. Piotr Zgliczy\'nski} \vspace{35mm} \\ 
        {\includegraphics[width=7em,clip=true,trim =0mm 00mm 00mm 00mm]{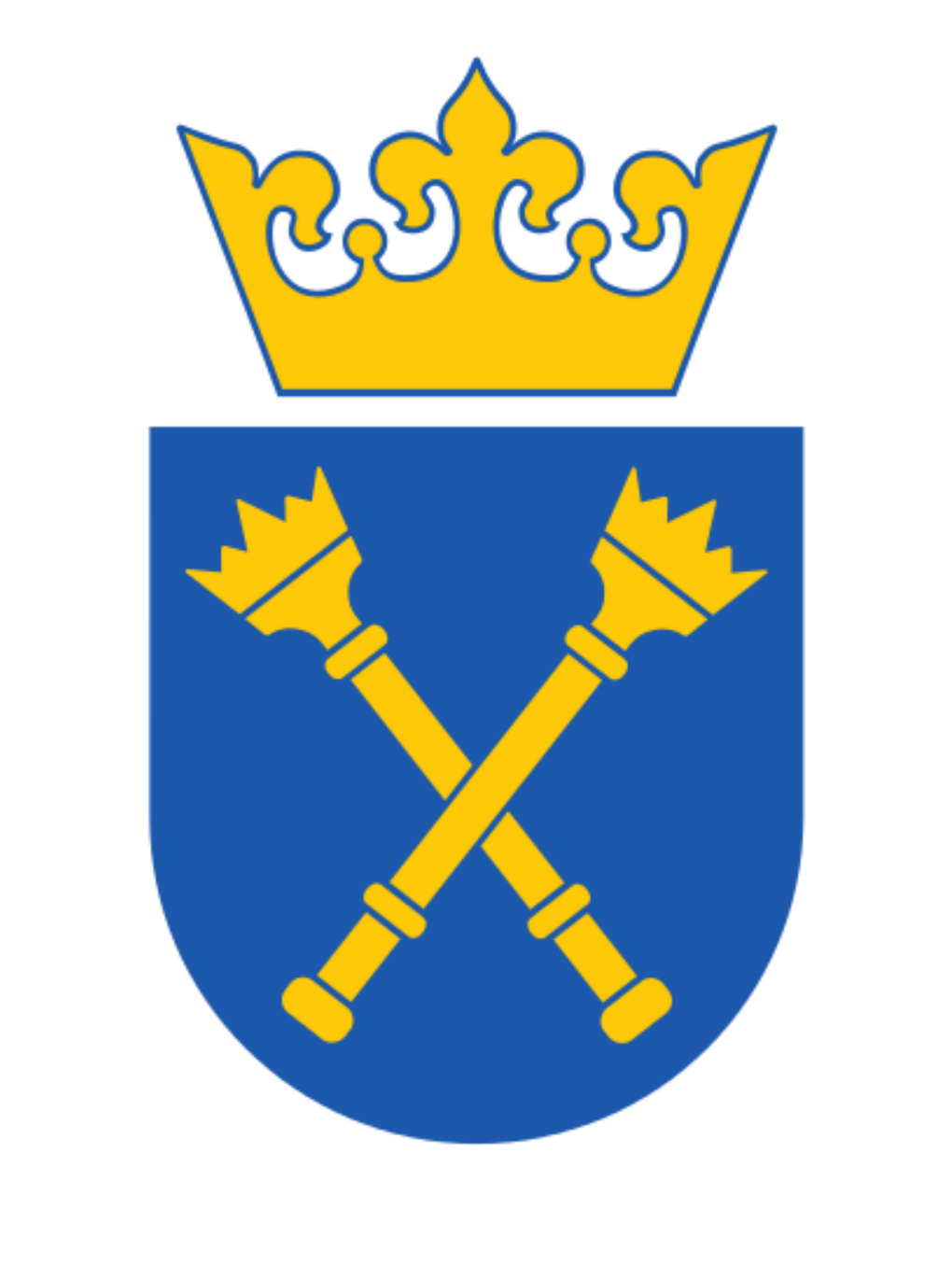}}\vfill
        {\large \scshape Krak\'ow 2015}
      \end{titlepage}

\clearpage \thispagestyle{empty} 

\par\vspace*{\fill}

\noindent During the major part of the research, which led to preparation of this thesis, the author 
was supported by the Foundation for Polish Science under the MPD Programme \emph{Geometry
and Topology in Physical Models}, co-financed by the EU European Regional Development Fund,
Operational Program Innovative Economy 2007-2013.

\vspace{1cm}

\noindent
Contents of this thesis have been previously partially published as a preprint

\noindent A. Czechowski and P. Zgliczyński. Existence of periodic solutions of the FitzHugh-Nagumo 
equations for an explicit range of the small parameter. arXiv:1502.02451, 2015.

\vspace{1cm}

\noindent \copyright \ Aleksander Czechowski 2015

\noindent All rights reserved
\vspace{3cm}

\chapter*{Abstract}
{
  Many mathematical models are governed by fast-slow systems. 
Such systems are difficult to analyze using standard numerical methods,
due to their stiffness, which is inversely proportional to the value of the so-called small parameter $\epsilon$.
At $\epsilon=0$ the problem decomposes into two independent lower-dimensional differential equations known
as the slow subsystem and the fast subsystem.
A famous technique, commonly referred to as geometric singular perturbation theory, can then be applied 
in an attempt to describe the dynamics in the full system,
based on the properties of these two subsystems.
In particular, such methods can be used to prove existence of certain homoclinic and periodic trajectories in a given system 
for $\epsilon \in (0, \epsilon_0]$, $\epsilon_0$ ``small enough''.

In this thesis we propose a framework, which allows to replace the words ``small enough''
with an explicit value of $\epsilon_0$. 
Our approach is based on a combination of two topological methods.
The first one is the method of covering relations, which is used to describe how sets are mapped across each other under Poincar\'e maps.
This method has proved itself to be very effective in several previous studies,
though it requires rigorous integration and cannot be applied in the stiff region driven by the dynamics of the slow subsystem.
There, we apply the second tool: the method of isolating segments. 
It allows to retrieve topological information about Poincar\'e maps based solely on the topology of the vector field,
without numerical integration. 
This approach is often more cumbersome to apply than covering relations. 
However, in our scenario we can employ it easily, by exploiting the high hyperbolicity generated by the slow dynamics.   

We state a definition of isolating segments convenient to apply in autonomous systems 
and in conjunction with covering relations.
Then, we prove several theorems on how suitable chains of isolating segments and covering relations
imply existence of periodic and connecting orbits in a given (not necessarily fast-slow) system.
Finally, we apply our theorems and conduct a computer assisted proof
showing the existence of traveling waves in the FitzHugh-Nagumo model in an explicit range of $\epsilon$.
For the case of the traveling pulse, the proof involves a local estimate on the unstable and stable manifold of the stationary point
in the traveling wave equation, which is performed via an $\epsilon$-dependent cone field.

Additionally, we extend the parameter range of existence of the periodic wave train by a rigorous continuation
procedure based on covering relations, with a varying number of transversal sections. 
We achieve $\epsilon_0$ large enough, so that a standard proof based on the interval Newton-Moore method 
applied to a sequence of Poincar\'e maps succeeds at that parameter value.
 }
\clearpage \thispagestyle{empty} 
\par\vspace*{\fill}

{ \large \tableofcontents }

{
\titlespacing*{\chapter}{0pt}{2em}{11em}
\chapter{Introduction}\label{chap:intro}

\section{Background}

The work reported in this thesis lies on the intersection of several areas of research
in the qualitative theory of dynamical systems. 
This section serves as a very brief introduction to these subjects.

Let us first recall the definitions of basic invariant sets in ordinary differential equation (ODEs).
Consider an ODE 
\begin{equation}\label{ODEintro}
  \dot{x}=F(x), \quad x \in \rr^N,
\end{equation}
with a smooth right-hand side. A \emph{stationary point/solution} (or an \emph{equilibrium}) of~\eqref{ODEintro} is 
a solution satisfying $\frac{dx(t)}{dt}=0$ for all $t \in \rr$;
a \emph{periodic orbit} is a solution satisfying $x(t+T) = x(t)$ for all $t \in \rr$ and some $T >0$
and a \emph{connecting orbit} (between equilibria) is a nonstationary solution for which $\lim_{t \to \infty} x(t)$ and $\lim_{t \to -\infty} x(t)$ exist and are finite.
In such case the points $\lim_{t \to \pm \infty} x(t)$ are equilibria of~\eqref{ODEintro}; 
if they are equal, we say that $x$ is a \emph{homoclinic orbit}, elsewise we call it a \emph{heteroclinic orbit}.

Throughout this section we, somewhat unfortunately, refer to several theorems 
stated in subsequent parts of the thesis, and in particular statements of the main
results are postponed to Section~\ref{sec:results}.
Our decision for such structure was motivated as follows.
Statements of our main theorems (especially Theorems~\ref{thm:main1} and~\ref{thm:main4})
are merely examples of application of a machinery developed through the whole thesis,
and require a proper background motivation. On the other hand,
as we introduce the background theory, we would like to comment on how our main results fit into it.
To avoid unnecessary repetitions we refer ahead and rely on goodwill of the reader
to look up the referenced material, wherever it is necessary.

\subsection{Fast-slow systems}

\emph{Fast-slow systems} are systems of ODEs of the form 
\begin{equation}\label{slowfastIntro}
  \begin{aligned}
    \dot{x}&= f(x,y,\epsilon),\\
    \dot{y}&=\epsilon g(x,y,\epsilon),
\end{aligned}
\end{equation}
with $f: \rr^n \times \rr^m \times \rr \to \rr^n$, $g: \rr^n \times \rr^m \times \rr \to \rr^m$ smooth, and $\epsilon$ such that $0<\epsilon \ll 1$.
The parameter $\epsilon$ is usually referred to as the \emph{small parameter}.
Consequently, the velocities of $y$'s are most often much lower than these of $x$'s, 
hence the former are called the \emph{slow variables} and the latter the \emph{fast variables}.
Such equations are difficult to study with numerical methods,
as small time steps are needed to accurately capture dynamics in $y$.
This phenomenon, known as \emph{stiffness}, inevitably leads to
long integration times and significant propagation of numerical errors.  
In particular, sensitivity to initial conditions encountered in simulation of chaotic or unstable trajectories of saddle type
becomes even a bigger problem than usual. 

What is an obstacle to numerical methods, turns out to be advantageous 
to a certain pen-and-paper analysis, which we will now outline. The first step is to separate the time scales. 
By setting $\epsilon$ to 0 
one obtains a lower dimensional system 
\begin{equation}
    \dot{x}= f(x,y,0),
\end{equation}
parameterized by $y$, called the \emph{fast subsystem}. The fast subsystem
can approximate the dynamics of the full system with some degree of accuracy
in regions away from the set of its stationary solutions
\begin{equation}
  C_0 := \{ (x,y) \in \rr^{n+m}:\ f(x,y,0)=0 \}
\end{equation}
referred to as the \emph{slow manifold}\footnote{Sometimes in literature this manifold is referred to as the critical manifold, and the term
slow manifold refers to normally hyperbolic invariant manifolds in its close proximity.}.
At the slow manifold one considers the \emph{slow subsystem} given by the following differential-algebraic equation:
\begin{equation}
  \begin{aligned}
    0 &= f(x,y,0),\\
    \dot{y} &= g(x,y,0).
  \end{aligned}
\end{equation}

A set of methods developed in 70s and 80s, known as \emph{geometric singular perturbation theory} (GSPT) was designed to give qualitative information on the
dynamics of the full system~\eqref{slowfastIntro} by studying its slow and fast subsystems.
The theory was founded on a set of results of Fenichel~\cite{Fenichel} and
its good overview is presented in~\cite{JonesBook} and~\cite{KuehnBook}, Chapter 3, see also~\cite{Arnold}.
Typical assertions in applications of GSPT to~\eqref{slowfastIntro} are valid for $\epsilon \in (0,\epsilon_0]$, where $\epsilon_0$ is ``small enough'' 
and unspecified. 

We note that the results on existence of bounded orbits in systems of the form~\eqref{slowfastIntro}, for $\epsilon \in (0,\epsilon_0]$, $\epsilon_0$ small enough 
can be proved using alternative, topological techniques, developed even earlier than GSPT~\cite{Conley, Carpenter} (cf. Subsection~\ref{subsec:backsegments}).
One of the primary goals of this thesis is to propose a computer assisted method of this type that allows to reproduce these results with
an \emph{explicit} upper bound $\epsilon_0$.  
Our sample results of this type are Theorems~\ref{thm:main1},~\ref{thm:main4} concerning the existence of periodic
and homoclinic orbits in the FitzHugh-Nagumo fast-slow system, presented in Section~\ref{sec:results}.

\subsection{Rigorous numerics for differential equations}

Numerical methods are a quick and convenient way to provide insight on dynamics of nonlinear problems.
However, can we trust nonrigorous computations?  
Many ODEs exhibit a sensitive dependence to initial conditions, and  
numerical integration can produce vastly differing results for very slight variations in initial values.
Moreover, it has been well-documented that a cumulation of round-off errors 
may occasionally lead to appearance of artificial, spurious solutions~\cite{humphries, sweby}.

Recent developments in the field of \emph{rigorous numerics} address these issues.
Its primary principles are that the computations are performed on sets rather than points;
and that the result set of each operation has to contain all actual solutions, in other words the algorithms need to form \emph{enclosures} for the operations.
It is particularly convenient to use sets constructed from intervals with representable (e.g. \texttt{double}) endpoints.
By an appropriate adjustment of the processor rounding settings,
the enclosure principle can be realized for the elementary arithmetic operations of addition, subtraction, multiplication and division
on intervals as follows:
\begin{equation}
  \begin{aligned}
    [a,b] + [c,d] &\subset \left[\downarrow(a+c), \uparrow(b+d)\right],\\
    [a,b] - [c,d] &\subset \left[\downarrow(a-d),\uparrow(b-c)\right],\\
    [a,b] \cdot [c,d] &\subset \left[\min(\downarrow ac, \downarrow ad, \downarrow bc, \downarrow bd ), \max(\uparrow ac, \uparrow ad, \uparrow bc, \uparrow bd ) \right],\\
    [a,b] / [c,d] &\subset [a,b] \cdot [\downarrow(1/d), \uparrow(1/c)], \quad \text{iff}\ 0 \notin [c,d]
  \end{aligned}
\end{equation}
where $\uparrow$ and $\downarrow$ are the operations of rounding up and down, respectively, to the nearest representable number.
Such framework is known under the name of \emph{interval arithmetics}, and can be used to provide rigorous enclosures of outputs
for all algorithms formed by these elementary arithmetic operations.

\subsubsection{Rigorous integration}

Interval arithmetics can be employed to give quantitative information on a given dynamical system 
in terms of enclosures of vector fields, maps or their derivatives. 
It can also be used to produce enclosures of solutions of initial value problems for differential equations.
This procedure is known as \emph{rigorous integration} 
and heavily employed throughout the thesis. Therefore we will now outline
a basic rigorous integration algorithm, based on expansion of the solution into Taylor series.

Consider the ODE~\eqref{ODEintro} and assume that $F: \rr^N \to \rr^N$ is analytic and given by elementary functions and elementary arithmetic operations.
Values of the solution operator $\varphi(t,x)$ can be enclosed as follows
\begin{equation}\label{Taylor}
  \varphi(h,[X_0]) \subset [X_0]  + \sum_{k=1}^{r}  \frac{\partial^{k} \varphi(0,[X_0])}{\partial t^k} \cdot \frac{h^k}{k!} 
  + \frac{\partial^{r+1} \varphi(0,\varphi(\theta h,[X_0]))}{\partial t^{r+1}} \cdot \frac{h^{r+1}}{(r+1)!}, \ \theta \in [0,1],
\end{equation}
where $h>0$ and the initial condition $[X_0]$ can be a point as well as a product of intervals.

The following remarks hold:
\begin{itemize}
  \item higher order derivatives of the flow with respect to $t$ can be obtained by a repeated differentiation 
    of the equality $\frac{\partial \varphi(0,x)}{\partial t} = F(x)$. An efficient way to implement this by
    using \emph{automatic differentiation}, see~\cite{RallBook};
  \item an enclosure for $\varphi(\theta h,[X_0])$ can be generated by the following reasoning.
    If 
    \begin{equation}
      [Y]:=[X_0] + [0,h] F([Z]) \subset \inter [Z]
    \end{equation}
    for some product of intervals $[Z]$, then from the integral form of the solution to~\eqref{ODEintro}
    we obtain
    \begin{equation}
      \varphi([0,h],[X_0]) \subset [Y].
    \end{equation}
    Such enclosure is very coarse, but for large $r$ the last term in formula~\eqref{Taylor} will still be negligibly small.
\end{itemize}

Similar methods can be applied to produce enclosures of time step maps for partial derivatives of the flow.
We remark that integration performed by a sequential evaluation of formula~\eqref{Taylor} in interval arithmetics
can lead to a rapid accumulation of overestimates, also known as the \emph{wrapping effect}. 
Different approaches can be used to suppress this problem; one of them is the \emph{Lohner algorithm}~\cite{Lohner}, 
which represents the sets in a suitable evolving coordinate frame.
Rigorous integration of the flow and its variational equations
based on the Lohner algorithm, and resulting rigorous computation of Poincar\'e maps\footnote{First return maps to transversal sections, see Subsection~\ref{subsec:poinc}.}
and their derivatives
is implemented in e.g. the CAPD library for rigorous computations~\cite{CAPD}
and documented in~\cite{Zgliczynski, CrLohner}.

In this thesis we use CAPD integrators for our proofs, however one should keep in mind that 
our methods are independent of the integration scheme.
Therefore, later on we will often use Poincar\'e maps without much reference to underlying numerical procedures,
and the reader should be aware that procedures for rigorous computation of enclosures of their values
are publicly available, and at our disposal.

\subsubsection{Computer assisted proofs}

Interval arithmetics can also be employed to give qualitative information, in form of computer assisted proofs.
For dynamical systems the typical setting is when we would like to 
to verify some qualitative scenario in a specific system, based on an abstract theorem,
which assumptions would be tedious to check with pen-and-paper calculations.
If these assumptions are amenable to verification by a finite algorithm 
and endure a certain amount of overestimates, then there is hope to conduct a proof
with assistance of interval arithmetics.
Although such proofs are limited to concrete systems with preset values of parameters,
they are usually proofs of concept.
Methods are often of a greater scientific value than the properties of the particular equation,
and a successful implementation in an emblematic example ``proves'' that they should
be applicable to all problems with similar structure.

In the context of computer assisted proofs for differential equations, rigorous integration becomes
an immensely useful tool, as it allows to recast problems for continuous systems as problems for (Poincar\'e) maps.
For instance, one can use rigorous integration to verify assumptions of the Brouwer fixed point theorem 
for a Poincar\'e map of a given ODE to prove the existence of an (apparently stable) periodic solution.
Rigorous integration has been successfully used to prove the existence of periodic orbits~\cite{Galias2, Zgliczynski},
connecting orbits~\cite{szczelir, WilczakZgliczynski2}, bifurcations~\cite{WilczakZgliczynski, Kokubu} and
chaos~\cite{MischaikowMrozek, GaliasZgliczynski} in various ODEs. 
One of the famous results obtained that way was the verification of existence of a strange attractor in Lorenz equations, solving Smale's 14th problem~\cite{Tucker}.
Certain scenarios have been also verified in dissipative partial differential equations (PDEs)~\cite{CyrankaZgliczynski, ZgliczynskiKS2, ZgliczynskiKS3, ArioliKoch2}.
References provided by us in this paragraph are in no way meant to be complete.

We remark, that there are other ways to design computer assisted proofs in differential equations;
one of them relies on rewriting the problem as a zero of an infinite-dimensional operator in some function space, e.g.~\cite{Reinhardt, Lessard},
the other uses solely the topology of the vector field e.g.~\cite{Wanner}.

It is clear that most computations 
using such methods will result in overestimates, therefore it is impossible to use interval arithmetics to check equalities.
However, a finite number of strict inequalities is already suitable for such verification.
Expressions in such inequalities typically vary in a continuous fashion with system parameters;
therefore only assumptions for \emph{structurally stable} scenarios\footnote{That is scenarios, which persist under small perturbations to the system of some given class (e.g. $C^1$).}
can be readily verified in such setup.
Problems, which are not structurally stable (such as bifurcations) usually require 
additional care in preparation of the inequalities.

Little attention has been given so far to computer assisted proofs for fast-slow systems.
One of the reasons is their numerically stiff nature, however this is not the biggest hurdle.
Typically solutions of interest to problems of the form~\eqref{slowfastIntro} 
exist for ranges $\epsilon \in (0,\epsilon_0]$ and are nonvanishing in the slow (i.e. $y$) variable.
Designing a proof for such range yields a problem that is certainly not structurally stable, as for $\epsilon = 0$ there can be no nonconstant in $y$ solutions.
One of the main results of this thesis is a method to bypass this instability and 
treat the whole range $\epsilon \in (0,\epsilon_0]$ in a computer assisted proof, with $\epsilon_0$ explicit
(e.g. Theorems~\ref{thm:main1}, \ref{thm:main4}).
For this purpose, the inequalities in assumptions of our theorems have to be carefully crafted,
so the parameter $\epsilon$ can be factored out before applying interval arithmetic verification.
To this end we have to resign from rigorous integration for certain parts of the proof
and substitute them with~\emph{isolating segments} described in the subsequent subsection.

In previous computer assisted studies of fast-slow systems authors would consider only a single value of $\epsilon$~\cite{ArioliKoch};
or use computer assistance only to aid the perturbation methods and derive results for $\epsilon_0$ ``small enough''~\cite{Haiduc}.
To the best of our knowledge the only published paper that treats such type of range with computer assistance 
is on the existence of homoclinic tangencies~\cite{GuckenheimerJohnson}.
There, the authors exploit a different mechanism, most likely inapplicable in our problems -- in their scenario the proof for the whole range follows from 
the (computer assisted) proof for its upper bound. 
Results similar to ours, with use of similar methods have been recently, independently obtained by Matsue, and released as a preprint~\cite{Matsue} 
-- in Subsection~\ref{subsec:matsue} we compare our approaches and address the question 
of chronology in which the results of us and Matsue appeared online.

\subsection{The topological method of isolating segments}\label{subsec:backsegments}

A key component in our proofs is the construction of certain compact sets called isolating segments.
The precise definition of these objects is postponed to Subsection~\ref{sec:segments}
and we refer the reader to Figure~\ref{segmentCovFig:A} given therein to grasp some intuition behind their geometry.
In short, isolating segments are solids diffeomorphic to hypercubes, and akin to isolating blocks from Conley index theory~\cite{Easton, Conley2}.
In our definition (Definition~\ref{defn:isegment}) the variables given by the diffeomorphism to a cube induce new directions 
some of them we label as exit, some as entry and one as central. 
It is required, that the faces in the exit direction are immediate exit sets for the flow and the faces in the entry direction are immediate entrance sets. 
Moreover, the projection onto the one-dimensional central direction has to be monotone along the trajectories of the flow.
The analogy between isolating segments and isolating blocks can be seen as follows. 
While isolating blocks give information on the structure of invariant sets in their interior, based on the direction of the vector field on their boundary,
isolating segments are designed to provide the same type of information for certain Poincar\'e maps
defined on sections containing their faces.
We emphasise that no integration of initial value problems is needed to construct an isolating segment and deduce this information.
The computational cost of placing a segment lies mainly in evaluation of scalar products of normals to their faces with the vector field.

Isolating segments were first introduced by Srzednicki~\cite{Srzednicki, SrzednickiWojcik} to obtain information
about the period shift map in time-periodic nonautonomous ODEs. The definition of Srzednicki
was restricted to nonautonomous equations with the central direction fixed to be the time direction, in particular it
was not adaptable to global problems in autonomous equations.
In this thesis we propose a definition of isolating segments suitable for use in autonomous ODEs 
and in conjunction with the method of~\emph{covering relations} of Zgliczy\'nski~\cite{GideaZgliczynski} applied to Poincar\'e maps.
In short, the method of covering relations is a transversality condition describing how certain compact
sets (called \emph{h-sets}) are mapped across each other. 
To verify a covering relation in and ODE one usually needs to perform rigorous integration.
We relate covering relations to isolating segments (Theorems~\ref{is:1} \ref{is:2}, \ref{is:3} and~\ref{is:4})
and include them in abstract theorems on existence of periodic and homoclinic orbits (Theorems~\ref{thm:1}, \ref{thm:2}, \ref{thm:hom1} and~\ref{thm:hom2}).

Our definition of an isolating segment is amenable to rigorous verification in interval arithmetics
and especially useful in verification of existence of strongly hyperbolic orbits.
This makes it ideal to employ in detection of strongly expanding orbits, where 
rigorous integration suffers from a rapid growth of error bounds.
In fast-slow systems such situation is encountered in the proximity of the slow manifold.
An important ingredient in the proofs of Theorems~\ref{thm:main1},~\ref{thm:main4}
is a construction of sequences of isolating segments enclosing fragments of branches of the slow manifold,
which allows us to avoid the problems connected with rigorous integration in that region.
By a suitable placement of the segments,
parameter $\epsilon$ causing the major hindrance of structural instability of the problem can be factored out in some computations.
This would have been impossible to achieve using rigorous integration alone.

We remark that in the early proofs of existence of periodic and homoclinic solutions
to the FitzHugh-Nagumo equation for small $\epsilon$ by Conley and Carpenter~\cite{Carpenter, Conley}, the authors
employed sequences of isolating blocks to track the solutions. 
At a certain level of abstraction our topological methods probably coincide with the ones of Carpenter and Conley.
In that context, the contribution of this thesis is that we provide a framework of definitions and theorems which  
can be efficiently applied in computer assisted proofs of existence for explicit ranges of $\epsilon$.

The original ideas of Srzednicki are still a subject of active research and evolved into computer assisted proofs
employing short-time integration~\cite{SrzednickiMrozek, SrzednickiWeilandt}.
However, at the current stage their applications are still limited to low-dimensional nonautonomous ODEs.

\subsubsection{Isolating segments in ill-posed PDEs -- a digression}\label{subsec:Boussinesq}

In this subsection we give a very brief digest of results from~\cite{CzechowskiZgliczynski2},
which were obtained by the author during his PhD studies, but have not been included in this thesis
to keep its exposition compact and devoted to fast-slow systems.

In the above-mentioned paper we consider
the nonautonomously forced Boussinesq equation~\cite{Boussinesq}:
\begin{equation}\label{eq:bsq}
  \frac{\partial^2 u}{\partial t^2} = \frac{\partial^2 u}{\partial x^2} + \beta \frac{\partial^4 u}{\partial x^4} + \sigma \frac{\partial^2 (u^2)}{\partial x^2} + \epsilon f(t,x),
\end{equation}
with $t,x \in \rr$ and functions $u$ and $f$ $2\pi$-periodic and even in $x$.
In addition, we assume that $f$ is smooth and $T$-periodic in $t$ for some $T>0$.

For $\beta>0$ the equation~\eqref{eq:bsq} suffers from ill-posedness, i.e. almost
all solutions of initial value problems blow up to infinity in their high Fourier modes, and therefore immediately lose smoothness.
Therefore, any attempts of numerical (rigorous or not) integration of such system
to find time-periodic solutions is bound to fail.
This behavior is caused by strong hyperbolicity coming from the linear part.
In the Fourier basis the linearized equation is of a block-diagonal form with 2x2 blocks.
Almost all blocks have eigenvalues of opposite signs, which
grow polynomially with respect to the block coordinate.
This causes a strong expansion in high frequencies, both forward and backward in time, which accumulates into nonlinearities 
and initiates the blowup.

We adapted the ``old'' definition of periodic isolating segments for nonautonomous ODEs from~\cite{Srzednicki}
to the framework of~\emph{self-consistent bounds} (cf.~\cite{ZgliczynskiMischaikow, ZgliczynskiKS3})
to produce a similar tool for infinite-dimensional nonautonomous systems.
Then, we applied it to give a computer assisted proof of existence of smooth $T$-periodic solutions of~\eqref{eq:bsq} for certain choices of $f$, $\sigma$, $\beta>0$
and certain (small, explicit) ranges of $\epsilon$.
High hyperbolicity of the linear part allowed us to construct isolating segments in a suitable diagonalization without much difficulty.
These solutions continue from the zero equilibrium at $\epsilon=0$ (a trivial periodic solution itself) and as such
are formed by a \emph{regular perturbation}, contrary to singular perturbations considered in this thesis.

This is the first computer assisted result of this type for an ill-posed system that we are aware of.
We remark that the choice of time-periodic forcing is somewhat artificial;
in future we would like to perform a similar proof for smooth, periodic, non-stationary solutions of the ``standard'' autonomous
Boussinesq equation (i.e.~\eqref{eq:bsq} for $\epsilon=0$), perhaps by applying 
(a suitable modification of) isolating segments for autonomous systems, defined in this thesis.
\pagebreak

\section{Main results}\label{sec:results}
\subsection{The FitzHugh-Nagumo model}
The FitzHugh-Nagumo model with diffusion
\begin{equation}\label{RDiff}
  \begin{aligned}
    \frac{\partial u}{\partial \tau} &= \frac{1}{\gamma} \frac{\partial^{2} u}{\partial x^{2}} + u(u-a)(1-u) - w, \\
    \frac{\partial w}{\partial \tau} &= \epsilon (u - w),
 \end{aligned}
\end{equation}
was introduced as a simplification of the Hodgkin-Huxley model
for the nerve impulse propagation in nerve axons~\cite{FitzHugh, Nagumo}.
The variable $u$ represents the axon membrane potential and $w$ a slow negative feedback.
\emph{Traveling wave} solutions (that is solutions that can be represented as maps of argument $x+\theta \tau$ only, for some $\theta \neq 0$)
of~\eqref{RDiff} are of particular interest in neurobiology
as they resemble an actual motion of the nerve impulse~\cite{Hastings}.

By plugging the traveling wave ansatz $(u,w)(\tau ,x) = (u,w)(x+\theta \tau) = (u,w)(t)$, $\theta>0$
and rewriting the system as a set of first order equations we arrive at an ODE
\begin{equation}\label{FhnOde}
     \begin{aligned}
       u'&=v, \\
       v'&=\gamma(\theta v - u(u-a)(1-u) + w), \\
       w'&= \frac{\epsilon}{\theta} (u - w),
     \end{aligned}
\end{equation}
to which we will refer to as the \emph{FitzHugh-Nagumo system} or the \emph{FitzHugh-Nagumo equations}.
The FitzHugh-Nagumo system is a fast-slow system with two fast variables $u,v$ and one slow variable $w$. 
The parameter $\theta$ represents the wave speed and, as usual, the parameter $\epsilon$ is the small parameter, so $0 < \epsilon \ll 1$.
To focus our attention, following \cite{ArioliKoch, Champneys, GuckenheimerKuehn} we set the two remaining parameters to
\begin{align}\label{eq:parameters}
  a := 0.1, \qquad  \gamma := 0.2,
\end{align}
throughout the rest of the thesis.

Aside from its physical meaning, the FitzHugh-Nagumo system 
is one of the most prominent examples of its class, along with the van der Pol system.
Their analysis led to many important methodological innovations in the field of multiple time scale dynamics,
applicable to a wide range of biological and physical systems, and driven the development of large portions of GSPT, see~\cite{JonesBook,KuehnBook}.

Bounded solutions of~\eqref{FhnOde} yielding different wave profiles have been studied by many authors both rigorously and numerically, see
\cite{Conley, Carpenter, Hastings2,
Hastings3, Maginu, Mischaikow, Kopell, Ambrosi, ArioliKoch, Jones, Yanagida, Nagumo, Champneys, GuckenheimerKuehn} and references given there.
Periodic orbits leading to \emph{periodic wave trains} exist for an open range of $\theta$'s and were treated in \cite{Conley,Hastings2,Hastings3,Maginu,Mischaikow, ArioliKoch};
\emph{traveling pulses} generated by homoclinic orbits exist for two isolated values of $\theta$, and are sometimes referred to as the \emph{slow pulse} and the \emph{fast pulse}.
The slow pulse is generated by a regular perturbation from a homoclinic of the fast subsystem~\cite{Krupa} and forms
an unstable wave~\cite{Flores}; as such it will not be of our interest in this thesis.
Existence of the fast pulse was proved in~\cite{Carpenter,Hastings3, Kopell, Ambrosi, ArioliKoch}.
Stability of waves was discussed in~\cite{Jones,Maginu,Yanagida, ArioliKoch}.
Proofs of existence use various methods, but most share the same perturbative theme\footnote{In~\cite{Ambrosi, ArioliKoch} authors
use non-perturbative computer assisted methods for a single value $\epsilon=0.01$ where the system becomes a regular, although stiff ODE, cf. Subsection~\ref{subsec:ArioliKoch}.}.
We outline it below, first for the periodic orbit.

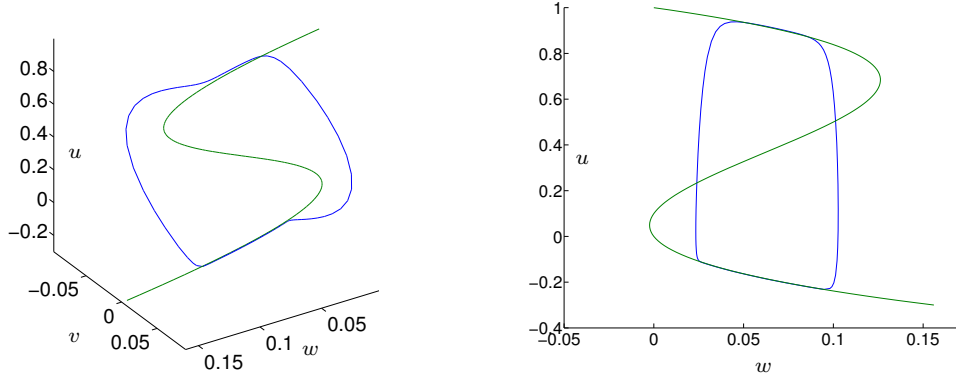
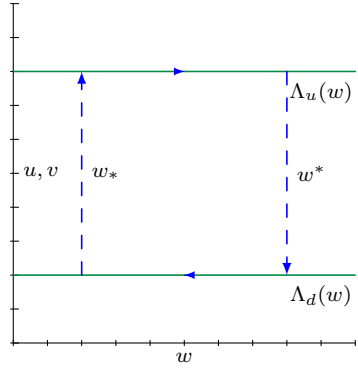
\begin{figure}[t]
  \centering
  \begin{subfigure}{0.48\textwidth}
    \centering
    \begin{overpic}[width=0.7\textwidth, clip=true, trim = 60mm 20mm 100mm 0mm]{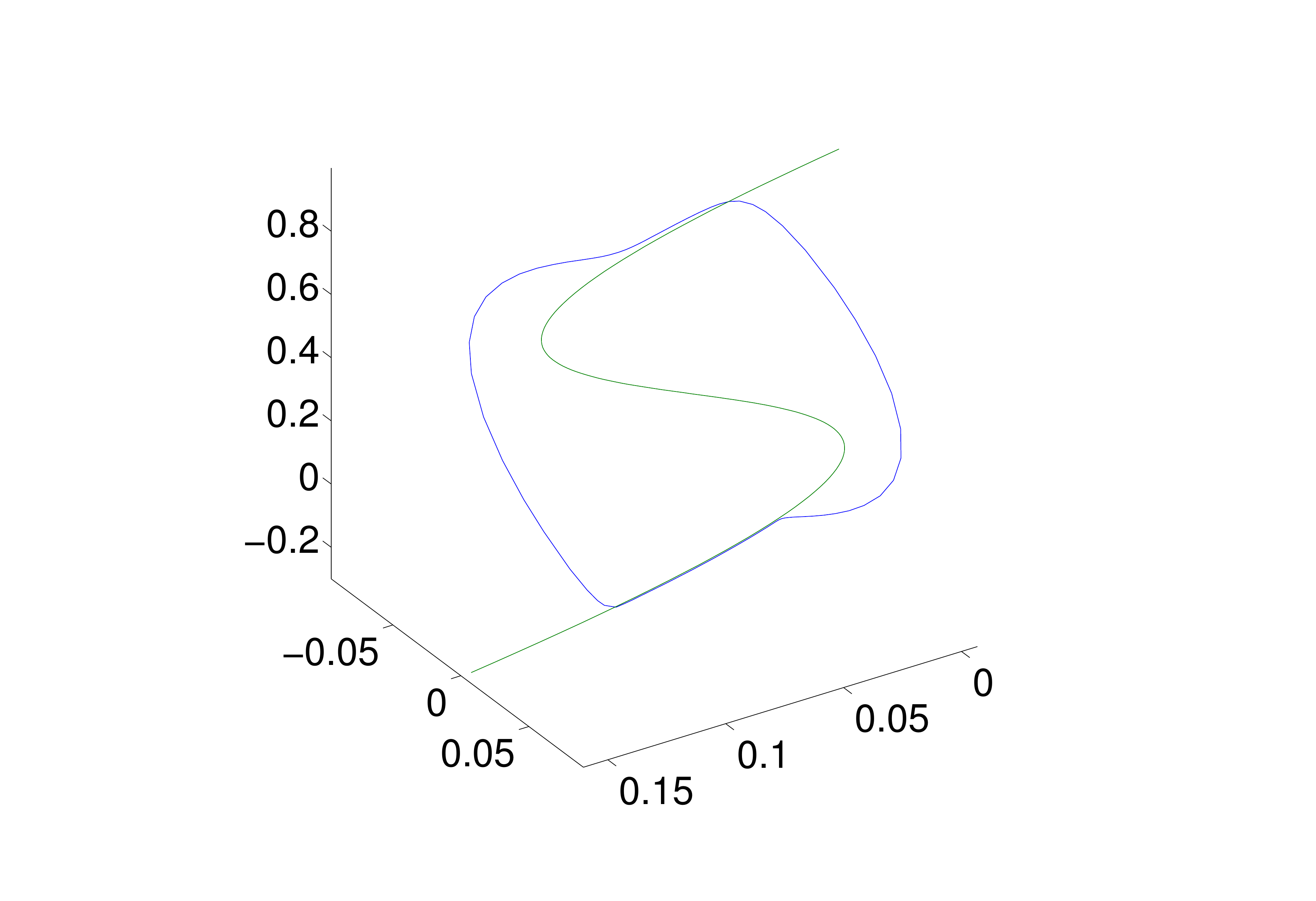}
      \put (17,52) {\scriptsize{$u$}}
      \put (17,9) {\scriptsize{$v$}}
      \put (71,6){\scriptsize{$w$}}
    \end{overpic}
    \caption{Approximate periodic orbit for $\epsilon=0.001$ in blue.}
  \end{subfigure} 
  \begin{subfigure}{0.48\textwidth}
    \centering
    \begin{overpic}[width=0.78\textwidth, clip=true, trim = 25mm 0mm 70mm 10mm]{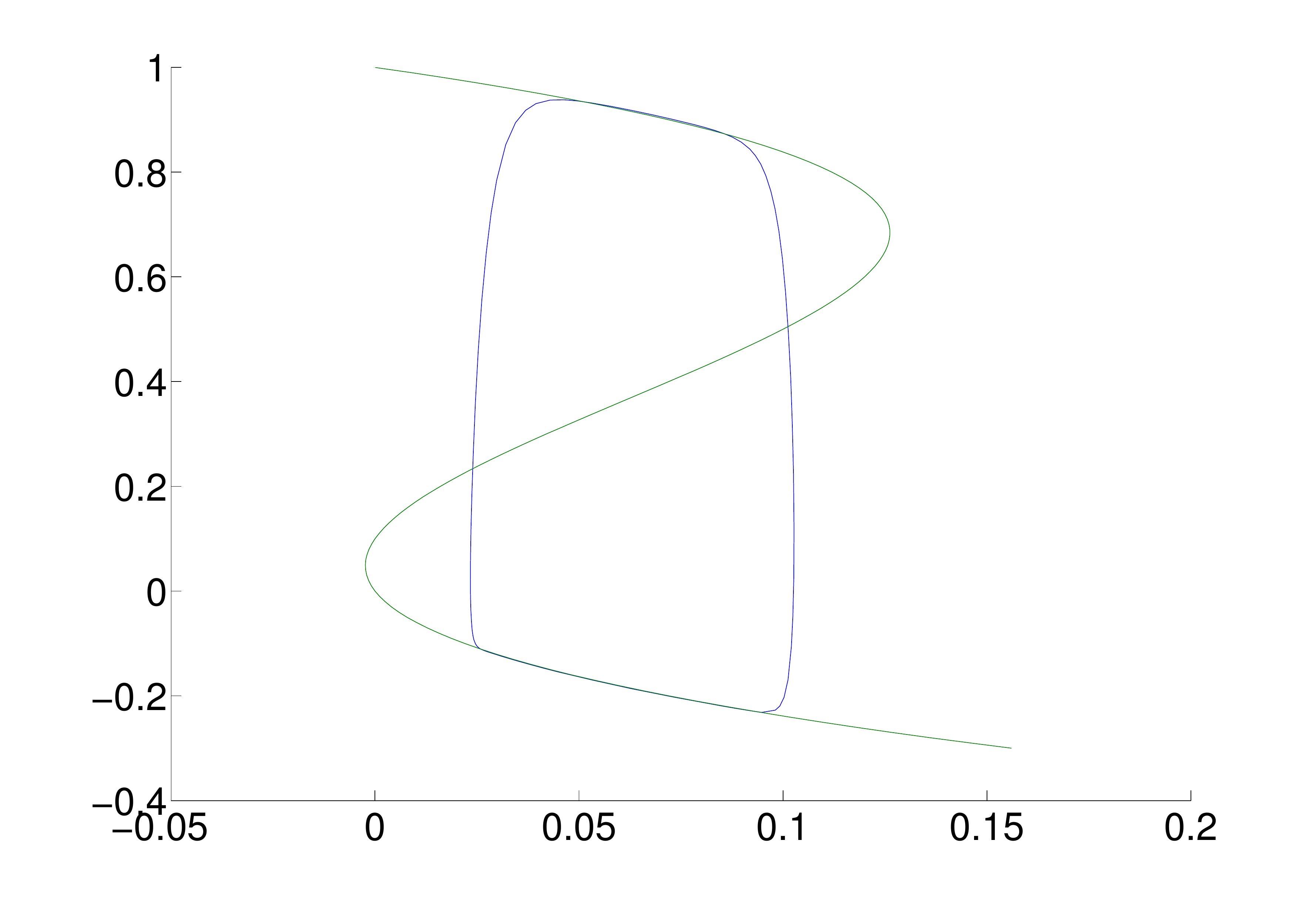}
      \put (11,49) {\scriptsize{$u$}}
      \put (53, 0) {\scriptsize{$w$}}
    \end{overpic}
    \caption{Projection onto $(u,w)$ plane.}
  \end{subfigure}
  \\
  \begin{subfigure}{0.45\textwidth}
    \centering
      \begin{tikzpicture}[line cap=round,line join=round,>=latex,x=4.5mm,y=4.5mm]
 
  \draw[color=black] (-0.,0.) -- (10,0.);
  \foreach \x in {0,...,10}
  \draw[shift={(\x,0)},color=black] (0pt,1pt) -- (0pt,-1pt) node[below] {};

  \draw[color=black] (0.,0.) -- (0.,10);
  \foreach \y in {0,...,10}
  \draw[shift={(0,\y)},color=black] (2pt,0pt) -- (-1pt,0pt) node[left] {};
 

  \clip(0,-1) rectangle (10,10);

  \draw [color=darkgreen,semithick,domain=0.:10.] plot(\x,{8});
  \draw [color=darkgreen,semithick,domain=0.:10.] plot(\x,{2});

  \draw [->,semithick,dash pattern=on 5pt off 5pt,color=blue] (2.,2.) -- (2.,8.);
  \draw [->,semithick,dash pattern=on 5pt off 5pt,color=blue] (8.,8.) -- (8.,2.);

  \draw [->,color=blue] (4.99,8.) -- (5.01,8.);
  \draw [->,color=blue] (5.01,2.) -- (4.99,2.);

  \draw (9.,8.) node[anchor=north] {\scriptsize{$\Lambda_u(w)$}};
  \draw (9.,2.) node[anchor=north] {\scriptsize{$\Lambda_d(w)$}};

  \draw (2.,5.) node[anchor=west] {\scriptsize{$w_{*}$}};
  \draw (8.,5.) node[anchor=west] {\scriptsize{$w^{*}$}};

  \draw (5.,0.) node[anchor=north] {\scriptsize{$w$}};
  \draw (0.,5.) node[anchor=west] {\scriptsize{$u,v$}};

\end{tikzpicture}
      \caption{A schematic drawing of the singular periodic orbit.}
  \end{subfigure}
  \caption{A numerical approximation of the periodic orbit close to the singular orbit, the slow manifold in green.}
  \label{orbitFig}
\end{figure}

Consider the limit equation at $\epsilon = 0$. There, the velocity of $w$
is zero and the phase space can be fibrated into a family of two-dimensional fast subsystems parameterized by $w$.
These subsystems serve as a good approximation to the system with $\epsilon > 0$ small, except for regions of phase
space near the Z-shaped slow manifold
\begin{equation}
  C_{0} = \{(u,v,w):\ v=0,\ w = u(u-0.1)(1-u) \},
\end{equation}
where the velocities of fast variables become small and the slow flow takes over.
For a range of $w$ the slow manifold has exactly three branches - by looking from a perspective of the respective $u$ values the lower and the upper one are formed by saddles,
and the middle one is formed by sources. We denote the upper/lower branches of saddles by $\Lambda_u(w)$ and $\Lambda_d(w)$, respectively.
For exactly two values $w \in \{w_{*}, w^{*} \}$, with $w_{*} < w^{*}$ there are heteroclinic connections from $\Lambda_d(w_{*})$ to $\Lambda_u(w_{*})$
and from $\Lambda_u(w^{*})$ to $\Lambda_d(w^{*})$ (a proof of that phenomenon is given in~\cite{Conley}).
It happens that in the range $[w_{*},w^{*}]$ the slow flow on the branch $\Lambda_u$ is monotonically decreasing
and on the branch $\Lambda_d$ monotonically increasing, so by connecting the heteroclinics with pieces of the slow manifold
one assembles the \emph{singular periodic orbit}, see Figure~\ref{orbitFig}.
The proof of existence of an actual periodic orbit goes by perturbing to $\epsilon>0$ small and using
certain arguments based on topological methods~\cite{Conley,Carpenter,Mischaikow} or Fenichel theory and differential forms~\cite{Kopell, JonesBook}.

\begin{figure}[t]
  \centering
  \begin{subfigure}{0.48\textwidth}
    \centering
    \begin{overpic}[width=0.8\textwidth, clip=true, trim = 70mm 110mm 70mm 110mm]{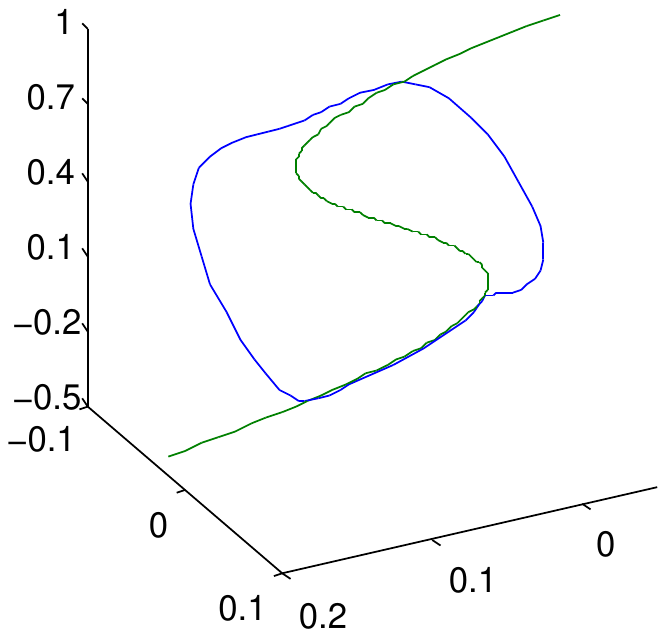}
      \put (14,57) {\scriptsize{$u$}}
      \put (17,12) {\scriptsize{$v$}}
      \put (66,4){\scriptsize{$w$}}
    \end{overpic}
    \caption{Approximate homoclinic orbit for $\epsilon=0.005$, $\theta \approx 1.2$ in blue.}
  \end{subfigure} 
  \begin{subfigure}{0.48\textwidth}
    \centering
    \begin{overpic}[width=0.8\textwidth, clip=true, trim = 50mm 97mm 60mm 100mm]{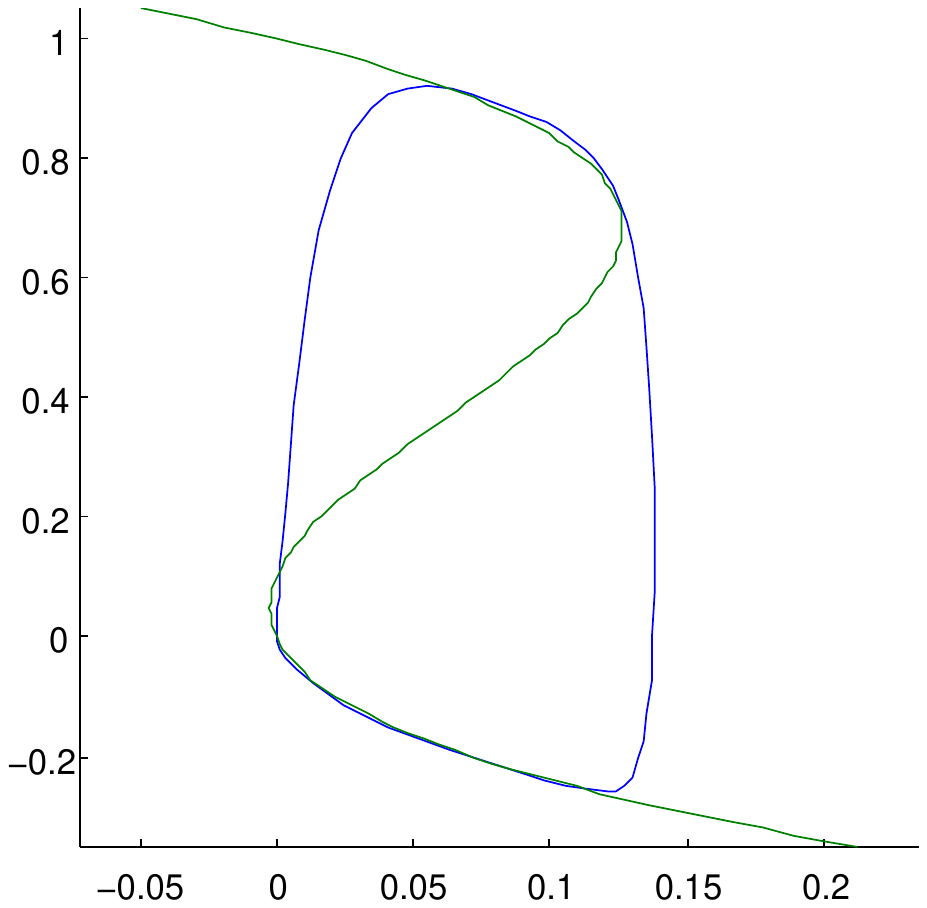}
      \put (23,52) {\scriptsize{$u$}}
      \put (58, 0) {\scriptsize{$w$}}
    \end{overpic}
    \caption{Projection onto $(u,w)$ plane.}
  \end{subfigure}
  \\
  \begin{subfigure}{0.45\textwidth}
    \centering
      \begin{tikzpicture}[line cap=round,line join=round,>=latex,x=4.5mm,y=4.5mm]
 
  \draw[color=black] (-0.,0.) -- (10,0.);
  \foreach \x in {0,...,10}
  \draw[shift={(\x,0)},color=black] (0pt,1pt) -- (0pt,-1pt) node[below] {};

  \draw[color=black] (0.,0.) -- (0.,10);
  \foreach \y in {0,...,10}
  \draw[shift={(0,\y)},color=black] (2pt,0pt) -- (-1pt,0pt) node[left] {};
 

  \clip(0,-1) rectangle (10,10);

  \node[draw,circle,inner sep=1pt,fill, color=blue] at (2,2){};

  \draw [color=darkgreen,semithick,domain=0.:10.] plot(\x,{8});
  \draw [color=darkgreen,semithick,domain=0.:10.] plot(\x,{2});

  \draw [->,semithick,dash pattern=on 5pt off 5pt,color=blue] (2.,2.4) -- (2.,8.);
  \draw [->,semithick,dash pattern=on 5pt off 5pt,color=blue] (8.,8.) -- (8.,2.);

  \draw [->,color=blue] (4.99,8.) -- (5.01,8.);
  \draw [->,color=blue] (5.01,2.) -- (4.99,2.);
  \draw [->,color=blue] (0.99,2.) -- (1.01,2.);
  \draw [->,color=blue] (3.01,2.) -- (2.99,2.);
  \draw [->,color=blue] (7.01,2.) -- (6.99,2.);

  \draw (9.,8.) node[anchor=north] {\scriptsize{$\Lambda_u(w)$}};
  \draw (9.,2.) node[anchor=north] {\scriptsize{$\Lambda_d(w)$}};

  \draw (2.,5.) node[anchor=west] {\scriptsize{$\theta=\theta_{*}$}};
  \draw (2.,4.3) node[anchor=west] {\scriptsize{$w_*=0$}};
  \draw (8.,5.) node[anchor=west] {\scriptsize{$w^{*}$}};

  \draw (5.,0.) node[anchor=north] {\scriptsize{$w$}};
  \draw (0.,5.) node[anchor=west] {\scriptsize{$u,v$}};

\end{tikzpicture}
      \caption{A schematic drawing of the singular homoclinic orbit.}
  \end{subfigure}
  \caption{A numerical approximation of the homoclinic orbit close to the singular orbit, the slow manifold in green.}
  \label{homFig}
\end{figure}

The homoclinic orbit is formed in a similar scenario.
As in the periodic case, it is constructed by perturbing a \emph{singular homoclinic orbit} at $\epsilon=0$, consisting of two heteroclinic connections
in the fast subsystems and two fragments of the slow manifold, on which the flow is monotone.
However, since the we are looking for a homoclinic to the point $(0,0,0)$, to construct the singular orbit we need to set $w_{*}=0$, and
another parameter needs to be varied to create a heteroclinic connection in the respective fast subsystem.
Conveniently, we have the wave speed $\theta$, which is set to some $\theta_{*}$ to create the desired connecting orbit
from $\Lambda_d(0)$ to $\Lambda_u(0)$. For such $\theta_{*}$ the slow variable $w$ can
be varied to find $w^{*}>0$ such that the second heteroclinic from $\Lambda_u(w^{*})$ to $\Lambda_d(w^{*})$ is formed.
These two are in turn connected by pieces of the slow manifold, on which the slow flow is monotone, to assemble the singular homoclinic, see Figure~\ref{homFig}.

Using methods described above proofs of existence have been given for $\epsilon \in (0,\epsilon_0]$, $\epsilon_0$ ``small enough''.
With aid of computer we are able to improve these results.
We give an \emph{explicit} $\epsilon_{0}$, such that for parameter range $\epsilon \in (0,\epsilon_{0}]$ there exists a periodic orbit
and a homoclinic orbit of~\eqref{FhnOde}.
Our secondary goal is to make $\epsilon_{0}$ as large as possible. 
In particular, for the periodic orbit we show that for $\epsilon \geq \epsilon_{0}$ one can perform
further continuation using well-established computer-aided methods such as the interval Newton-Moore method
applied to a sequence of Poincar\'e maps~\cite{Alefeld, Neumaier, Moore}.

Our main results concerning periodic solutions are:
\begin{thm}\label{thm:main1}
  For each $\epsilon \in (0,1.5 \times 10^{-4}]$, for $\theta = 0.61$ and other parameter values given in~\eqref{eq:parameters} there exists a periodic orbit of~\eqref{FhnOde}.
\end{thm}

\begin{thm}\label{thm:main2}
  For each $\epsilon \in [1.5 \times 10^{-4},0.0015]$, for $\theta = 0.61$ and other parameter values given in~\eqref{eq:parameters} there exists a periodic orbit of~\eqref{FhnOde}.
\end{thm}

\begin{thm}\label{thm:main3}
  For $\epsilon = 0.0015$, $\theta = 0.61$ and other parameter values given in~\eqref{eq:parameters} there exists a periodic orbit of~\eqref{FhnOde},
  which is formed from a locally unique fixed point of a Poincar\'e map.
\end{thm}

The reason we do not merge statements of Theorems~\ref{thm:main1} and~\ref{thm:main2} is a significant difference in proof techniques.
For the proof of Theorem~\ref{thm:main1} we exploit the fast-slow structure and construct a sequence of \emph{isolating segments} and \emph{covering relations}
around the singular orbit.
For Theorem~\ref{thm:main2} we perform a ``regular ODE'' type of proof,
with a parameter continuation method based on verifying covering relations around an approximation of the periodic orbit.
Finally, in Theorem~\ref{thm:main3} we are far enough from $\epsilon=0$, so that a proof by the interval Newton-Moore method applied to a sequence of Poincar\'e maps 
succeeds, establishing both the existence and local uniqueness.

The motivation for this choice of wave speed was that in numerical simulations parts of the periodic orbit
near the slow manifold stretched relatively long, which allowed us to fully exploit its hyperbolicity.
In the program files a lot of values were hardcoded for this particular $\theta$, but we report that
by substituting $\theta=0.53$, $\theta=0.47$, $\theta \in [0.55,0.554]$ we were also able
to produce results like Theorem~\ref{thm:main1}, for a (shorter) range of $\epsilon \in (0, 5 \times 10^{-5}]$.
We think that by spending time tuning the values in the proof, the range of $\epsilon$ for these $\theta$'s could have been made wider.
This is of course the very same orbit and if one had enough patience, then continuation in $\theta$ would be theoretically possible.

Our analogue of Theorem~\ref{thm:main1} for the homoclinic loop is as follows.

\begin{thm}\label{thm:main4}
  For each $\epsilon \in (0, 5 \times 10^{-5}]$ and for other parameter values given in~\eqref{eq:parameters} there exists a 
  wave speed $\theta=\theta(\epsilon) \in [1.2624, 1.2675]$  
  such that the system~\eqref{FhnOde} possesses a homoclinic orbit to the zero equilibrium.
\end{thm}
 
Further continuation of the homoclinic solution (analogues of Theorems~\ref{thm:main2}, \ref{thm:main3})
is still an open problem due to computational, rather than theoretical, difficulties.

Let us observe that proofs employing \emph{interval arithmetics} are fairly easy to adapt to compact parameter ranges. 
A \emph{validated continuation} can be performed by
subdividing the parameter interval finely and feeding the program with small parameter intervals instead of one exact value~\cite{Day,Lessard,Galias2,WilczakZgliczynski}.
However, we emphasise again that
when trying to prove Theorems~\ref{thm:main1}, \ref{thm:main4} we are dealing with a half-open range 
and such straightforward approach is bound to fail -- as $\epsilon \to 0^{+}$ 
the integration time along parts of these orbits near the slow manifold grows to infinity (as does the period of the periodic orbit)
and in the singular limit $\epsilon=0$ both of them are destroyed.
Therefore not all premises in theorems implying existence of such orbits can be verified for $\epsilon \in [0,\epsilon_0]$;
the assumptions need to be formulated in such a way, that the ones which are computationally difficult are possible to check with computer aid for $\epsilon \in [0,\epsilon_0]$,
and the leftover ones are in a simple form where $\epsilon$ can be factored out ``by hand'', assuming $\epsilon > 0$.

Full proofs are executed with computer assistance and described 
in detail in Section~\ref{sec:implementation}.
For the computer assisted assumption verification, 
in particular computation of enclosures of Poincar\'e maps and their derivatives, we use the previously mentioned CAPD library~\cite{CAPD}.
The code which executes the necessary computations is available at the author's homepage~\cite{Czechowski}.

Below we outline the basic ideas of each proof.

\subsection{Outline of the proof of Theorem~\ref{thm:main1}}

We conduct a phase space proof based on a reduction to a sequence of Poincar\'e sections and a fixed point argument for a sequence of Poincar\'e maps.
For $\epsilon > 0$ small the orbit switches between two regimes - the fast one close to heteroclinics of the fast subsystem
and the slow one along the branches $\Lambda_d, \Lambda_u$ of the slow manifold. The strategy is to form a closed sequence of
covering relations and isolating segments and deduce the existence of a fixed point of a sequence of Poincar\'e maps
via a topological theorem -- Theorem~\ref{thm:2}. 

For the fast regime we employ rigorous integration to compute Poincar\'e maps and check covering relations among h-sets
placed on Poincar\'e sections near the points $\Lambda_d(w_{*}),\Lambda_u(w_{*})$, $\Lambda_u(w^{*}),\Lambda_d(w^{*})$.
The h-sets come equipped with a coordinate system in which one direction is specified as exit and the other as entry.
To verify a covering relation by a Poincar\'e map between such two sets $X$, $Y$
one needs to check that the exit direction edges of $X$ are stretched over $Y$ in the exit direction
and that the image of $X$ is contained in the entry direction width of $Y$, see Figure~\ref{coveringFig} in Section~\ref{sec:covpoinc}.
The ``shooting'' in the exit direction is in fact made possible by a non-degenerate intersection of stable and unstable manifolds of the respective fixed
points in the singular limit $\epsilon=0$, as described in condition (P2) in Section~\ref{sec:covslowfast}.

Around the slow manifold branches we place isolating segments, which allow us to track the orbit in this region
by additional coverings, see Theorems~\ref{is:1}, \ref{is:2}.
Recall, that for each isolating segment one distinguishes three directions: exit, entry and central, and it is required that the faces
in the exit direction are immediate exit sets for the flow, the faces in the entry direction are immediate entrance sets
and the flow is monotone along the one-dimensional central direction.
The first two assumptions are checked by a computer (for $\epsilon \in [0,1.5 \times 10^{-4}]$), exploiting the hyperbolicity of branches of the slow manifold.
The last one we can easily fulfill by aligning the central direction of segments with the slow variable direction.
This setup reduces the central direction condition to a verification of whether $\frac{dw}{dt} \neq 0$ for all points in such segment. 
Under assumption $\epsilon \neq 0$ we can then factor out $\epsilon$ from the slow velocities,
and our condition reduces to a question whether $u \neq w$ for all points in each segment, which is straightforward to check.
This is the only moment in the proof when we need to assume that $\epsilon$ is strictly greater than $0$.

We place additional four ``corner segments'', containing the corner points $\Lambda_d(w_{*}),\Lambda_u(w_{*})$, $\Lambda_u(w^{*}),\Lambda_d(w^{*})$,
the role of which is to connect the h-sets with the segments around the slow manifold. 
From the viewpoint of definition these are no different than regular isolating segments.
However, the mechanism of topological tracking of the periodic orbit here is slightly distinct, as
the central direction changes roles with the exit/entry ones, see Theorems~\ref{is:3}, \ref{is:4}.

To obtain a closed loop, the sizes of the first and the last h-set in the sequence have to match.
For that purpose isolating segments around the slow manifold may need to
grow in the exit direction and compress in the entry one as we move along the orbit.
This way we can offset the size adjustments of the h-sets, which may be necessary to obtain covering relations in the fast regime.
The analysis of a model example performed in Theorem~\ref{thm:heuristic} is devoted to providing an argument for why this should work for $\epsilon$ small.
The main idea is that as $\epsilon \to 0^{+}$ the vector field in the slow/central direction decreases to 0 and
the slope of the segment becomes irrelevant when checking isolation, see Figure~\ref{slopeFig} in Section~\ref{sec:covslowfast}.

A schematic drawing representing the idea of the proof for the model example is given in Figure~\ref{schemeFig}, in Section~\ref{sec:covslowfast}.

\subsection{Outlines of the proofs of Theorems~\ref{thm:main2}, \ref{thm:main3}}

We are already at some distance from $\epsilon=0$, but for small $\epsilon$ the periodic orbit's normal bundle is consisting
of one strongly repelling and one strongly contracting direction, so any attempts of approximating the orbit by numerical integration, either forward or backward in time, fail.
On the other hand, the singular orbit at $\epsilon=0$ no longer serves as a good approximation for the purpose of a computer assisted proof.
As we can see, the challenge now is on the numerical, rather than the conceptual side.
To find our good numerical guess we introduce a large amount of sections, so that the integration times between each two of them
do not exceed some given bound and then apply Newton's method to a problem of the form
\begin{equation}\label{eq:problemform}
  \begin{aligned}
    P_{1}(x_{1}) - x_{2} &= 0, \\ P_{2}(x_{2}) - x_{3}&=0, \\ \dots  \\ P_{k}(x_{k}) - x_{1}&=0,
  \end{aligned}
\end{equation}
where $P_{i}$'s are the respective Poincar\'e maps.
Then, we construct a closed sequence of h-sets on these sections and verify covering relations between each two consecutive ones,
to prove the periodic orbit by means of Theorem~\ref{cov:sequence} (Corollary 7 in \cite{GideaZgliczynski}).
Since we control the integration times, isolating segments
are not needed anymore -- in Theorem~\ref{thm:main1} they were used for pieces of the orbit where the integration time tended to infinity.

By a rigorous continuation with parameter $\epsilon$, we are able to
get an increase of one order of magnitude for the upper bound of the range of $\epsilon$'s, for which the periodic orbit is confirmed.
Without much effort we show that for this value of $\epsilon$ the classical (see \cite{Galias, Zgliczynski, Alefeld, Neumaier} and references given there) method 
of application of the interval Newton-Moore operator to a problem of the form~\eqref{eq:problemform} succeeds.
This requires a rigorous $C^{1}$ computation, but these are handled efficiently by the $C^{1}$ Lohner algorithm implemented in CAPD~\cite{Zgliczynski}.

\subsection{Outline of the proof of Theorem~\ref{thm:main4}}

The main idea of the proof is to vary $\theta$ and create an
intersection of the two-dimensional stable manifold $W^s(0,0,0)$ and the one-dimensional unstable manifold $W^u(0,0,0)$ of the zero equilibrium.
For that we first need to obtain some ($\epsilon$-independent) bounds on these
two objects in vicinity of $(0,0,0)$. Let us observe that computing even these local bounds 
results in a structurally unstable problem, as for $\epsilon=0$ the stable manifold
degenerates to one dimension.

The steps below are performed for $\epsilon$ fixed as a range of the form $(0,\epsilon_0]$,
and same considerations as for the periodic orbit on when in computations this range can be enclosed by $[0,\epsilon_0]$
apply.

To produce the local bounds we use the method of isolating blocks
with cones, described in~\cite{ZgliczynskiMan}.
Details are provided in Section~\ref{sec:cc}.
In short, we have to estimate the directions of the vector field on the boundary of such block
and check positive definiteness of a matrix formed from a symmetrization of a product of a certain quadratic form 
(the \emph{cone field}) and an enclosure of the derivative of the vector field over the block, see Theorem~\ref{szczelir}
(Theorem 4 in~\cite{szczelir}).
For the first task, we can factor out $\epsilon$ wherever necessary in an analogous way 
as for the isolating segments, by aligning one of the directions of the block with the slow direction.
For verifying the positive definiteness of the matrix, parameter $\epsilon$ can be factored out
by a suitable choice of an $\epsilon$-dependent cone field~\eqref{eq:jeps},
again utilizing that $\epsilon>0$.

The above analysis is performed for a small range of $\theta$ containing $\theta_{*}$.
The next step is to use a topological theorem -- Theorem~\ref{thm:hom2} --
to connect $W^u(0,0,0)$ with $W^s(0,0,0)$.
Informally speaking, a part of the stable manifold is tracked backward in time 
through neighborhoods of $\Lambda_d(w^{*})$, $\Lambda_u(w^{*})$ and $\Lambda_u(w_{*})$
by a sequence of covering relations and isolating segments around slow manifolds, 
in the same manner as the periodic orbit, and for all $\theta$ in the preset range.
Once we reach an h-set on a suitable section in proximity of $\Lambda_u(w_{*})$
we propagate our upper bound on (a branch of) $W^u(0,0,0)$ to this section, by a Poincar\'e map
computed with $\theta$ varying in our range.
If the intersection with that h-set is topologically transverse with respect to $\theta$ 
(i.e. a certain covering relation has to be fulfilled),
then, by continuity, $W^u(0,0,0)$ and $W^s(0,0,0)$ have to intersect for some $\theta=\theta(\epsilon)$
from the preset range of $\theta$.

In Theorem~\ref{thm:heuristic2} in Section~\ref{sec:covslowfast2} we show that our construction
is bound to succeed on a model system sharing certain properties with FitzHugh-Nagumo system.
Figure~\ref{schemeFig2} contains a schematic drawing of the strategy of that proof.

\subsection{Related works}

In this subsection we compare our results with the ones reported in the~\cite{ArioliKoch, Matsue},
where a rigorous, computer assisted analysis of traveling waves in the FitzHugh-Nagumo model was also performed.

\subsubsection{Results for a single value of $\epsilon$ by Arioli \& Koch~\cite{ArioliKoch}}\label{subsec:ArioliKoch}

Authors perform a computer assisted proof of existence and stability of 
both the periodic wave train and the fast pulse in the FitzHugh-Nagumo system for $\epsilon$ set to $0.01$.
from the point of view of applications this is a more realistic parameter value than the ranges considered by us.
Contrary to our approach, the analysis does not exploit the fast-slow nature of the system and gives no insight
on how to design a computer assisted proof for $\epsilon$ arbitrarily small, positive
(which may be of interest in other systems coming from applications, where $\epsilon$ is much smaller and regular ODE tools suffer from stiffness,
see~\cite{KuehnBook}).

Existence of the periodic orbit in the traveling wave equation was proved by an application of a Newton-like operator to find solutions
of the system in a suitable space of smooth periodic maps. 
The homoclinic orbit was constructed by expanding the stable and unstable manifold of the zero equilibrium into power
series and propagating them by a rigorous integrator.
What appears to be the central result of the paper is a rigorous computer assisted proof of stability 
of both waves by the method of the Evans function~\cite{EvansIII, EvansIV}.
The Evans function has been previously used to determine stability for $\epsilon \in (0,\epsilon_0]$, $\epsilon_0$ small enough~\cite{Maginu, Jones}
and it would be interesting to see whether these proofs can also be adapted to such range with an explicit $\epsilon_0$.

\subsubsection{Results for explicit parameter ranges by Matsue~\cite{Matsue}}\label{subsec:matsue}

This recent preprint contains results very much alike ours.
Namely, the author proves the existence of invariant manifolds near the slow manifold and the existence 
of periodic, homoclinic and heteroclinic cycles in explicit ranges $\epsilon \in (0,\epsilon_0]$.
Similar tools to ours are used, sometimes under different names, e.g. isolating blocks of a certain form take place of what we call isolating segments,
and assumptions on chains of covering relations and isolating segments are referred to as the \emph{covering-exchange property}.

It seems that Matsue follows the methods of GSPT (see~\cite{JonesBook}) more closely than us
and also makes a more extensive use of cones. Even though no $C^1$ claims, such as on uniqueness or stability of the orbits (that is, as solutions
of the traveling wave ODE), are made, such analysis is certainly very helpful to solving these problems in future.
In addition, the author proves the existence of a heteroclinic cycle and the existence of invariant manifolds near the slow (critical) manifold.
These objects were not considered by us in this thesis, but we remark that
the method of proof for a heteroclinic cycle is not much different than for a homoclinic orbit,
and such result for sufficiently small $\epsilon$ was previously known, see~\cite{Deng};
and the validation of invariant manifolds near slow manifolds is a regular perturbation problem, which was previously discussed in e.g.~\cite{GuckenheimerJohnson},
and as such, was not pursued by us.

On the other hand, the author restricts himself in many places to the fast-slow setting, whereas
statements of topological theorems presented in this thesis are general and applicaple to other types of stiff ODEs.
Moreover, our statements of Theorems~\ref{thm:main1} and~\ref{thm:main4} are valid for ranges of $\epsilon$ wider by over one order of magnitude than the ones in 
corresponding theorems of Matsue.
In addition, in Theorem~\ref{thm:main2} we extend the range of existence of periodic orbit by an additional order of magnitude;
such continuation is not performed in~\cite{Matsue} and it is not made clear whether further continuation would have been computationally possible
from such small upper bounds on $\epsilon$.

Our first preprint~\cite{CzechowskiZgliczynski}
appeared online on ar{X}iv in February 2015, approximately five months before the one of Matsue (July 2015),
and to the best of our knowledge is the first computer assisted result for a fast-slow system of such type.
Both papers are yet to be published in a journal. We remark that in~\cite{CzechowskiZgliczynski}
we stated results for the periodic orbit only, so Matsue was first with the proof for the homoclinic orbit (and the heteroclinic orbit,
which we did not consider).
Our proof for the homoclinic orbit,
presented for the first time in this thesis was obtained independently; in fact for a long time we were unaware that someone else is working on 
the same topic.

\section{Organization of the thesis}

The contents of this thesis are arranged as follows. 

Chapter~\ref{chap:intro} is introductory, we motivate our research
and announce the results for the FitzHugh-Nagumo system. 
Even though we refer to these as our main theorems,
they are rather a demonstration of feasibility and significance of several abstract theorems and concepts,
stated and proved in subsequent chapters.

In Chapter~\ref{chap:covseg}
we provide a self-contained theory on how to incorporate isolating segments into the method of covering relations.
Contents of this chapter are independent of the fast-slow structure of the FitzHugh-Nagumo system
and can be applied to other types of ODEs. 
Section~\ref{sec:covpoinc} contains mostly prerequisites on h-sets, covering relations, cone conditions and isolating blocks.
Novel contributions start from Section~\ref{sec:segments}, where we introduce
the definition of an isolating segment for an autonomous system and show existence of certain 
covering relations among its faces (Theorems~\ref{is:1}, \ref{is:2}, \ref{is:3}, \ref{is:4}).
Based on these theorems, in Section~\ref{sec:applications} we 
prove theorems on how chains of covering relations and isolating segments 
can be used to track orbits of ODEs and imply 
the existence of periodic and connecting orbits (Theorems~\ref{thm:1}, \ref{thm:2},
\ref{thm:conmap}, \ref{thm:hom1}, \ref{thm:hom2}).
Theorems~\ref{thm:1},~\ref{thm:conmap} and~\ref{thm:hom1} are valid in arbitrary dimensions,
and stated mostly for future reference. Theorems~\ref{thm:2} and~\ref{thm:hom2}, 
restricted to 3-dimensional systems, additionally employ a certain switch between directions in segments (see Theorems~\ref{is:3}, \ref{is:4}),
which can be viewed as a topological version of the Exchange Lemma (cf. Chapter 5 in~\cite{JonesBook}) from GSPT.
These two theorems are in turn applied to prove Theorems~\ref{thm:main1} and~\ref{thm:main4}.

Chapter~\ref{chap:waves} contains applications of the previously introduced
theory to prove the existence of traveling waves (i.e. suitable periodic and homoclinic orbits) in the FitzHugh-Nagumo model.
To use theorems for connecting orbits we first need
some local estimates on the stable and the unstable manifold of the equilibrium at origin.
Section~\ref{sec:cc} deals with the question of how to organize the computations
to obtain these bounds from a suitable isolating block with cones.
Although this section does not contain qualitatively new abstract theorems, 
we stress its importance, as it allows us to 
obtain $\epsilon$-independent bounds in fast-slow systems 
for a range $\epsilon \in (0,\epsilon_0]$, with $\epsilon_0$ explicit,
and with tedious estimates passed to the computer.

Section~\ref{sec:model} contains two simplified, model fast-slow systems
sharing some qualitative properties with the FitzHugh-Nagumo system. 
We perform a pen-and-paper construction of suitable chains of covering relations and isolating segments
for $\epsilon \in (0,\epsilon_0]$, $\epsilon_0$ ``small enough''.
The purpose of this section is to give reader an insight on why our topological methods are bound to work
in certain singular perturbation scenarios, and its contents are not necessary
to prove any of the results outside of it.
For that reason, and to make this part reasonably short, the exposition is sometimes not very strict.
In some sense, in this section we repeat the methods and the results from~\cite{Conley,Carpenter,Hastings}, 
recast in a language of covering relations and isolating segments.

In Section~\ref{sec:implementation} we give details of the computer assisted proofs of Theorems~\ref{thm:main1}, \ref{thm:main2}, \ref{thm:main3}
and~\ref{thm:main4} and provide some numerical data from the programs.
Let us remark that description of the proof of Theorem~\ref{thm:main2} is quite concise as it did not contain any novel mathematical theory.
However, design and implementation of good heuristic algorithms for rigorous continuation of a periodic orbit in a stiff problem was quite a challenge,
and we consider it a certain achievement from the numerical point of view.

In Chapter~\ref{chap:conclusions} we formulate several concluding remarks
and outline possible future directions for research, based on results from this thesis.

\section{Notation} 

Most of the notation is introduced through definitions or at the beginning of the respective chapters, sections, subsections and theorems.
In the second case symbols are defined ``locally'', and outside the scope of a given part of the text they may be reused 
for other purposes.
Notation introduced here applies also to preceding sections.

By $\nn$, $\zz$ and $\rr$ we will denote natural, integer and real numbers, respectively.
We will also write $\zz^{*}$ for the set of non-zero integers and $\rr^+$ for positive reals. 
By $[a,b]$ we will denote
the closed interval $\{ x \in \rr: a \leq x \leq b \}$.

Unless otherwise stated $\vectornorm{\cdot}$ can be any fixed norm in $\rr^{n}$.
Sometimes we will restrict ourselves to the $\max$ norm, that is the norm given by
\begin{equation}
  \begin{aligned}
  \vectornorm{x}_{\max} &= \max\{x_1, \dots, x_n \}, \\ 
  x&=(x_1, \dots, x_n).
\end{aligned}
\end{equation}

Given a norm, by $B_{n}(c,r)$ we will denote the ball of radius $r$ centered at $c \in \rr^{n}$.
By $\langle \cdot , \cdot \rangle$ we will denote the standard dot product in $\rr^{n}$.

We assume that $\rr$ is always equipped with the following norm: $\vectornorm{x} = |x|$.

Given a set $Z$, by $\inter Z$, $\overline{Z}$, $\partial Z$ and $\conv{Z}$
we will denote the interior, closure, boundary and the convex hull of $Z$, respectively.

Given a topological space $X$, a subspace $D \subset \rr \times X$ and a local flow $\varphi : D \to X$,
by writing $\varphi(t,x)$ we will implicitly state that $(t,x) \in D$,
so for example by
\begin{equation}
 \varphi(t,x) = y
\end{equation}
we will mean $\varphi(t,x)$ exists and $\varphi(t,x) = y$.

By $\id_{X}$ we denote the identity map on $X$.

The symbol $\const$ denotes a constant -- usually some uniform bound -- the value of which being not important to us,
so for example the expression $f(x) > \const, \ x \in X$, means that there exists $C \in \rr$ such that
$f(x)>C \ \forall x \in X$.

The $i$-th partial derivative of a differentiable map $f(x)=f(x_1, \dots, x_n)$, $f: \rr^n \to \rr^m$ will be denoted by $\frac{\partial f}{\partial x_i}$
or $\frac{d f}{d x_i}$.
By $D_{x_0} f$, $Df(x_0)$ and $\frac{df}{dx} (x_0)$ we will denote the matrix of partial derivatives of $f$ for $x:=x_0$.
If $n=1$, by $\dot{f}$ we will denote the vector of derivatives of $f$ with respect to its only argument.
If $m=1$, $\nabla f (x_0)$ will denote the gradient of $f$ at $x_0$.
In case $n=m=1$ we will sometimes write $f'$ instead of $\dot{f}$.

We use the big $O$ notation exclusively to describe the limiting behavior near $0$
and only in a sublinear context -- so the meaning of $f(x)=O(x)$ is $|f(x)| \leq C |x|$ for some $C>0$.

By smoothness we mean $C^{1}$ smoothness.
In some assumptions differentiability would be enough,
but we do not go into such details.

\chapter{Topological tools}\label{chap:covseg}

\section{H-sets, covering relations and cone conditions}\label{sec:covpoinc}

In this section we recall the definitions of h-sets, covering relations and \emph{backcovering relations} 
for maps as presented in~\cite{GideaZgliczynski}.
We make a following change in the nomenclature: in~\cite{GideaZgliczynski} various objects related to h-sets (directions, subsets, etc.) 
are being referred to as \emph{unstable} or \emph{stable}. We will refer to them as \emph{exit} and \emph{entry}/\emph{entrance}, respectively; 
we think that this reflects better their dynamical nature and does not lead to misunderstandings.
However, we keep the original symbols $u,s$, so $u$ should be connoted with exit and $s$ with entry.

\begin{defn}[Definition 1 in~\cite{GideaZgliczynski}]\label{h-set}
  An h-set is formed by a quadruple 
  \begin{equation}
    X=(|X|,u(X),s(X),c_{X})
  \end{equation}
  consisting of
  a compact set $|X| \subset \rr^{n}$ - \emph{the support}, a pair of numbers $u(X), s(X) \in \mathbb{N}$
  such that $u(X)+s(X) = n$ (the number of exit and entry directions, respectively)
  and a coordinate change homeomorphism $c_{X}: \rr^{n} \rightarrow \rr^{u(X)} \times \rr^{s(X)}$
  such that
  \begin{equation}\label{h-set:1}
    c_X(|X|)=\overline{B_{u(X)}(0,1)} \times \overline{B_{s(X)}(0,1)}.
  \end{equation}

  We set:
  \begin{equation}
    \begin{aligned}
      X_{c}&:=\overline{B_{u(X)}(0,1)} \times \overline{B_{s(X)}(0,1)}, \\
      X_{c}^{-}&:=\partial B_{u(X)}(0,1) \times \overline{B_{s(X)}(0,1)}, \\
      X_{c}^{+}&:=\overline{B_{u(X)}(0,1)} \times \partial B_{s(X)}(0,1), \\
      X^{-} &:= c_{X}^{-1}(X_{c}^{-}), \\
      X^{+} &:= c_{X}^{-1}(X_{c}^{+}).
    \end{aligned}
  \end{equation}

  We will refer to $X^{-}$/$X^{+}$ as the exit/entrance sets, respectively.
  To shorten the notation we will sometimes drop the bars in the symbol $|X|$ and just write $X$ to denote both the
  h-set and its support.
\end{defn}

\begin{rem}
  Due to condition~\eqref{h-set:1}, it is enough to specify $u(X), s(X)$ and $c_{X}$ to unambiguously define an h-set $X$.
\end{rem}

Let us recall the standard axiomatic definition of the Brouwer degree. We follow the exposition given in~\cite{Schwartz}.

\begin{thmdefn}[Chapter III in~\cite{Schwartz}]
  \normalfont
  Let $\Omega \subset \rr^n$ be bounded and open and let $g: \overline{\Omega} \to \rr^n$ be continuous.
  Assume that $p \notin g(\partial \Omega )$. Then there is a unique integer $\deg(p,g,\Omega)$
  with the properties:
  \begin{itemize}
    \item[(A1)] \emph{Invariance under homotopy.} If $h(\xi ,x): [0,1] \times \overline{\Omega} \to \rr^n$
      is continuous and $p  \notin h([0,1],\partial \Omega )$, then
      \begin{equation}
        \deg(p,h(0,\cdot),\Omega) = \deg(p, h(\xi,\cdot), \Omega), \quad \forall \xi \in (0,1].
      \end{equation}
    \item[(A2)] \emph{Dependence on the boundary values.} Let $\tilde{g} : \Omega \to \rr^n$ be continuous.
      If $g_{|_{\partial \Omega}} = \tilde{g}_{|_{\partial \Omega}}$, then
      \begin{equation}
        \deg(p, g , \Omega) = \deg(p, \tilde{g} , \Omega).
      \end{equation}
    \item[(A3)] \emph{Continuity.} There exists
      a neighbourhood $U$ of $g$ in the space of continuous maps from $\Omega$ to $\rr^n$ (with the $\sup$ norm topology), 
      such that if $\tilde{g} \in U$, then $p \notin \tilde{g}(\partial \Omega)$ and
      \begin{equation}
        \deg(p,g,\Omega) = \deg(p,\tilde{g},\Omega).
      \end{equation}
    \item[(A4)] \emph{Degree is locally constant.} If $p \notin g(\overline{\Omega})$, then $\deg(p,g,\Omega)=0$.
      If $p$ and $q$ belong to the same connected component of $\rr^{n} \setminus g(\overline{\Omega})$,
      then
      \begin{equation}
        \deg(p,g,\Omega) = \deg(q,g,\Omega).
      \end{equation}
    \item[(A5)] \emph{Decomposition of the domain.} Let $\Omega = \bigcup_{i \in I} \Omega_i$, where the family $\{\Omega_i\}_{i \in I}$
      consists of disjoint open sets, and $\partial \Omega_i \subset \partial \Omega \ \forall i$.
      It holds that
      \begin{equation}
        \deg(p, g, \Omega)  = \sum_{i \in I} \deg(p, g, \Omega_i)
      \end{equation}
    \item[(A6)] \emph{The excision property.} If $\tilde{\Omega}$ is an open subset of $\Omega$, such that $g^{-1}(p) \cap \Omega \subset \tilde{\Omega}$,
      then 
      \begin{equation}
        \deg(p,g,\Omega) = \deg(p,g,\tilde{\Omega}).
      \end{equation}
    \item[(A7)] \emph{The product property.} Let $\Omega = \Omega_1 \times \Omega_2 \subset \rr^{n_1} \times \rr^{n_2}$, $g = (g_1,g_2)$,
      $g_i : \Omega_i \to \rr^{n_i}$, $i \in \{1,2\}$ and $p=(p_1,p_2) \in \rr^{n_1} \times \rr^{n_2}$. Then
      \begin{equation}
        \deg(p,g,\Omega) = \deg(p_1,g_1, \Omega_1) \deg(p_2, g_2, \Omega_2),
      \end{equation}
      whenever the right-hand side is defined.
    \item[(A8)]  \emph{Degree for smooth maps.} Assume $g$ is smooth (in the sense, that it has a smooth extension to some neighborhood of $\overline{\Omega}$, and 
      for all $x \in g^{-1}(p)$ the derivative $D_x f$ is nonsingular. Then 
      \begin{equation}
        \deg(p,g,\Omega) = \sum_{x \in g^{-1}(p)} \sgn \det D_x g.
      \end{equation}
    \end{itemize}
\end{thmdefn}

We remark, that only a subset of the axioms given above is needed to uniquely define the degree.
We are now ready to state the definition of the covering relation between h-sets.

\begin{defn}[Definitions 2, 6 in~\cite{GideaZgliczynski}]\label{covering}
  Assume that $X,Y \subset \rr^{n}$ are h-sets, such that $u(X)=u(Y)=u$ and $s(X)=s(Y)=s$. Let $g:\Omega \to \rr^{n}$ be a
  map with $|X| \subset \Omega \subset \rr^n$. Let $g_{c}= c_{Y} \circ g \circ c_{X}^{-1}: X_{c} \to \rr^{u} \times \rr^{s}$
  and let $w$ be a non-zero integer.
  We say that $X$ $g$-covers $Y$ with degree $w$ and write
  \begin{equation}
     X \cover{g,w} Y
  \end{equation}
  iff $g$ is continuous and the following conditions hold
  \begin{itemize}
    \item[1.] there exists a continuous homotopy $h: [0,1] \times X_{c} \to \rr^{u} \times \rr^{s}$, such that
     \begin{align}
       h_{0}&=g_{c},\label{cover:1a} \\
       h([0,1],X_{c}^{-}) \cap Y_{c} &= \emptyset ,\label{cover:1b}  \\
       h([0,1],X_{c}) \cap Y_{c}^{+} &= \emptyset .\label{cover:1c}
    \end{align}
  \item[2.] There exists a continuous map $A:\rr^{u} \to \rr^{u}$, such that
    \begin{equation}\label{cover:2}
      \begin{aligned}
        h_{1}(p,q)&=(A(p),0) \quad \forall p \in \overline{B_{u}(0,1)} \text{ and } q \in \overline{B_{s}(0,1)}, \\
        A(\partial B_{u}(0,1)) &\subset  \rr^{u} \setminus \overline{B_{u}(0,1)}, \\
        \deg \left( 0,A_{|_{\overline{B_u(0,1)}}}, B_u(0,1) \right) &= w.
      \end{aligned}
   \end{equation}
  \end{itemize}

  In case $A$ is a linear map, from (A8) we get $\deg \left( 0,A_{|_{\overline{B_u(0,1)}}}, B_u(0,1) \right) = \sgn \det A = \pm 1$.
   In such situation we will often say that $X$ $g$-covers $Y$, omit the degree and write $X \cover{g} Y$.
\end{defn}

\begin{rem}[Remark 3 in~\cite{GideaZgliczynski}]
  For $u=0$ we have $B_{u}(0,1) = \emptyset$ and $X$ $g$-covers $Y$ iff $g(|X|)$ is a subset of $\inter |Y|$. 
  In that case, we formally set the degree $w$ to $1$.
\end{rem}

\begin{defn}[Definition 3 in~\cite{GideaZgliczynski}]
  Let $X$ be an h-set. We define the \textit{transposed h-set} $X^{T}$ as follows:
  \begin{itemize}
    \item $|X| = |X^{T}|$,
    \item $u(X^{T}) = s(X)$ and $s(X^{T}) = u(X)$,
    \item $c_{X^{T}}(x) = j(c_{X}(x))$, where $j: \rr^{u(X)} \times \rr^{s(X)} \to \rr^{s(X)} \times \rr^{u(X)}$ is given by $j(p,q)=(q,p)$.
  \end{itemize}
\end{defn}

Observe that $(X^{T})^{+} = X^{-}$ and $(X^{T})^{-} = X^{+}$, thus transposition
changes the roles of exit and entry directions.

\begin{defn}[Definition 4, 7 in~\cite{GideaZgliczynski}]
  Let $X,Y$ be h-sets with $u(X) = u(Y)$ and $s(X) = s(Y)$. Let $g: \Omega \subset \rr^{n} \to \rr^{n}$.
  We say that $X$ $g$-backcovers $Y$ with degree $w$
  and write $X \backcover{g,w} Y$ iff 
  \begin{itemize}
    \item $g^{-1}: |Y| \to \rr^{n}$ exists and is continuous, 
    \item $Y^{T}$ $g^{-1}$-covers $X^{T}$ with degree $w$.
   \end{itemize}
\end{defn}

\begin{defn}[Definition 5 in~\cite{GideaZgliczynski}]
  We will use the notation $X \gencover{g,w} Y$ and say that $X$ generically $g$-covers $Y$ with degree $w$ iff any of these two hold:
  \begin{itemize}
    \item $X$ $g$-covers $Y$ with degree $w$,
    \item $X$ $g$-backcovers $Y$ with degree $w$.
  \end{itemize}
\end{defn}

Again, we will sometimes omit the degree in our notation, in case the homotopy can be given to a linear map.

In the next chapters we will sometimes work with parameter-dependent maps.
Let $g: \Omega \times Z \to \rr^n$, where $\Omega \subset \rr^n$ and $Z \subset \rr^p$ represents a set of parameters. 
The meaning of the expression $X \longcover{g(\cdot,z),w} Y \ \forall z \in Z$ is
clear, however the meaning of $X \longgencover{g(\cdot,z),w} Y \ \forall z \in Z$ can be ambiguous.
A strict interpretation would imply that the parameter domain $Z$ be split so 
the map $g$ produce a covering relation only for some $z \in Z$ and a backcovering relation for others.
We would like to have an either-or relation, therefore
to shorten the formulation of several theorems we (re)define the parameter-dependent generic covering as follows.

\begin{defn}
  We will use the notation $X \longgencover{g(\cdot,z),w} Y \ \forall z \in Z$ 
  and say that $X$ generically $g$-covers $Y$ with degree $w$ for all $z \in Z$, iff any of these two hold:
  \begin{itemize}
    \item $g$ is continuous on $|X| \times Z$ and $X$ $g(\cdot,z)$-covers $Y$ for all $z \in Z$ with degree $w$,
    \item for any fixed $z \in Z$ 
      $g^{-1}(y, z)$ is defined for all $y \in |Y|$ and is continuous as a map on $|Y| \times Z$. 
      Moreover $X$ $g(\cdot,z)$-backcovers $Y$ for all $z \in Z$ with degree $w$.
  \end{itemize}
\end{defn}

\subsection{Verification of covering relations in low dimensions}

For an h-set $X$ with $u(X)=1, \ s(X)=s$ we have:
\begin{equation}
  \begin{aligned}
    X_{c} &= [-1,1] \times \overline{B_{s}(0,1)}, \\
    X_{c}^{-} &= (\{-1\} \times \overline{B_{s}(0,1)}) \cup (\{1\} \times \overline{B_{s}(0,1)}).
  \end{aligned}
\end{equation}

We will often use the following geometrical criterion for verifying of the covering relation in such case:

\begin{lem}[Theorem 16 in \cite{GideaZgliczynski}]\label{covlemma}
  Let $X,Y$ be h-sets with $u(X)=u(Y)=1$, $s(X)=s(Y)=s$. Let $g:X \to \rr^{s+1}$ be a continuous map.
  Assume that both of the following conditions hold:
  \begin{itemize}
    \item[(C1)] We have
      \begin{equation}
        g_{c}(X_{c}) \subset \inter( ((-\infty, -1) \times \rr^{s}) \cup Y_{c} \cup ((1, \infty) \times \rr^{s}) ),
      \end{equation}
    \item[(C2)] either
      \begin{equation}
        \begin{aligned}
          g_{c}(\{-1\} \times \overline{B_{s}(0,1)}) \subset (-\infty, -1) \times \rr^{s} &\text{ and } g_{c}(\{1\} \times \overline{B_{s}(0,1)}) \subset (1, \infty) \times \rr^{s} \\
          &\text{or} \\
          g_{c}(\{-1\} \times \overline{B_{s}(0,1)}) \subset (1,\infty) \times \rr^{s} &\text{ and } g_{c}(\{1\} \times \overline{B_{s}(0,1)} ) \subset (-\infty, -1) \times \rr^{s}.
        \end{aligned}
      \end{equation}
  \end{itemize}

  Then
  \begin{equation}
   X \cover{g} Y.
  \end{equation}
\end{lem}

\begin{rem}\label{rem:covlemma}
In applications it is convenient to introduce the notation $X^{-,l} = c_{X}^{-1}(\{-1\} \times \overline{B_{s}(0,1)})$ (\emph{the left exit edge})
and $X^{-,r} = c_{X}^{-1}(\{1\} \times \overline{B_{s}(0,1)})$ (\emph{the right exit edge}) and check (C1), (C2)
by putting

\begin{equation}
  \begin{aligned}
    g_{c}(X_{c}) &= (c_{Y} \circ g)(|X|),\\
    g_{c}(\{-1\} \times \overline{B_{s}(0,1)}) &= (c_{Y} \circ g) ( X^{-,l} ),\\
    g_{c}(\{1\} \times \overline{B_{s}(0,1)}) &= (c_{Y} \circ g) ( X^{-,r} ),
  \end{aligned}
\end{equation}
see Figure~\ref{coveringFig}.
\end{rem}

\begin{figure}
    \centering
      \begin{tikzpicture}[line cap=round,line join=round,>=latex,x=2.5mm,y=2.5mm]
 
  \clip(-19,-8.4) rectangle (19,7);

  \fill [color=brown!50] (-5.,-5.) -- (-5.,5.) -- (5.,5.) -- (5.,-5.);
  \draw [color=orange,thick] (-5.,5.) -- (-5.,-5.);
  \draw [color=orange,thick] (5.,5.) -- (5.,-5.);
  \draw [color=brown,thick] (-5.,5.) -- (5.,5.);
  \draw [color=brown,thick] (5.,-5.) -- (-5.,-5.);

  \filldraw[draw=blue, fill=blue!50, opacity=0.8]
      (-9,-4) .. controls (-2,1) and (3,-4) .. (11,-1) -- (10,3) .. controls (3,0) and (-2,5) .. (-10,0) -- (-9,-4);

  \draw [color=red,thick] (11,-1) -- (10,3);
  \draw [color=red,thick] (-10,0) -- (-9,-4);

  \draw (3.,-5.) node[anchor=south] {\scriptsize{$\displaystyle Y_{c}$}};
  \draw (0.,1.) node[anchor=north] {\scriptsize{$\displaystyle(c_{Y} \circ g)\left(|X|\right)$}};
  \draw (10.5,1.) node[anchor=west] {\scriptsize{$\displaystyle(c_{Y} \circ g)(X^{-,r})$}};
  \draw (-9.5,-2.) node[anchor=east] {\scriptsize{$\displaystyle(c_{Y} \circ g)(X^{-,l})$}};

  \draw[color=black,->] (-7.,-7.) -- (-5,-7.);
  \draw[color=black,->] (-7.,-7.) -- (-7,-5.);
  \draw (-7.,-6) node[anchor=east] {\scriptsize{$x_{s}$}};
  \draw (-6.,-7) node[anchor=north] {\scriptsize{$x_{u}$}};

\end{tikzpicture}
      \caption{A covering relation $X \cover{g} Y$.}
      \label{coveringFig}
\end{figure}

Analogously, if for an h-set $Y$ we have $u(Y)=u$ and $s(Y)=1$, then
\begin{equation}
  \begin{aligned}
    Y_{c} &= \overline{B_{u}(0,1)} \times [-1,1], \\
    Y_{c}^{+} &= (\overline{B_{u}(0,1)} \times \{-1\}) \cup (\overline{B_{u}(0,1)} \times \{1\}),
  \end{aligned}
\end{equation}
and we can apply the same principle to transposed sets:
\begin{lem}\label{backcovlemma}
  Let $X,Y$ be h-sets with $u(X)=u(Y)=u$ and $s(X)=s(Y)=1$. Let $\Omega \subset \rr^{u+1}$ and $g: \Omega \to \rr^{u+1}$ be continuous.
  Assume, that $g^{-1}: |Y| \to \rr^{u+1}$ exists, is continuous and that both of the following conditions hold:
  \begin{itemize}
    \item[(C1a)] We have
      \begin{equation}
        g^{-1}_{c}(Y_{c}) \subset \inter( ( \rr^{u} \times (-\infty, -1) ) \cup X_{c} \cup (\rr^{u} \times (1, \infty)) );
      \end{equation}
    \item[(C2a)] either
      \begin{equation}
        \begin{aligned}
          g^{-1}_{c}( \overline{B_{u}(0,1)} \times \{-1\} ) \subset \rr^{u} \times (-\infty, -1) &\text{ and }g^{-1}_{c}(\overline{B_{u}(0,1)} \times \{1\}) \subset \rr^{u} \times (1,\infty) \\
          &\text{or} \\
          g^{-1}_{c}( \overline{B_{u}(0,1)} \times \{-1\}) \subset \rr^{u} \times (1,\infty) \text{ and } &g^{-1}_{c}(\overline{B_{u}(0,1)} \times \{1\}) \subset \rr^{u} \times(-\infty, -1) .
        \end{aligned}
      \end{equation}
  \end{itemize}
  Then
  \begin{equation}
   X \backcover{g} Y.
  \end{equation}
\end{lem}

In such case we will sometimes operate with the notation $Y^{+,l} = c_{Y}^{-1}( \overline{B_{u}(0,1)} \times \{-1\} )$ (\emph{the left entrance edge})
and $Y^{+,r} =  c_{Y}^{-1}( \overline{B_{u}(0,1)} \times \{1\} )$ (\emph{the right entrance edge}).
Conditions (C1a) and (C2a) can then be rephrased in the same manner as in Remark~\ref{rem:covlemma}.

\subsection{Horizontal and vertical disks}\label{sec:conebasic}

Below we recall the definitions of \emph{cone conditions} and \emph{horizontal} and \emph{vertical disks},
as introduced in~\cite{ZgliczynskiMan}. These tools will be later used to represent and 
control unstable and stable manifolds (to be introduced in Definition~\ref{defn:man}) of a stationary point of a flow.
Our exposition with regard to cone conditions is very brief.
We only give the definitions of cone conditions for horizontal and vertical disks,
and our only purpose is to enjoy the convenient parametrization of these disks as graphs of Lipschitz maps, given by Theorem~\ref{thm:lipschitz}.
There is a quite developed theory of cone conditions for maps, with applications
to proving existence and uniqueness of invariant manifolds, which we omit in our presentation (cf.~\cite{ZgliczynskiMan, Capinski2}).
For instance, the proof of theorem about the existence of the stable and the unstable manifold within an isolating block with cones, cited by us in the next subsection 
(Theorem~\ref{szczelir}) follows from an application of cone conditions to time step maps of a flow.

\begin{defn}[Definition 5 in~\cite{ZgliczynskiMan}]
  Let $X$ be an h-set with $u(X) > 0$. Let $b: \overline{B_{u(X)} (0,1)} \to |X|$ be continuous and set $b_c := c_X \circ b$.
  We say that $b$ is a horizontal disk in $X$ if there exists a homotopy $h: [0,1] \times \overline{B_{u(X)} (0,1)} \to X_c$, such that
  \begin{align}
    h_0 &= b_c,\\
    h_1(x) &= (x,0) \quad \forall x \in \overline{B_{u(X)}(0,1)},\\
    h\left([0,1],x\right) &\subset X_c^{-} \quad \forall  x \in \partial B_{u(X)}(0,1).
  \end{align}
\end{defn}

\begin{defn}[Definition 6 in~\cite{ZgliczynskiMan}]
  Let $X$ be an h-set with $s(X) > 0$. Let $b: \overline{B_{s(X)} (0,1)} \to |X|$ be continuous and set $b_c := c_X \circ b$.
  We say that $b$ is a vertical disk in $X$ if there exists a homotopy $h: [0,1] \times \overline{B_{s(X)} (0,1)} \to X_c$, such that
  \begin{align}
    h_0 &= b_c,\\
    h_1(y) &= (0,y) \quad \forall y \in \overline{B_{s(X)}(0,1)},\\
    h\left([0,1],y\right) &\subset X_c^{+} \quad \forall  y \in \partial B_{s(X)}(0,1).
  \end{align}
\end{defn}

For an h-set $X$, we will sometimes say that the subset of its support $Y \subset |X|$ is a horizontal/vertical disk.
By that we mean, that there exists a horizontal/vertical disk $b$ in $X$, such that $Y$ is the image of $b$.

\begin{defn}[Definition 8 in~\cite{ZgliczynskiMan}]
 Let $X \subset \rr^n$ be an h-set and let $Q$ be a quadratic form given by
 \begin{equation}
   Q(x,y) = \alpha (x_u) - \beta (x_s), \quad (x_u,x_s) \in \rr^{u(X)} \times \rr^{s(X)},
 \end{equation}
 where $\alpha: \rr^{u(X)} \to \rr$ and $\beta: \rr^{s(X)} \to \rr$ are positive definite quadratic forms.
 The pair $(X,Q)$ will be called an \emph{h-set with cones}.
\end{defn}

\begin{defn}[Definition 9 in~\cite{ZgliczynskiMan}]
 Let $(X,Q)$ be an h-set with cones and let $b$ be a horizontal disk in $X$.
 We will say that $b$ satisfies the~\emph{cone condition} (with respect to $Q$) iff for
 any $x_{u,1}, x_{u,2} \in \overline{B_{u(X)}(0,1)}$, $x_{u,1} \neq x_{u,2}$ we have
 \begin{equation}
   Q(b_c(x_{u,1}) - b_c(x_{u,2})) > 0.
 \end{equation}
\end{defn}

\begin{defn}[Definition 10 in~\cite{ZgliczynskiMan}]
 Let $(X,Q)$ be an h-set with cones and let $b$ be a vertical disk in $X$.
 We will say that $b$ satisfies the~\emph{cone condition} (with respect to $Q$) iff for
 any $x_{s,1}, x_{s,2} \in \overline{B_{s(X)}(0,1)}$, $x_{s,1} \neq x_{s,2}$ we have
 \begin{equation}
   Q(b_c(x_{s,1}) - b_c(x_{s,2})) < 0.
 \end{equation}
\end{defn}

The geometrical intuition behind the notion of a horizontal and a vertical disk satisfying the cone condition is portrayed in Figure~\ref{fig:disk}.

\begin{figure}
  \centering
  \begin{tikzpicture}[line cap=round,line join=round,>=latex,x=3.0mm,y=3.0mm]

  \fill [color=shadecolor!43] (-0.,-0.) -- (-0.,10.) -- (10.,10.) -- (10.,-0.); 
  \draw[color=black] (-0.,0.) -- (10,0.);

  \draw[color=black] (0.,0.) -- (0.,10);
 
  \draw[color=black] (0.,10.) -- (10.,10);
 
  \draw[color=black] (10.,10.) -- (10.,0.);
 
  \draw [semithick,dash pattern=on 5pt off 5pt,color=black] (-1,11) -- (11,-1);
  \draw [semithick,dash pattern=on 5pt off 5pt,color=black] (11,11) -- (-1,-1);

  \draw [color=red] (5,5) .. controls (6,7) and (4,8) .. (5,10);
  \draw [color=red] (5,5) .. controls (4,3) and (6,2) .. (5,0);
  
  \draw [color=blue] (5,5) .. controls (7,4) and (8,6) .. (10,5);
  \draw [color=blue] (5,5) .. controls (3,6) and (2,4) .. (0,5);

  \clip(-5,-3) rectangle (15,13);
  
  \begin{scriptsize}
  \draw (10.,5.) node[anchor=west] {$\scriptsize{X^{-,r}}$};
  \draw (0.,5.) node[anchor=east] {$\scriptsize{X^{-,l}}$};

  \draw (5.,10.) node[anchor=south] {$\scriptsize{X^{+,r}}$};
  \draw (5.,0.) node[anchor=north] {$\scriptsize{X^{+,l}}$};

  \draw[color=black,->] (-2.,-2.) -- (-0,-2.);
  \draw[color=black,->] (-2.,-2.) -- (-2,-0.);
  \draw (-2.,-1) node[anchor=east] {\scriptsize{$x_{s}$}};
  \draw (-1.,-2) node[anchor=north] {\scriptsize{$x_{u}$}};
\end{scriptsize}

\end{tikzpicture}
  \caption{A horizontal (blue) and a vertical (red) disk satisfying the cone condition in an h-set $X$ with cones $Q(x_u,x_s)=x_u^2 -x_s^2$.}\label{fig:disk}
\end{figure}
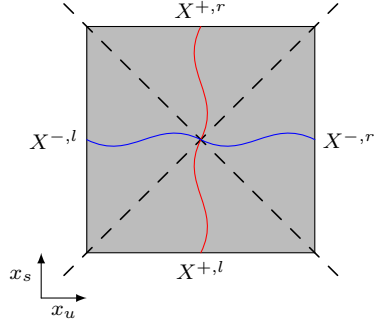 

\begin{thm}[Lemma 5 in~\cite{ZgliczynskiMan}]\label{thm:lipschitz}
  Let $(X,Q)$ be an h-set with cones and let $b$ be a horizontal disk in $X$ satisfying the cone condition.
  Then there exists a Lipschitz map $y: \overline{B_{u(X)}(0,1)} \to \overline{B_{s(X)}(0,1)}$
  such that
  \begin{equation}
      b_c(x_u) = (x_u,y(x_u)).
  \end{equation}
  Analogously, if $b$ is a vertical disk in $X$ satisfying the cone condition,
  then there exists a Lipschitz map $y: \overline{B_{s(X)}(0,1)} \to \overline{B_{u(X)}(0,1)}$
  such that
  \begin{equation}
      b_c(x_s) = (y(x_s),x_s).
  \end{equation}
\end{thm}

\subsection{H-sets for vector fields and isolating blocks}

In this subsection we will consider an ODE
\begin{equation}\label{ODE}
  \begin{aligned}
    \dot{x} &= f(x), \\
    \ x &\in \rr^{N}
  \end{aligned}
\end{equation}
given by a smooth vector field $f$, and describe how h-sets and covering relations can be used in such setting.

\subsubsection{Transversal sections and Poincar\'e maps}\label{subsec:poinc}

Assume, that we are given a diffeomorphism $\Phi: \rr^N \to \rr^N$, 
and let $\Sigma \subset \rr^{N}$ be a subset of the hypersurface $\Xi:=\Phi^{-1}( \{0\} \times \rr^{N-1} )$.
A point $x \in \Xi$ is \emph{regular} iff $\langle n(x), f(x) \rangle \neq 0$, where $n(x)$ is a normal to $\Xi$ at $x$.
If every point $x \in \Sigma$ is regular, then we will say that $\Sigma$ is a \emph{transversal section}.

For a given $x_{0} \in \rr^{N}$ we will denote by
\begin{equation}\label{flow}
  \varphi(t,x_{0})
\end{equation}
the local flow generated by $f$, that is the value of the solution $x(t)$
to~\eqref{ODE} with the initial condition $x(0)=x_{0}$.

Let $\Sigma_{1},\ \Sigma_{2}$ be two transversal sections such that:
\begin{itemize}
  \item either $\Sigma_1 \subset \Sigma_2$ or $\Sigma_1 \cap \Sigma_2 = \emptyset$,
  \item we have $\overline{\inter \Sigma_1} = \Sigma_1$ and $\overline{\inter \Sigma_2} = \Sigma_2$, where closures and interiors are taken in
    the hypersurface topology,
  \item for each $x \in \Sigma_1$ there exists a $\tau > 0$ such that $\varphi(\tau,x) \in \Sigma_2$
\end{itemize}
It is well known that the \emph{Poincar\'e map}:
\begin{equation}\label{poincareMap}
  P : \Sigma_1 \ni x \to  \inf_{\tau: \varphi(\tau,x) \in \Sigma_{2}} \varphi(\tau,x) \in \Sigma_{2}
\end{equation}
is well-defined and smooth for points $x \in \inter \Sigma_1$ such that $P(x) \in \inter \Sigma_2$ (interiors in the hypersurface topology).
The proof can be found in e.g.~\cite{Kapela}.

To make the formulation of some theorems in future shorter, we extend the above definition of a Poincar\'e map
in the scenario $\Sigma_1 \subset \Sigma_2$ to also cover the embedding by identity $\id : \Sigma_1 \to \Sigma_2$.
In such case we will always specifically refer to such map as the identity map, to differentiate from a Poincar\'e map $P$ given by~\eqref{poincareMap}.

For such a Poincar\'e map we define the h-sets in a natural manner. We can identify $\Sigma_{1}, \Sigma_{2}$ with two copies
of $\rr^{N-1}$. Then we can proceed to describe the h-sets on each of these copies - note that they will be h-sets
in $\rr^{N-1}$, not $\rr^{N}$.

\begin{rem}\label{rem:hset}
  Treating h-sets as subsets of sections is a slight abuse when compared to Definition~\ref{h-set}, where they were subsets of 
  the Euclidean space $\rr^{N}$.
  Nevertheless, we can always compose the change of coordinates homeomorphism for the h-set with the global coordinate frame on the section
  to get back to the Euclidean space.
  Therefore, given a section $\Sigma \subset \rr^{N}$, for an h-set $X \subset \Sigma$ the actual coordinate
  change will take the form $c_{X} =  \tilde{c}_{X} \circ \Phi$,
  where $\Phi : \Sigma \to \{0\} \times \rr^{N-1}$ is the global coordinate frame for the section and $\tilde{c}_{X} : \rr^{N-1} \to \rr^{u(X)} \times \rr^{s(X)}$
  is a coordinate change homeomorphism satisfying \eqref{h-set:1}.
\end{rem}

\subsubsection{Isolating blocks}\label{sec:blockbasic}

In this subsection we will give a working definition of an isolating block for~\eqref{ODE}, which is a special case of the
classical definition from Conley theory used in~\cite{szczelir},
and cite a theorem from that paper on the existence of unstable and stable manifolds in blocks.

For convenience of formulation we assume, that we only work in the $\max$ norm.
Let $\pi_i$ denote the projection onto $i$-th coordinate in $\rr^N$.

\begin{defn}[cf. Lemma 16 in~\cite{ZgliczynskiMan}]
  Let $B \subset \rr^N$ be an $N$-dimensional h-set in $\rr^N$ with $c_B$ given by a diffeomorphism.
  We say that $B$ is an isolating block, iff  
 the following conditions hold for all:
  \begin{itemize}
    \item[(B1)] $\frac{d}{dt} | \pi_{i} c_{B}( \varphi(t,x) ) |_{|_{t=0}} > 0$ for all $i \in \{1,\dots , u(B)\}$ 
      and all $x \in B^{-}$ such that $|\pi_i c_B(x)|=1$ (\emph{exit set isolation}),
    \item[(B2)] $\frac{d}{dt} | \pi_{i} c_{B}( \varphi(t,x) ) |_{|_{t=0}} < 0$ for all $i \in \{u(B)+1,\dots , N\}$ 
      and all $x \in B^{+}$ such that $|\pi_i c_B(x)|=1$ (\emph{entrance set isolation}).
  \end{itemize}
\end{defn}

Note, that contrary to h-sets on transversal sections, isolating blocks are objects of dimension $N$.
The following remark gives alternative conditions for isolation.

\begin{rem}\label{alternativeB}
  The conditions (B1)-(B2) are equivalent to:
  \begin{itemize}
    \item[(B1a)] $\langle \nabla |( \pi_{i} \circ c_{B} )(x)|, f(x) \rangle > 0$ for all $i \in \{1,\dots , u(B)\}$ and all $x \in B^{-}$ such that $|\pi_i c_B(x)|=1$,
    \item[(B2a)] $\langle \nabla |( \pi_{i} \circ c_{B} )(x)|, f(x) \rangle < 0$ for all $i \in \{u(B)+1,\dots , N\}$ and all $x \in B^{+}$ such that $|\pi_i c_B(x)|=1$,
  \end{itemize}
  respectively.

  Since $B^{-}$, $B^{+}$ are formed of subsets of the level sets of the form $\{ x \in \rr^{N}: | \pi_i c_{B}(x) | = 1 \}$,
  and gradients are normals to level sets, (B1) and (B2) can also be expressed as as:
  \begin{itemize}
    \item[(B1b)] $\langle n_{i}(x), f(x) \rangle > \const > 0$ for all $i \in \{1,\dots , u(B)\}$ and all $x \in B^{-}$ such that $|\pi_i c_B(x)|=1$,
    \item[(B2b)] $\langle n_{i}(x), f(x) \rangle < \const < 0$ for all $i \in \{u(B)+1,\dots , N\}$ and all $x \in B^{+}$ such that $|\pi_i c_B(x)|=1$ ,
  \end{itemize}
  respectively, 
  where $n_{i}(x)$ are normals to $\{ x \in \rr^{N}: | \pi_i c_{B}(x) | = 1 \}$, 
  pointing in the outward direction of $|B|$.
\end{rem}

Put simply, an isolating block is a set diffeomorphic to a ball, faces of which being transversal sections.
In addition, the vector field on opposite faces is required to point in opposite directions, see Figure~\ref{blockFig}.

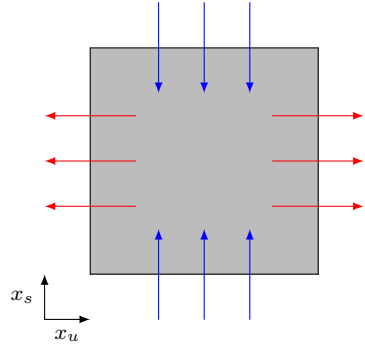
\begin{figure}
    \centering
      \begin{tikzpicture}[line cap=round,line join=round,>=latex,x=3mm,y=3mm]
 
  \clip(-19,-8.4) rectangle (19,7);

  \fill [color=shadecolor!43] (-5.,-5.) -- (-5.,5.) -- (5.,5.) -- (5.,-5.);
  \draw [color=black] (-5.,5.) -- (-5.,-5.);
  \draw [color=black] (5.,5.) -- (5.,-5.);
  \draw [color=black] (-5.,5.) -- (5.,5.);
  \draw [color=black] (5.,-5.) -- (-5.,-5.);
 
 

  \draw[color=red,->] (3.,2.) -- (7.,2.);
  \draw[color=red,->] (3.,-2.) -- (7.,-2.);
  \draw[color=red,->] (3.,-0.) -- (7.,-0.);

  \draw[color=red,->] (-3.,2.) -- (-7.,2.);
  \draw[color=red,->] (-3.,-2.) -- (-7.,-2.);
  \draw[color=red,->] (-3.,-0.) -- (-7.,-0.);

  \draw[color=blue,->] (2.,-7.) -- (2.,-3.);
  \draw[color=blue,->] (-2.,-7.) -- (-2.,-3.);
  \draw[color=blue,->] (-0.,-7.) -- (-0.,-3.);

  \draw[color=blue,->] (2.,7.) -- (2.,3.);
  \draw[color=blue,->] (-2.,7.) -- (-2.,3.);
  \draw[color=blue,->] (-0.,7.) -- (-0.,3.);

  \draw[color=black,->] (-7.,-7.) -- (-5,-7.);
  \draw[color=black,->] (-7.,-7.) -- (-7,-5.);
  \draw (-7.,-6) node[anchor=east] {\scriptsize{$x_{s}$}};
  \draw (-6.,-7) node[anchor=north] {\scriptsize{$x_{u}$}};

\end{tikzpicture}
      \caption{An isolating block.}
      \label{blockFig}
\end{figure}

\begin{defn}[Definition 11 in~\cite{szczelir}]
Let $B$ be an h-set in $\rr^N$ such that $c_B$ is given by a diffeomorphism. We define a vector field $f_c$ on $B_c$ by
\begin{equation}
  f_c(x) = Dc_B (c_B^{-1}(x))f(c_B^{-1}(x)).
\end{equation}
\end{defn}
The new field $f_c$ is in fact the vector field~\eqref{ODE} expressed in variables given by $c_B$.

\begin{defn}[Definition 12 in~\cite{szczelir}]
  Let $U$ be a closed subset of $\rr^N$ with a nonempty interior. Let $g: U \to \rr^M$ be a smooth map.
  We define the interval enclosure of $Dg(U)$ by
  \begin{equation}
    [Dg(U)]:= \left\{ A \in \rr^{N \times M}: A_{ij} \in \left[ \inf_{x \in U} \frac{\partial g_i}{\partial x_j}, \sup_{x \in U} \frac{\partial g_i}{\partial x_j} \right] 
  \ \forall i,j \right\}.
  \end{equation}
  We say that $[Dg(U)]$ is positive definite, iff for all $A \in [Dg(U)]$ the matrix $A$ is positive definite.
\end{defn} 

We recall the definition of the unstable and stable manifold of an equilibrium point.
\begin{defn}\label{defn:man}
  Let $B \subset \rr^N$ and let $x_0 \in B$ be such that $f(x_0)=0$. We define the sets
  \begin{align}
    W^u(x_0) &:= \{ x: \ \lim_{t \to -\infty} \varphi(t,x) = x_0 \},\\
    W^s(x_0) &:= \{ x: \ \lim_{t \to \infty} \varphi(t,x) = x_0 \},\\
    W^u_B(x_0) &:= \{ x: \ \forall t\leq 0 \ \varphi(t,x) \in B \text{ and } \lim_{t \to -\infty} \varphi(t,x) = x_0 \},\\
    W^s_B(x_0) &:= \{ x:\ \forall t\geq 0 \ \varphi(t,x) \in B \text{ and } \lim_{t \to \infty} \varphi(t,x) = x_0 \}.
  \end{align}
  We say that $W^{u}(x_0)$ is the \emph{unstable manifold} of $x_0$ and $W^s(x_0)$ is the \emph{stable manifold} of $x_0$. We will
  use the same names for $W^{u}_B(x_0)$ and $W^{s}_B(x_0)$, respectively, but it will always be clear from the context, 
  to which object we are referring to. 
\end{defn}

The statement of the following Theorem was given in~\cite{szczelir}. The proof follows from the proof of
the stable manifold theorem given in~\cite{ZgliczynskiMan} (Theorem 26).
\begin{thm}[Theorem 4 in~\cite{szczelir}]\label{szczelir}
  Let $(B,Q)$ be an h-set with cones, such that $B$ is an isolating block for~\eqref{ODE} and the matrix
  \begin{equation}\label{eq:szczelir}
    [Df_c(B_c)]^T Q  + Q [Df_c(B_c)]
  \end{equation}
  is positive definite. Then, there exists a unique stationary point $x_0 \in |B|$ of~\eqref{ODE}. Moreover, $W_B^u(x_0)$ is a horizontal disk in $B$ 
  satisfying the cone condition and $W_B^s(x_0)$ is a vertical disk in $B$ satisfying the cone condition.
\end{thm}

\begin{defn}
  In a situation when $B$ is an isolating block that can be equipped with cones $Q$, such that the condition~\eqref{eq:szczelir} is satisfied,
  we will sometimes say that $B$ is an isolating block satisfying the cone condition.
\end{defn}
 
\begin{rem}[Lemma 27 in~\cite{ZgliczynskiMan}]
  The geometrical meaning of positive definiteness of the matrix~\eqref{eq:szczelir} is that it 
  implies the following condition, known as \emph{the cone condition for flows}:
  \begin{equation}
    \frac{d}{dt} Q ( x_1(t) - x_2(t) )_{|_{t=0}} > 0,
  \end{equation} 
  for any two distinct solutions $x_1, x_2: [0,T] \to \rr^N$ of $\dot{x}=f_{c}(x)$ such that $x_1(0), x_2(0) \in B_c$.
\end{rem}

The following theorem follows from combined proofs of Theorems 26 and 20 (on continuous dependence of the manifolds on parameters) 
in~\cite{ZgliczynskiMan}. Continuous and smooth dependence of unstable/stable manifolds of hyperbolic equilibria on parameters is in fact a standard
result in dynamical systems theory, e.g. Theorem 4.1 in~\cite{Hirsch}.
  
\begin{thm}\label{szczelirPar}
  Assume that $(B,Q)$ is an h-set with cones and let 
  \begin{equation}
    \dot{x} = f(x,z), \quad x \in \rr^N,\ z \in \rr^p
  \end{equation}
  be a vector field depending on parameter $z$, given by a smooth map $f$.
  Assume that for all $z \in Z \subset \rr^p$ 
  the vector field $f(\cdot,z)$ together with the pair $(B,Q)$ satisfy assumptions of Theorem~\ref{szczelir}.
  Then, the horizontal and vertical disks $W_{B,z}^u$, $W_{B,z}^s$ parameterizing the unstable and the stable manifold of the equilibrium $x_0(z) \in |B|$
  are continuous as maps from $\overline{B_{u(B)}(0,1)} \times Z$ and $\overline{B_{s(B)}(0,1)} \times Z$ to $|B|$, respectively.
\end{thm}

In what follows, we will show that the intersection of a horizontal/vertical disk satisfying the cone condition with a part of the block boundary
is a horizontal/vertical disk within h-sets on the boundary of the block.
We will use that lemma later, to conclude that the two-dimensional stable manifold of an equilibrium of the FitzHugh-Nagumo equation
generates a one-dimensional vertical disk on a transversal section.

Till the rest of this subsection $\hat{\pi}_i$  will denote the projection from $\rr^N$ onto $\rr^{N-1}$ given by omitting the $i$-th coordinate.

\begin{defn}\label{defn:bdset}
  Let $B$ be an h-set such that $B$ is an isolating block for~\eqref{ODE}.
  We define its $i$-th \emph{boundary h-set} (see Remark~\ref{rem:hset}) by
  \begin{equation}
    \begin{aligned}
    |X_{B,i}| &:=  \{x \in \partial |B|: ( \pi_{|i|} \circ c_B )(x) = \sgn{i}  \},\\  
    c_{X_{B,i}} &:= \hat{\pi}_{|i|} \circ c_B,\\
    i &\in \{-N,\dots,-1,1,\dots,N\}.
  \end{aligned}
  \end{equation}
  For $|i| \leq u(B)$ we set $u(X_{B,i}) := u(B)-1$, $s(X_{B,i}) := s(B)$, elsewise $u(X_{B,i}) := u(B)$, $s(X_{B,i}) := s(B)-1$.
\end{defn}

\begin{lem}\label{lem:sidedisk}
  Let $(B,Q)$ be an h-set with cones, such that $B$ is an isolating block for~\eqref{ODE} and $u(B)>1$ .
  Let $b$ be a horizontal disk in $B$ satisfying the cone condition. 
  Then the sets $b(\overline{B_{u(B)}(0,1)}) \cap |X_{B,i}|$ are horizontal disks in $X_{B,i}$,
  for $|i| \in \{1,\dots,u(B)\}$.

  Analogously, if $s(B)>1$ and $b$ is a vertical disk satisfying the cone condition, then the sets
  $b(\overline{B_{s(B)}(0,1)}) \cap |X_{B,i}|$ are vertical disks for $|i| \in \{u(B)+1, \dots,N\}$.
\end{lem}

\begin{proof}
  We will only prove the first assertion, the case $i=1$. Proofs of all the other cases are analogous.

  We define a disk $\tilde{b}: \overline{B_{u(B)-1}(0,1)} \to |X_{B,1}|$ by 
  \begin{equation}
    \tilde{b}(x) = b\left(1,x_1,\dots,x_{u(B)-1}\right) \quad \text{ for } x=(x_1,\dots,x_{u(B)-1}).
  \end{equation}

  By Theorem~\ref{thm:lipschitz} we have $b_c(x)= c_B( b(x) ) = (x,y(x))$ for all $x \in \overline{B_{u(B)}(0,1)}$ and some Lipschitz map $y$. 
  Therefore, $\tilde{b}$ indeed maps into $|X_{B,1}|$, moreover
  \begin{equation}
    \tilde{b}_c(x) = (c_{X_{B,1}} \circ \tilde{b})(x) = (\hat{\pi}_1 \circ c_{B} \circ \tilde{b})(x)= \hat{\pi}_1 (1,x,y(x)) = (x,y(1,x)),
  \end{equation}
  for all $x \in \overline{B_{u(B)-1}(0,1)}$.
  The homotopy required in the definition of a horizontal disk can now be given by
  \begin{equation}
    \tilde{h}(\xi,x) = (x,(1-\xi) y(1,x)).
  \end{equation}
  Indeed, we have  
  \begin{equation}
    \tilde{h}\left([0,1], \partial B_{u(B)-1}(0,1) \right) \subset \partial B_{u(B)-1}(0,1) \times \overline{B_{s(B)}(0,1)},
  \end{equation}
  hence $\tilde{b}$ is a horizontal disk.
\end{proof}

\section{Isolating segments}\label{sec:segments}

In this section we will give a definition of isolating segments in an autonomous ODE and prove several useful theorems about them.
Such objects will be geometrically similar to isolating blocks, and can be used to track trajectories of the given system
without need for rigorous integration.

Assume, that we are given a smooth vector field
\begin{equation}
  \begin{aligned}
  \dot{x} &= f(x),\\
  x &\in \rr^N,
\end{aligned}
\end{equation}
an associated local flow $\varphi(t,x)$
and a pair of transversal sections $\Sigma_{\text{in}}, \ \Sigma_{\text{out}}$.

\begin{defn}
  A segment between two transversal sections $\Sigma_{\text{in}}$ and $\Sigma_{\text{out}}$ is formed by a quadruple $S = (|S|, u(S), s(S), c_{S})$,
  consisting of a compact set $|S| \subset \rr^{N}$ (\emph{the support}), a pair of numbers $u(S), s(S) \in \mathbb{N}$
  with $u(S) + s(S) = N-1$ (the number of exit and entrance directions, respectively)
  and a coordinate change diffeomorphism $c_{S}: \rr^{N} \to \rr^{u(S)} \times \rr^{s(S)} \times \rr$ such that:
  \begin{equation}
    \begin{aligned}
      c_{S}(|S|) &= \overline{B_{u(S)}(0,1)} \times \overline{B_{s(S)}(0,1)} \times [0,1], \\
      c_{S}^{-1}( \overline{B_{u(S)}(0,1)} \times \overline{B_{s(S)}(0,1)} \times \{0\} ) &\subset \Sigma_{\text{in}}, \\
      c_{S}^{-1}( \overline{B_{u(S)}(0,1)} \times \overline{B_{s(S)}(0,1)} \times \{1\}) &\subset \Sigma_{\text{out}} . \\
    \end{aligned}
  \end{equation}

  We set:
  \begin{equation}
    \begin{aligned}
      S_{c}&:=\overline{B_{u(S)}(0,1)} \times \overline{B_{s(S)}(0,1)} \times [0,1], \\
      S_{c}^{-}&:=\partial B_{u(S)}(0,1) \times \overline{B_{s(S)}(0,1)} \times [0,1], \\
      S_{c}^{+}&:=\overline{B_{u(S)}(0,1)} \times \partial B_{s(S)}(0,1) \times [0,1], \\
      S^{-} &:= c_{S}^{-1}(S_{c}^{-}), \\
      S^{+} &:= c_{S}^{-1}(S_{c}^{+}).
    \end{aligned}
  \end{equation}

  We will refer to $S^{-}$/$S^{+}$ as \emph{the exit/entrance sets}, respectively.
  Again, to shorten the notation sometimes we will drop the bars in the symbol $|S|$ and just write $S$ to denote both the
  segment and its support.
\end{defn}

\begin{rem}
  As with h-sets, it is enough to give $u(S)$, $s(S)$ and $c_{S}$ to define a segment $S$.
\end{rem}

Given a segment $S$ we introduce the following notation for projections:

\begin{equation}
  \begin{aligned}
    \pi_{u}: \rr^{u(S)} \times \rr^{s(S)} \times \rr \ni (x_{u},x_{s},x_{\mu}) &\to x_{u} \in \rr^{u(S)}, \\
    \pi_{s}: \rr^{u(S)} \times \rr^{s(S)} \times \rr \ni (x_{u},x_{s},x_{\mu}) &\to x_{s} \in \rr^{s(S)}, \\
    \pi_{\mu}: \rr^{u(S)} \times \rr^{s(S)} \times \rr \ni (x_{u},x_{s},x_{\mu}) &\to x_{\mu} \in \rr.
  \end{aligned}
\end{equation}

\begin{defn}\label{defn:isegment}
  We say that $S$ is an isolating segment between two transversal sections $\Sigma_{\text{in}}$ and $\Sigma_{\text{out}}$
  if $S$ is a segment,
  the functions $x \to \vectornorm{\pi_{u}(x)}$, $x \to \vectornorm{\pi_{s}(x)}$, $x \to \vectornorm{\pi_{\mu}(x)}$,
  $x \in \rr^{u(S)} \times \rr^{s(S)} \times \rr$ are smooth everywhere except at $0$
  and the following conditions are satisfied:
  \begin{itemize}
    \item[(S1)] $\frac{d}{dt} \pi_{\mu} c_{S}( \varphi(t,x) )_{|_{t=0}} > 0$ for all $x \in |S|$ (\emph{monotonicity}),
    \item[(S2)] $\frac{d}{dt} \vectornorm{ \pi_{u} c_{S}( \varphi(t,x) ) }_{|_{t=0}} > 0$ for all $x \in S^{-}$ (\emph{exit set isolation}),
    \item[(S3)] $\frac{d}{dt} \vectornorm{ \pi_{s} c_{S}( \varphi(t,x) ) }_{|_{t=0}} < 0$ for all $x \in S^{+}$ (\emph{entrance set isolation}).
  \end{itemize}
\end{defn}

As one can see, our definition of an isolating segment $S$ relies on splitting the phase space into:
\begin{itemize}
  \item the \emph{exit} directions $\pi_{u} \circ c_{S}$,
  \item the \emph{entry} directions $\pi_{s} \circ c_{S}$,
  \item the one-dimensional \emph{central} direction $\pi_{\mu} \circ c_{S}$.
\end{itemize}
In that sense, it is a simplification of the concept of periodic isolating segments in nonautonomous systems, as originally introduced in~\cite{SrzednickiWojcik}
(also, under the name of periodic isolating blocks in~\cite{Srzednicki}),
where a wider range of boundary behavior was considered.
On the other hand, contrary to~\cite{SrzednickiWojcik}, we are able to work with an autonomous ODE
-- in~\cite{SrzednickiWojcik} the central direction had to be given by time.

When introducing an isolating segment, we will sometimes omit specifying the transversal sections $\Sigma_{\text{in}}, \Sigma_{\text{out}}$
-- in such situation we consider the suitable sections implicitly defined by $c_{S}$.

\begin{rem}\label{alternativeS}
  Each of the conditions (S1)-(S3) is equivalent to its following counterpart:
  \begin{itemize}
    \item[(S1a)] $\langle \nabla (\pi_{\mu} \circ c_{S})(x), f(x) \rangle > 0$ for all $x \in |S|$,
    \item[(S2a)] $\langle \nabla \vectornorm{ \pi_{u} \circ c_{S} }(x), f(x) \rangle > 0$ for all $x \in S^{-}$,
    \item[(S3a)] $\langle \nabla \vectornorm{ \pi_{s} \circ c_{S} }(x), f(x) \rangle < 0$ for all $x \in S^{+}$.
  \end{itemize}

  Since $S^{-}$, $S^{+}$ are subsets of the level sets $\{ x \in \rr^{N}: \vectornorm{ \pi_{u} \circ c_{S}(x) } = 1 \}$,
  $\{ x \in \rr^{N}: \vectornorm{ \pi_{s} \circ c_{S}(x) } = 1 \}$, respectively, and gradients are normals to level sets,
  (S2a) and (S3a) can also be rewritten as:
  \begin{itemize}
    \item[(S2b)] $\langle n_{-}(x), f(x) \rangle > \const > 0$ for all $x \in S^{-}$,
    \item[(S3b)] $\langle n_{+}(x), f(x) \rangle < \const < 0$ for all $x \in S^{+}$,
  \end{itemize}
  respectively, 
  where $n_{\mp}(x)$ are normals to $S^{\mp}$, pointing in the outward direction of $|S|$\footnote{The sets $S^{\mp}$ are manifolds with boundary, so by normals
  at the boundary points we mean normals to any smooth extension of $S^{\mp}$ to a manifold without boundary.}.
\end{rem}

In our applications the faces of segments will always lie in affine subspaces,
hence conditions (S2b) and (S3b) are easy to check by an explicit computation.
Let $\pi_i$ be the projection onto $i$-th variable, $i \in 1,\dots,N$.
In the central direction our changes of coordinates will take an affine form
\begin{equation}\label{eq:centralform}
  \pi_{\mu} c_{S}(x) = a \pi_i (x) + b,
\end{equation}
for $a \neq 0,\ b \in \rr$.
In that situation (S1a) is equivalent with
\begin{equation}\label{eq:sgna}
  \sgn(a) \pi_i (f(x)) > 0, \ \forall x \in |S|,
\end{equation}
which again is easily established. In particular, if the sign of $\pi_i(f(x))$ is negative,
one needs to orient the segment in the direction reverse to the $i$-th coordinate direction
by giving $a$ a negative sign.

We will now introduce the notion of the transposed segment, analogous to the transposed h-set.

\begin{defn}
  Given a segment $S$ between two transversal sections $\Sigma_{\text{in}}$ and $\Sigma_{\text{out}}$ we define the
  transposed segment $S^{T}$ between $\Sigma_{\text{out}}$ and $\Sigma_{\text{in}}$ by setting:
  \begin{equation}
    \begin{aligned}
      |S^{T}| &:= |S|, \\
      u(S^{T}) &:= s(S), \\
      s(S^{T}) &:= u(S), \\
      c_{S^{T}} &:= {\rm o} \circ c_{S};
    \end{aligned}
  \end{equation}
  where ${\rm o}: \rr^{u(S)} \times \rr^{s(S)} \times \rr \to \rr^{u(S^{T})} \times \rr^{s(S^{T})} \times \rr$, ${\rm o}(p,q,r) = (q,p,1-r)$.

  Observe that
  \begin{equation}
    \begin{aligned}
      (S^{T})^{-} &= S^{+}, \\
      (S^{T})^{+} &= S^{-}.
    \end{aligned}
  \end{equation}

\end{defn}

\begin{prop}\label{ts:1}
  Let $S$ be an isolating segment between transversal sections $\Sigma_{\text{in}}$ and $\Sigma_{\text{out}}$ for $\dot{x} = f(x)$.
  Then $S^{T}$ is an isolating segment between $\Sigma_{\text{out}}$ and $\Sigma_{\text{in}}$ for the inverted vector field $\dot{x} = -f(x)$.
  The sections $\Sigma_{\text{out}}$ and $\Sigma_{\text{in}}$ are transversal for the inverted vector field.
\end{prop}

\subsection{Isolating segments imply coverings}

Given a segment $S$ between transversal sections $\Sigma_{\text{in}}$ and $\Sigma_{\text{out}}$ there is a natural structure of h-sets defined
on the faces given by intersections $\Sigma_{\text{in}} \cap |S|$ and $\Sigma_{\text{out}} \cap |S|$.

\begin{defn}
  We define the h-sets:
  \begin{itemize}
    \item $X_{S,\text{in}} \subset \Sigma_{\text{in}}$ (\emph{the front face}),
    \item $X_{S,\text{out}} \subset \Sigma_{\text{out}}$ (\emph{the rear face}),
  \end{itemize}
  as follows:
  \begin{itemize}
    \item $u( X_{S,\text{in}} ) = u( X_{S,\text{out}} ) := u(S)$ and $s( X_{S,\text{in}} ) = s( X_{S,\text{out}} ) := s(S)$;
    \item $|X_{S, \text{in}}| =\Sigma_{\text{in}} \cap |S|$ and $|X_{S, \text{out}}| = \Sigma_{\text{out}} \cap |S|$;
    \item $c_{ X_{S,\text{in}} } := (\pi_{u},\pi_{s}) \circ {c_{S}}_{|_{\Sigma_{\text{in}}}}$ 
      and $c_{ X_{S,\text{out}} } := (\pi_{u},\pi_{s}) \circ {c_{S}}_{|_{\Sigma_{\text{out}}}}$.
  \end{itemize}

\end{defn}

\begin{defn}
  Let $S$ be an isolating segment. We define \emph{the exit map} $E_{S}: |X_{S,\text{in}}| \to S^{-} \cup |X_{S,\text{out}}|$ by
  \begin{equation}
   E_{S}(x)=\varphi(t_{e},x), \ t_{e} = \min \left\{ t \in \rr^{+} \cup \{0\}: \ \varphi(t,x) \in S^{-} \cup |X_{S,\text{out}}| \right\},
  \end{equation}
  and \emph{the persistent set} by
  \begin{equation}
   S^{0} := \{ x \in |X_{S,\text{in}}|: E_{S}(x) \in |X_{S,\text{out}}| \}.
  \end{equation}
\end{defn}

\begin{rem}
  From (S1), (S2), (S3) it follows that the function $E_{S}$ is well-defined and a homeomorphism onto its image. 
\end{rem}

\begin{thm}\label{is:1}
  Let $S$ be an isolating segment between transversal sections $\Sigma_{\text{in}}$ and $\Sigma_{\text{out}}$.
  Define $V := \{x \in \Sigma_{\text{in}}: \exists \tau > 0 : \varphi(\tau,x) \in \Sigma_{\text{out}} \}$
  and a Poincar\'e map $P: V \to \Sigma_{out}$ as in equation~\eqref{poincareMap}. Then
  \begin{itemize}
    \item $V \neq \emptyset$;
    \item there exists a diffeomorphism $R: \Sigma_{\text{in}} \to \Sigma_{\text{out}}$ such that we have a covering relation
      \begin{equation}
        X_{S, \text{in}} \cover{R} X_{S, \text{out}}
      \end{equation}
    and
    \begin{equation}\label{R=P}
      P(x) = R(x) \quad \forall x \in S^{0};
    \end{equation}
    \item it holds that 
      \begin{equation}\label{S0=RSet}
        S^{0}= \{ x \in |X_{S,\text{in}}|: R(x) \in |X_{S,\text{out}}| \}.
      \end{equation}
      In particular, for every $x \in |X_{S, \text{in}}|$ such that $R(x) \in |X_{S, \text{out}}|$ the part of the trajectory
      between $x$ and $P(x)=R(x)$ is contained in $|S|$.
  \end{itemize}
\end{thm}

The intuition behind this theorem is portrayed in Figure~\ref{segmentCovFig}.
The proof was inspired by the algorithm for integration of ill-posed PDEs, presented in~\cite{ZgliczynskiUnpublished}. 

\begin{figure}
    \centering
      \input{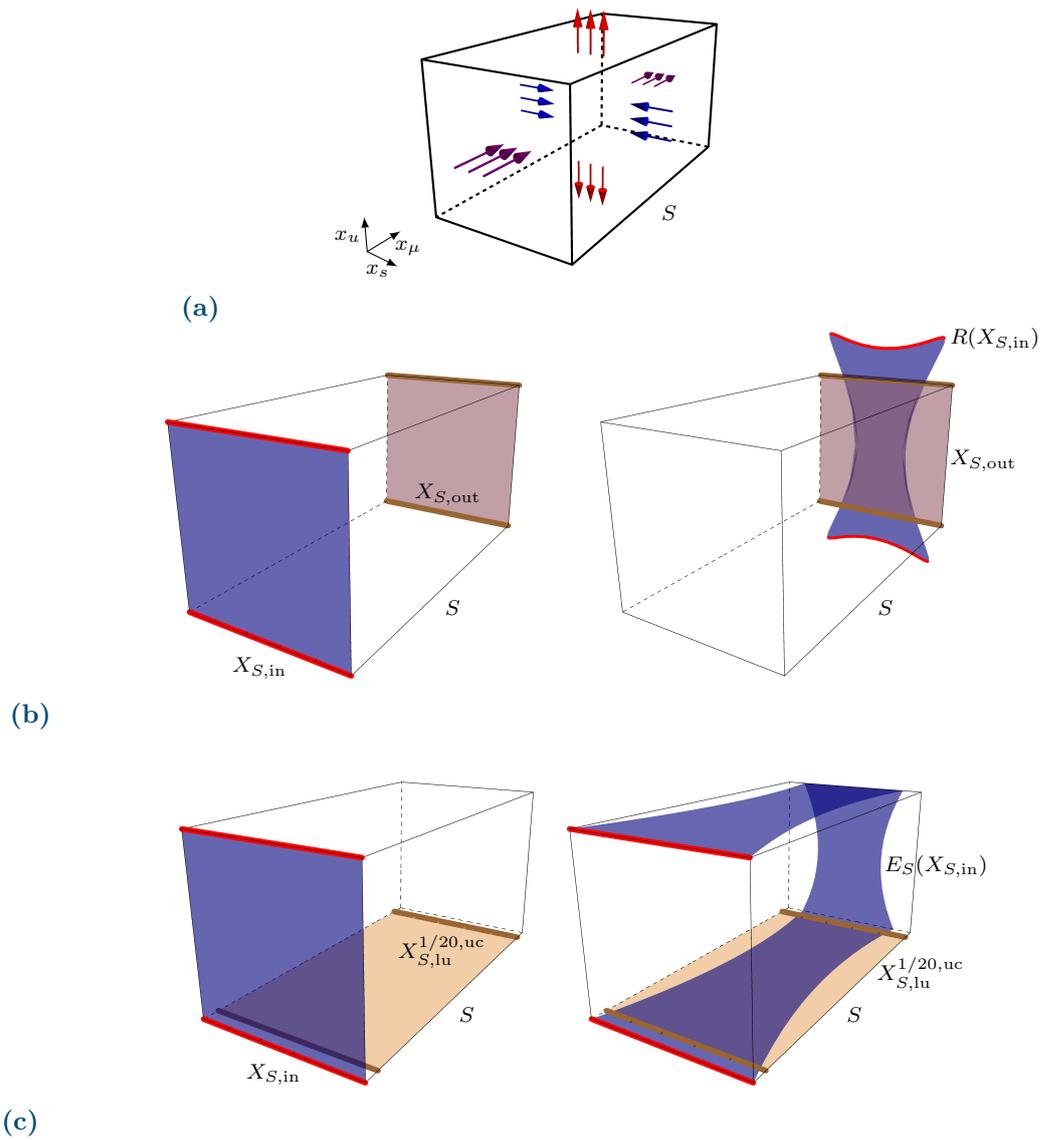}
      \caption{Schematic drawings of covering relations in a three-dimensional isolating segment $S$ with $u(S)=s(S)=1$.}
      \label{segmentCovFig}
\end{figure}

\begin{proof}
  To make the formulas clearer, without any loss of generality we assume that $c_{S}=\id_{\rr^{N}}$. Define
  \begin{equation}
    g(x) = \frac{d}{dt} \tilde{\varphi}( t, x )_{|_{t=0}},
  \end{equation}
  where
  \begin{equation}
    \tilde{\varphi}(t,x) = \left( \begin{array}{c} e^{t} \pi_{u}(x) \\ e^{-t} \pi_{s}(x) \\ t + \pi_{\mu}(x) \end{array} \right)
  \end{equation}
  is a global flow on $\rr^{u(S)} \times \rr^{s(S)} \times \rr$.

  We have
    \begin{align}
      \langle \nabla \vectornorm{ \pi_{u} (x) }, g(x) \rangle &= \frac{d}{dt} \vectornorm{ \pi_{u} \tilde{\varphi}(t, x ) }_{|_{t=0}} = \vectornorm{ \pi_{u} (x) },\label{eq:ga} \\
      \langle \nabla \vectornorm{ \pi_{s} (x) }, g(x) \rangle &= \frac{d}{dt} \vectornorm{ \pi_{s} \tilde{\varphi}(t, x ) }_{|_{t=0}} = -\vectornorm{ \pi_{s} (x) },\label{eq:gb}
    \end{align}
  for all $x \neq 0$ and   
  \begin{equation}
    \langle \nabla \pi_{\mu} (x), g(x) \rangle = \frac{d}{dt} \pi_{\mu} \tilde{\varphi}(t, x )_{|_{t=0}} = 1,\label{eq:gc}
  \end{equation}
  for all $x \in \rr^{N}$.

  Let $U$ be a bounded, open neighborhood of $|S|$, small enough so that the following conditions are satisfied:
    \begin{align}
      \langle \nabla \vectornorm{ \pi_{u} (x) }, f(x) \rangle > \const > 0 \ &\forall x \in U: \vectornorm{ \pi_{u} (x) } \geq 1, \label{eq:fa}\\
      \langle \nabla \vectornorm{ \pi_{s} (x) }, f(x) \rangle < \const < 0 \ &\forall x \in U: \vectornorm{ \pi_{s} (x) } \geq 1, \label{eq:fb}\\
      \langle \nabla \pi_{\mu} (x), f(x) \rangle > \const > 0 \ &\forall x \in U.\label{eq:fc}
    \end{align}

  Let $\eta: \rr^{N} \to [0,1]$ be a $C^{\infty}$ function equal to $1$ on $|S|$ and equal to $0$ on $\rr^{N} \backslash U$.
  Put $\hat{f}(x) = \eta(x) f(x) + (1-\eta(x))g(x)$. Denote by $\hat{\varphi}(t,x)$ the local flow generated by
  \begin{equation}
    \dot{x} = \hat{f}(x), \quad x \in \rr^{N}.
  \end{equation}
  Since $\hat{\varphi}(t,x) = \tilde{\varphi}(t, x),\ x \in \rr^{N} \backslash U$ and $U$ is bounded,
  $\hat{\varphi}$ is also a global flow.
  From~\eqref{eq:gc} and~\eqref{eq:fc} we have:
  \begin{equation}
    \begin{aligned}
      \langle \nabla \pi_{\mu} (x), \hat{f}(x) \rangle
      &= \eta(x)\langle \nabla \pi_{\mu} (x), f(x) \rangle
      + (1 - \eta(x)) \langle \nabla \pi_{\mu} (x), g(x) \rangle
      \\ &> \const > 0
    \end{aligned}
  \end{equation}
  for all $x \in \rr^{N}$. Therefore, the Poincar{\'e} map
  \begin{equation}
    P_{\hat{f}}: \rr^{u(S)} \times \rr^{s(S)} \times \{ 0 \} \to \rr^{u(S)} \times \rr^{s(S)} \times \{ 1 \},
  \end{equation}
  is a well-defined diffeomorphism. We set $R := P_{\hat{f}}$.

  First, we will prove that $X_{S, \text{in}}$ $R$-covers $X_{S, \text{out}}$, cf. Figures~\ref{segmentCovFig:A} and~\ref{segmentCovFig:B}.
  In what is below we identify the spaces $\rr^{u(S)} \times \rr^{s(S)} \times \{0\}$ and $\rr^{u(S)} \times \rr^{s(S)} \times \{1\}$ 
  with two copies $\rr^{u(S)} \times \rr^{s(S)}$ wherever necessary, by projecting/embedding the first $u(S)+s(S)$ coordinates.
 
  We need a homotopy of $R$ to a linear map.
  Consider the parameterized family of vector fields
  \begin{equation}
    f_{\xi}(x) = (1 - \xi) \hat{f}(x) + \xi g(x), \quad x \in \rr^{N},
  \end{equation}
  where $\xi \in [0,1]$. By the same reasoning as with $\hat{f}$ each of these vector fields
  generates a global flow and induces an associated Poincar{\'e} map
  \begin{equation}
    P_{f_{\xi}}: \rr^{u(S)} \times \rr^{s(S)} \times \{ 0 \} \to \rr^{u(S)} \times \rr^{s(S)} \times \{ 1 \}.
  \end{equation}
  We define a continuous
  homotopy of maps $h: [0,1] \times \rr^{u(S)} \times \rr^{s(S)} \to \rr^{u(S)} \times \rr^{s(S)}$:
  \begin{equation}
    \begin{aligned}
      h(\xi, \cdot )&:= P_{f_{2\xi}},\quad \xi \in [0, 1/2], \\
      h(\xi, \cdot )&:= \left[ \begin{array}{cc} e\id_{\rr^{u(S)}} & 0 \\ 0 & (2-2\xi)e^{-1}\id_{\rr^{s(S)}} \end{array} \right],\quad \xi \in [1/2, 1].
    \end{aligned}
  \end{equation}
  Indeed, the homotopy agrees at $1/2$. Moreover, $h(0,\cdot) = R$
  and $h(1,\cdot)$ is a linear map satisfying the requirements given by~\eqref{cover:2}. Since 
  it is also clear that~\eqref{cover:1b} and~\eqref{cover:1c}
  hold for $\xi \in [1/2, 1]$, we proceed to check these two conditions on the other half of the interval.

  Denote by $\varphi^{\xi}$ the family of global flows generated by $\dot{x} = f_{\xi}(x)$.
  From~\eqref{eq:fa} and~\eqref{eq:ga}, for $\xi \in [0, 1]$ and $x: \vectornorm{\pi_{u} (x)} \geq 1$ we get

  \begin{equation}\label{hatf:1}
   \begin{aligned}
     &\frac{d}{dt} \vectornorm{ \pi_{u}  \varphi^{\xi}(t, x)  ) }_{|_{t=0}} \\
     & =\langle \nabla \vectornorm{ \pi_{u} (x)}, f_{\xi}(x) \rangle\\
     &= (1-\xi)\langle \nabla \vectornorm{ \pi_{u} (x)} , \hat{f}(x) \rangle
     + \xi \langle \nabla \vectornorm{\pi_{u} (x)} , g(x) \rangle  \\
     &= (1-\xi) \eta(x) \langle \nabla \vectornorm{\pi_{u} (x)} , f(x) \rangle
     + ( 1-\eta(x) + \xi \eta(x) ) \langle \nabla \vectornorm{\pi_{u} (x)} , g(x) \rangle \\
      & > \const > 0.
   \end{aligned}
  \end{equation}

  Therefore, $\vectornorm{\pi_{u} (x)} = 1$ implies $\vectornorm{\pi_{u}(P_{f_{\xi}}(x))} > 1$ for all $\xi \in [0,1]$
  and proves \eqref{cover:1b}.
  By a mirror argument, from~\eqref{eq:fb} and~\eqref{eq:gb} we obtain 
 \begin{equation}\label{hatf:2}
   \begin{aligned}
     \frac{d}{dt} \vectornorm{ \pi_{s}  \varphi^{\xi}(t, x)  ) }_{|_{t=0}} < \const < 0, \quad x: \vectornorm{\pi_{s}(x)} \geq 1,
   \end{aligned}
  \end{equation} 
  hence $\vectornorm{\pi_{s}(P_{f_{\xi}}(x))} = 1$ implies $\vectornorm{\pi_{s} (x)} > 1$ for all $\xi \in [0,1]$.
  This proves \eqref{cover:1c}.

  We are left to prove~\eqref{R=P} and~\eqref{S0=RSet}. 
  Let us start with the latter. 
  
  Observe, that $\hat{f}_{|_{|S|}} = f_{|_{|S|}}$, which proves the ``$\subset$'' inclusion.
  For the other one we proceed as follows.
  Since $\varphi^{0} = \hat{\varphi}$, from~\eqref{hatf:1} and~\eqref{hatf:2}
  we obtain the forward invariance of the sets $\{ x \in \rr^{N}: \  \vectornorm{ \pi_{u}  x } \geq 1 \}$,
  $\{ x \in \rr^{N}: \  \vectornorm{ \pi_{s} ( x ) } \leq 1 \}$ under $\hat{\varphi}$.
  Therefore, for $x \in |X_{S,\text{in}}|$ such that $R(x) \in |X_{S,\text{out}}|$ we have
  \begin{equation}
    \begin{aligned}
      \vectornorm{ \pi_{u}\varphi(t,x) } \leq 1, \ \forall t \geq 0: \pi_{\mu}  \varphi(t,x)  \leq 1, \\
      \vectornorm{ \pi_{s}\varphi(t,x) } \leq 1, \ \forall t \geq 0: \pi_{\mu}  \varphi(t,x)  \geq 0.
    \end{aligned}
  \end{equation}

  As a consequence the part of the trajectory between $x$ and $P_{\hat{f}}(x)$ is wholly contained in $|S|$, where the vector field $\hat{f}$ is equal to $f$,
  hence $x \in S^{0}$. By the same argument~\eqref{S0=RSet} implies~\eqref{R=P}.
\end{proof}

One can also prove a backcovering lemma, which is superfluous in the context of our applications but illustrates
that covering and backcovering occur simultaneously in isolating segments.

\begin{thm}\label{is:2}
  Let $S$ be as in Theorem~\ref{is:1}. Then, there exists a diffeomorphism $\tilde{R}: \Sigma_{\text{in}} \to \Sigma_{\text{out}}$
  such that there is a backcovering relation:
  \begin{equation}
    X_{S, \text{in}} \backcover{\tilde{R}} X_{S, \text{out}}.
  \end{equation}
  Moreover
  \begin{equation}
    \tilde{R}(x) = P(x) \ \forall x \in S^{0},
  \end{equation}
  and it holds that
  \begin{equation}
    S^{0}= \{ x \in |X_{S,\text{in}}|: \tilde{R}(x) \in |X_{S,\text{out}}| \}.  
  \end{equation}
  In particular, for every $x \in |X_{S, \text{in}}|$ such that $\tilde{R}(x) \in |X_{S, \text{out}}|$ the part of the trajectory
  between $x$ and $P(x)=\tilde{R}(x)$ is contained in $|S|$.
\end{thm}
\begin{proof}
  Consider the reversed flow given by $\dot{x} = -f(x)$
  for which the transposed segment $S^{T}$ is an isolating segment (see Proposition~\ref{ts:1}). We have:
  \begin{equation}
    \begin{aligned}
      X_{S^{T}, \text{in}} = X_{S, \text{out}}, \\
      X_{S^{T}, \text{out}} = X_{S, \text{in}}.
    \end{aligned}
  \end{equation}
  Backcovering is now a consequence of applying Theorem~\ref{is:1} to $S^{T}$ and inverting the obtained diffeomorphism.
\end{proof}

\subsubsection{Additional coverings within an isolating segment -- the ``fast-slow switch''.}\label{subsec:switch}

Let us first explain the ideas behind this subsection without formality.
Consider a three-dimensional isolating segment $S$ with one exit and one entry direction
let us write $X_{S,\text{lu}}$, $X_{S,\text{ru}}$ for the two connected components of the exit set $S^{-}$
a ``left exit'' and a ``right exit'' one, respectively.
Each of them lies within a level set given by fixing the exit direction level to $\mp 1$.
They can be equipped with an h-set structure with one exit and one entry direction by
setting the entry direction of the segment as the entry one
and the central direction of the segment as the exit one.

If we now consider the function $E_{S}$, which maps each point of the front face $X_{S,\text{in}}$
to the point of $\partial S$ where the trajectory leaves $S$, then 
its image will give a similar alignment as in Lemma~\ref{covlemma}, see Figures~\ref{segmentCovFig:A} and~\ref{segmentCovFig:C}.
The left/right exit edges of $X_{S,\text{in}}$ remain stationary and coincide with the left exit edges
of $X_{S,\text{lu}}$, $X_{S,\text{ru}}$, so to get an actual covering one needs to constrict the h-sets in the image in the exit direction
by a small factor.

For the two connected components of $S^{+}$ -- the ``left/right entrance'' h-sets $X_{S,\text{ls}}$, $X_{S,\text{rs}}$ one needs
to fix the entry direction height in the segment coordinates so the central direction of the segment
takes its role, while the exit direction of the segment induces the exit direction for the h-set.
Then one can prove similar theorems with backcovering relations, by reversing the vector field.

In the context of the FitzHugh-Nagumo model such relations allow us to describe the passage between
the slow and the fast dynamics where the periodic orbit detaches from the slow manifold and starts following a heteroclinic connection of the fast subsystem.
With an eye on this application we will state the subsequent results for a range of dimension combinations
which allows an easy proof by Lemma~\ref{covlemma}.
We suspect similar theorems hold for all dimension combinations,
and it will be a subject of further studies to formulate adequate proofs.

\begin{defn}
  Let $S$ be a segment with $u(S) = 1$ and $s(S) = s$.
  We define the h-sets:
  \begin{itemize}
    \item $X_{S,\text{lu}} \subset c_{S}^{-1}(\{-1\} \times \rr^{s} \times \rr)$ (\emph{the left exit face}),
    \item $X_{S,\text{ru}} \subset c_{S}^{-1}(\{1\} \times \rr^{s} \times \rr)$ (\emph{the right exit face})
  \end{itemize}
  as follows:
  \begin{itemize}
    \item $u( X_{S,\text{lu}} ) = u( X_{S,\text{ru}} ):= 1$ and $s( X_{S,\text{lu}} ) = s( X_{S,\text{ru}} ):= s$;
    \item we set
      \begin{equation}
        \begin{aligned}
          |X_{S, \text{lu}}| &:= c_{S}^{-1}(\{-1\} \times \rr^{s} \times \rr) \cap |S|, \\
          |X_{S, \text{ru}}| &:= c_{S}^{-1}(\{1\} \times \rr^{s} \times \rr) \cap |S|;
        \end{aligned}
      \end{equation}
    \item we identify $\{ \mp 1\} \times \rr^{s} \times \rr$ with $\rr^{s+1}$
      and then set
       \begin{equation}
        \begin{aligned}
          c_{ X_{S,\text{lu}} } &:= \rho_{u} \circ {c_{S}}_{|_{ c_{S}^{-1}(\{-1\} \times \rr^{s} \times \rr) }}, \\
          c_{ X_{S,\text{ru}} } &:= \rho_{u} \circ {c_{S}}_{|_{ c_{S}^{-1}(\{1\} \times \rr^{s} \times \rr) }},
        \end{aligned}
      \end{equation}
      where $\rho_{u}(p,q,r) = (2r-1, q)$.
  \end{itemize}
\end{defn}

In the above definition the role of $\rho_{u}$ is to change the order of coordinates, as the third center variable in $S$ becomes
an exit variable in $X_{S,\text{lu}}$ and  $X_{S,\text{ru}}$.

\begin{defn}
  Let $S$ be a segment with $u(S)=u$ and $s(S)=1$.
  We define the h-sets:
  \begin{itemize}
    \item $X_{S,\text{ls}} \subset c_{S}^{-1}(\rr^{u} \times \{-1\} \times \rr)$ (\emph{the left entrance face}),
    \item $X_{S,\text{rs}} \subset c_{S}^{-1}(\rr^{u} \times \{1\} \times \rr)$ (\emph{the right entrance face})
  \end{itemize}
  as follows:
  \begin{itemize}
    \item $u( X_{S,\text{lu}} ) = u( X_{S,\text{ru}} ):=u$ and $s( X_{S,\text{ls}} ) = u( X_{S,\text{rs}} ) := 1$;
    \item we set
      \begin{equation}
        \begin{aligned}
          |X_{S, \text{ls}}| &:= c_{S}^{-1}(\rr^{u} \times \{-1\} \times \rr) \cap |S|, \\
          |X_{S, \text{rs}}| &:= c_{S}^{-1}(\rr^{u} \times \{1\} \times \rr) \cap |S|;
        \end{aligned}
      \end{equation}
    \item we identify $\rr^{u} \times \{ \mp 1\} \times \rr$ with $\rr^{u+1}$,
      then set
       \begin{equation}
        \begin{aligned}
          c_{ X_{S,\text{ls}} } &:= \rho_{s} \circ {c_{S}}_{|_{ c_{S}^{-1}(\rr^{u} \times \{-1\} \times \rr) }}, \\
          c_{ X_{S,\text{rs}} } &:= \rho_{s} \circ {c_{S}}_{|_{ c_{S}^{-1}(\rr^{u} \times \{1\} \times \rr) }};
        \end{aligned}
      \end{equation}
      where $\rho_{s}(p,q,r) = (p, 2r-1)$.
  \end{itemize}
\end{defn}

The role of $\rho_{s}$ is to change the center variable in $S$ to an entry variable in the h-sets $X_{S,\text{ls}}$ and $X_{S,\text{rs}}$.  

\begin{defn}
\label{def:hset-constr}
 Let $X$ be an h-set with $u(X)=u$ and $s(X)=s$ and let $\delta>0$. We define:
 \begin{itemize}
   \item the $\delta$-constricted in the exit direction h-set $X^{\delta,\text{uc}}$,
   \item the $\delta$-constricted in the entry direction h-set $X^{\delta,\text{sc}}$,
 \end{itemize}
 by setting:
 \begin{equation}
   \begin{aligned}
     c_{X^{\delta,\text{uc}}} &=  \upsilon_{\text{uc}} \circ c_{X}, \\
     c_{X^{\delta,\text{sc}}} &= \upsilon_{\text{sc}} \circ c_{X}, \\
     u(X^{\delta,\text{uc}}) = s(X^{\delta,\text{sc}}) &= u, \\
     u(X^{\delta,\text{uc}}) = s(X^{\delta,\text{sc}}) &= s;
   \end{aligned}
 \end{equation}
 where $\upsilon_{\text{uc}}, \upsilon_{\text{sc}}: \rr^{u} \times \rr^{s} \to \rr^{u} \times \rr^{s}$ and:
 \begin{equation}
   \begin{aligned}
     \upsilon_{\text{uc}}(p,q) &= ( (1+\delta)p, q ), \\
     \upsilon_{\text{sc}}(p,q) &= ( p, (1+\delta) q ).
   \end{aligned}
 \end{equation}
\end{defn}

Geometrically, $\delta$-constriction shortens the h-set by a factor $1/(1+\delta)$ in the exit/entry direction.
Our notation $\text{uc}$, $\text{sc}$ stands for constricted in the ``unstable''/``stable'' (i.e. exit/entry) direction.

\begin{thm}\label{is:3}
  Let $S$ be an isolating segment between transversal sections $\Sigma_{\text{in}}$ and $\Sigma_{\text{out}}$ with $u(S)=1$ and $s(S)=s$.
  We have the following covering relations:
  \begin{equation}
    \begin{aligned}
      X_{S,\text{in}}  &\cover{E_{S}} X^{\delta, \text{uc}}_{S, \text{lu}}, \\
      X_{S,\text{in}} &\cover{E_{S}} X^{\delta, \text{uc}}_{S, \text{ru}},
    \end{aligned}
  \end{equation}
  for all $\delta > 0$. 
\end{thm}

\begin{proof}
  We will only prove
  $X_{S,\text{in}} \cover{E_{S}}  X^{\delta, \text{uc}}_{S, \text{lu}}$, the other case is analogous.
  The idea of the proof should become immediately clear by looking at Figure~\ref{segmentCovFig:C}.
  We embed the codomain of $E_{S}$ in a folded a folded hyperplane $\Sigma_{S,u}$ 
  consisting of three parts:
  \begin{itemize}
    \item The ``upper part'' $\Sigma_{S,u}^{u} := c_{S}^{-1}\left( \{1\} \times \rr^{s} \times (-\infty, 1] \right)$;
    \item the ``middle part'' $\Sigma_{S,u}^{m} := c_{S}^{-1}\left( [-1,1] \times \rr^{s} \times \{ 1 \} \right)$;
    \item the ``lower part'' $\Sigma_{S,u}^{l} := c_{S}^{-1}\left( \{-1\} \times \rr^{s} \times (-\infty, 1] \right)$.
  \end{itemize}
  It can be regarded as a piecewise smooth section homeomorphic to $\rr^{s+1}$, transversal in the sense that there exist smooth extensions of its smooth pieces
  $\Sigma_{S,u}^{u}$, $\Sigma_{S,u}^{m}$, $\Sigma_{S,u}^{l}$ to manifolds without boundary which are transversal sections for the vector field.
  
  We equip $\Sigma_{S,u}$ with a coordinate system which is given by any homeomorphic extension
  of coordinates given on $\Sigma_{S,u}^{l}$ by ${c_{X_{S,\text{lu}}}}_{|_{\Sigma_{S,u}^{l}}}$ to all $\Sigma_{S,u}$ -- we denote this extension by $c_{\Sigma_{S,u}}$.
  
  The plan is to use Lemma~\ref{covlemma} and prove conditions that give the same topological alignment 
  as needed for a covering relation.

  Recall, that by $E_{S,c}$ we denote the exit map expressed in local coordinates of the h-set $X_{S,\text{in}}$ and the section $\Sigma_{S,u}$.
  In the $\Sigma_{S,u}$ coordinates the support of $X^{\delta, \text{uc}}_{S, \text{lu}}$ is a product of two balls 
  $[\frac{-1}{1+\delta}, \frac{1}{1+\delta}] \times \overline{B_{s}(0,1)}$. To be in formal
  agreement with the definition of the support we would need to stretch out the first ball to $[-1,1]$ but it is clear that assumptions of Lemma~\ref{covlemma}
  are given by geometrical conditions which persist under such rescaling. Therefore we omit this transformation to keep the notation simple.

  By definition of $\Sigma_{S,u}$ we have
  \begin{equation}\label{eq:intersection}
    |X_{S,\text{in}}| \cap \Sigma_{S,u} = X_{S,\text{in}}^{-},
  \end{equation}
  hence ${E_{S}}_{|_{X_{S,\text{in}}^{-}}} = \id_{X_{S,\text{in}}^{-}}$. Coupled with the coordinate system
  we have chosen on $\Sigma_{S,u}$ we get
  \begin{equation}\label{eq:image}
    \begin{aligned}
      E_{S,c}\left(\{-1\} \times \overline{B_{s}(0,1)}\right) &= \{-1\} \times \overline{B_{s}(0,1)}, \\
      \pi_{u}E_{S,c}\left(\{1\} \times \overline{B_{s}(0,1)}\right) &> 1.
    \end{aligned}
  \end{equation}
  This, after the aforementioned rescaling of the exit coordinate, implies Condition (C2) in Lemma~\ref{covlemma}.

  Condition (C1) follows easily. From~\eqref{eq:intersection} and (S3a) we have
  \begin{equation}\label{eq:image2}
    \pi_{s}\left(E_{S,c} ( X_{S,\text{in},c} ) \cap [-1/(1+\delta),1/(1+\delta)] \times \rr^{s}\right) \subset B_{s}(0,1),
  \end{equation}
  since we need non-zero positive time to reach $\Sigma_{S,u}$. 
\end{proof}

If we consider the exit map $E_{S^{T}}$ for the reversed flow $\dot{x}=-f(x)$ in a transposed segment $S^{T}$, we obtain the following theorem.

\begin{thm}\label{is:4}
  Let $S$ be an isolating segment between transversal sections $\Sigma_{\text{in}}$ and $\Sigma_{\text{out}}$ with $u(S)=u$ and $s(S)=1$.
  We have the following covering relations
  \begin{equation}
    \begin{aligned}
      X_{S,\text{out}}  &\backcover{E_{S^{T}}^{-1}} X^{\delta, \text{sc}}_{S, \text{ls}}, \\
      X_{S,\text{out}} &\backcover{E_{S^{T}}^{-1}} X^{\delta, \text{sc}}_{S, \text{rs}}
    \end{aligned}
  \end{equation}
  for all $\delta > 0$.
\end{thm}

\section{Applications}\label{sec:applications}

\subsection{Theorems for periodic orbits}

Let us recall the fundamental theorem motivating the use of covering relations for finding periodic points of sequences of maps.

\begin{thm}[Theorem 9 in \cite{GideaZgliczynski}]\label{cov:sequence}
  Let $X_i, \ i \in \{0, \dots, k\}$ be h-sets with $u(X_0) = \dots = u(X_k)$, $s(X_0)= \dots = s(X_k)$ and set $n=u(X_0)+s(X_0)$.
  Assume that we have the following chain of covering relations:
 \begin{equation}
  X_{0} \longgencover{g_{1},w_1} X_{1} \longgencover{g_{2},w_2} X_{2} \longgencover{g_{3},w_3}  \dots \longgencover{g_{k},w_k} X_{k}.
  \end{equation}
  for some $w_i \in \zz^{*}$. Then there exists a point $x \in \inter X_{0}$ such that
 \begin{equation}
  (g_{i} \circ g_{i-1} \circ \dots \circ g_{1}) (x) \in \inter X_{i}, \ i \in \{ 1, \dots k \}.
 \end{equation}
 Moreover, if $X_{k}=X_{0}$, then $x$ can be chosen so that
 \begin{equation}
  (g_{k} \circ g_{k-1} \circ \dots \circ g_{1}) (x) = x.
 \end{equation}
\end{thm}

We will now state and prove a theorem for finding periodic trajectories of systems given by vector fields,
which allows to use both covering relations and isolating segments.

\begin{thm}\label{thm:1}
  Let $\dot{x}=f(x), \ x \in \rr^{N}$ be given by a smooth vector field. Assume that there exists a sequence of transversal sections $\{ \Sigma_{i} \}_{i=0}^{k},\ k \in \mathbb{N}$
  and a sequence of h-sets
  \begin{equation}
   \mathcal{X} = \{X_{i}: \ |X_{i}| \subset \Sigma_{i}, \ i = 0,\dots, k \},
  \end{equation}
  such that for each two consecutive h-sets $X_{i-1}$, $X_{i} \in \mathcal{X}$ we have one of the following:
  \begin{itemize}
    \item there exists a Poincar\'e map $P_{i} : \Omega_{i-1} \to \Sigma_{i}$ with $\Omega_{i-1} \subset \Sigma_{i-1}$ and an integer $w_i \in \zz^{*}$ such that
        \begin{equation}\label{eq:pgencover}
         X_{i-1} \longgencover{P_{i},w_i} X_{i},
        \end{equation}
      \item there exists an isolating segment $S_{i}$ between $\Sigma_{i-1}$ and $\Sigma_{i}$ such that $X_{S_{i},\text{in}} = X_{i-1}$ and $X_{S_{i},\text{out}} = X_{i}$.
  \end{itemize}
  Then, there exists a solution $x(t)$ of the differential equation passing consecutively through the interiors of all $X_{i}$'s.
  Moreover: 
  \begin{itemize}
    \item whenever $X_{i-1}$ and $X_{i}$ are connected by an isolating segment, the solution passes through $S^{0}_{i}$;
    \item if $X_{0} = X_{k}$ the solution $x(t)$ can be chosen to be periodic.
  \end{itemize}
\end{thm}

\begin{proof}
  By applying Theorem~\ref{is:1} we get a chain of covering relations
  \begin{equation}
    X_{0} \longgencover{g_{1},w_1} X_{1} \longgencover{g_{2},w_2} X_{2} \longgencover{g_{3},w_3} \dots \longgencover{g_{k},w_k} X_{k},
  \end{equation}
  where $g_{i} = P_{i}$ or $g_{i} = R_{i}$, $R_{i}$ being the diffeomorphism given by Theorem~\ref{is:1} associated with the segment $S_{i}$
  (then $w_i=\pm 1$).
  From Theorem~\ref{cov:sequence} there exists a sequence
  $\{ x_{i}: x_{i} \in \inter |X_{i}|, i = 1,\dots, k \}$ such that $g_{i}( x_{i-1} ) = x_{i}$ and we can choose $x_{0}=x_{k}$ whenever $X_{0}=X_{k}$.

  Suppose that for certain $i$'s we have $g_{i} = R_{i}$. Since $x_{i-1} \in |X_{i-1}|$ and $R_{i}(x_{i-1}) = x_{i} \in |X_{i}|$,
  Theorem~\ref{is:1} implies that $x_{i-1} \in S_{i-1}^{0}$ and 
  $R_{i}(x_{i-1}) = P_{i}(x_{i-1})$, $P_{i}:V_{i-1} \to \Sigma_{i}$ being a Poincar\'e map defined on a subset of $\Sigma_{i-1}$.
  This proves that this orbit is an orbit of a full sequence of Poincar\'e maps, hence a real trajectory for the flow.
  Furthermore, it is a periodic trajectory if $x_{0} = x_{k}$ (notice that it cannot be an equilibrium as the vector field on transversal sections
  cannot equal $0$).
\end{proof}

\begin{cor}
  For an isolating segment $S$ the set $S^{0}$ is nonempty.
\end{cor}

\begin{rem}
  Backward covering in conditions like~\eqref{eq:pgencover} can be verified by computation of a Poincar\'e map from $\Sigma_i$
  to $\Sigma_{i-1}$ for the inverse vector field $\dot{x} = -f(x)$, denoted by us with some abuse of notation by $P_i^{-1}$. 
  This is indeed an inverse of some Poincar\'e map $P_i$ for $\dot{x}=f(x)$, provided the domain of $P_i$ is taken to be the image of $P_i^{-1}$.
\end{rem}

We consider Theorem~\ref{thm:1} as a prototypical theorem for application of methods of covering relations and isolating segments.
However, in the context of periodic orbits in 
FitzHugh-Nagumo equations (and fast-slow systems in general) we will use a following three-dimensional modification
that makes use of the fast-slow switch described in Subsection~\ref{subsec:switch}.

\begin{thm}\label{thm:2}
  Let $\dot{x}=f(x), \ x \in \rr^{3}$ be given by a smooth vector field. Assume that there exists a sequence of transversal sections $\{ \Sigma_{i} \}_{i=0}^{k},\ k \in \mathbb{N}$,
  and sequence of h-sets
  \begin{equation}
   \mathcal{X} = \{X_{i}: \ u(X_{i})=s(X_{i})=1, \ i = 0,\dots, k \}
  \end{equation}
  such that for each two consecutive h-sets $X_{i-1}$, $X_{i} \in \mathcal{X}$ we have one of the following:
  \begin{itemize}
    \item $X_{i-1} \subset \Sigma_{i-1}$, $X_{i} \subset \Sigma_{i}$ and there exists a Poincar\'e map
      $P_{i} : \Omega_{i-1} \to \Sigma_{i}$ with $\Omega_{i-1} \subset \Sigma_{i-1}$ and an integer $w_i \in \zz^{*}$ such that
       \begin{equation}
         X_{i-1} \longgencover{P_{i},w_i} X_{i},
       \end{equation}
      \item there exists an isolating segment $S_{i}$ between $\Sigma_{i-1}$ and $\Sigma_{i}$ such that $X_{S_{i},\text{in}} = X_{i-1}$ and $X_{S_{i},\text{out}} = X_{i}$;
      \item there exists an isolating segment $S_{i}$ between $\Sigma_{i-1}$ and $\Sigma_{i}$ such that $X_{S_{i},\text{in}} = X_{i-1}$ and either
        $X_{S_{i},\text{lu}} = X_{i}$ or $X_{S_{i},\text{ru}} = X_{i}$;
      \item there exists an isolating segment $S_{i}$ between $\Sigma_{i-1}$ and $\Sigma_{i}$ such that $X_{S_{i},\text{out}} = X_{i}$ and either
        $X_{S_{i},\text{ls}} = X_{i-1}$ or $X_{S_{i},\text{rs}} = X_{i-1}$.
    \end{itemize}
  Then there exists a solution $x(t)$ of the differential equation passing consecutively through the interiors of all $X_{i}$'s.
  Moreover: 
  \begin{itemize}
    \item whenever $X_{i-1}$ and $X_{i}$ are connected by an isolating segment as its front and rear faces, respectively, the solution passes through $S^{0}_{i}$;
   \item if $X_{0} = X_{k}$ the solution $x(t)$ can be chosen to be periodic.
  \end{itemize}
\end{thm}

\begin{proof}
  First, we replace all the h-sets $X_{i}$ of the form $X_{S_{i},\text{lu}}$,\ $X_{S_{i},\text{ru}}$
  by the constricted versions $X^{\delta_{i},\text{uc}}_{S_{i},\text{lu}}$,\ $X^{\delta_{i},\text{uc}}_{S_{i},\text{ru}}$
  and the h-sets of the form $X_{S_{i},\text{ls}},\ X_{S_{i},\text{rs}}$
  by $X^{\delta_{i},\text{sc}}_{S_{i},\text{ls}}\ X^{\delta_{i},\text{sc}}_{S_{i},\text{rs}}$.
  Let us denote the new h-sets by $\tilde{X}_{i}$.
  The replacement procedure is done one by one. Each time an h-set $X_{i}$ needs to be replaced
  we choose $\delta_{i}>0$ small enough, such that
  \begin{itemize}
    \item[(1.)] any covering relation $X_{i}$ was involved in is preserved for $\tilde{X}_{i}$,
    \item[(2.)] any isolating segment that was built including $X_{i}$ as either the front or the rear face 
      can be reconstructed as an isolating segment $\tilde{S}_{i}$/$\tilde{S}_{i+1}$ with the face $\tilde{X}_{i}$.
  \end{itemize}
  It is intuitively clear that both should hold for a sufficiently small perturbation.
  To show (1.) it is enough to observe that a covering relation 
  is a $C^{0}$-open condition with respect to homeomorphisms defining the h-sets
  and persists after constricting one (or both) h-sets with $\delta$ small enough.
  The proof of such proposition would be almost the same as the proof of Theorem 13 in \cite{GideaZgliczynski}
  stating stability of covering relations under $C^{0}$ perturbations, and therefore we omit it.

  For (2.) the segment $\tilde{S}_{i}$ is constructed so that $c_{S_{i}}$ is $O(\delta_{i})$-close in the $C^{1}$ norm to $c_{\tilde{S}_{i}}$.
  We omit the details; describing the construction by precise formulas would introduce a lot of unnecessary notation.
  It is easy to see that for $\delta_{i}$ small enough the conditions (S1)-(S3) (or their counterparts) will still hold.

  We apply Theorems \ref{is:1}, \ref{is:3}, \ref{is:4} to get and a chain of covering relations
  \begin{equation}
   \tilde{X}_{0} \longgencover{g_{1},w_1} \tilde{X}_{1} \longgencover{g_{2},w_2} \tilde{X}_{2} \longgencover{g_{3},w_3} \dots \longgencover{g_{k},w_k} \tilde{X}_{k},
  \end{equation}
  where for each $g_{i}$ we have one of the following:
  \begin{itemize}
    \item $g_{i} = P_{i}$,
    \item $g_{i} = R_{i}$, $R_{i}$ given by Theorem~\ref{is:1},
    \item $g_{i} = E_{S_{i}}$,
    \item $g_{i} = E_{S^{T}_{i}}^{-1}$.
  \end{itemize}
  From here, the proof continues in the same way as the proof of Theorem~\ref{thm:1}. We obtain a sequence of points
  $\{ x_{i}: x_{i} \in \inter X_{i}, i = 1,\dots, k \}$ such that $g_{i}(x_{i-1}) = x_{i}$ and we can choose $x_{0}=x_{k}$ whenever $X_{0}=X_{k}$.
  By the same argument as in Theorem~\ref{thm:1} the sequence lies on a true trajectory of the flow; the trajectory is periodic if $x_{0}=x_{k}$.
\end{proof}

We note that the formulation of Theorem~\ref{thm:2} is not aimed at full generality. By using only Theorem~\ref{is:3} or \ref{is:4} one can produce similar theorems
when one direction is expanding and arbitrary number of directions are contracting or vice versa.

\subsection{Theorems for connecting orbits}

In this subsection using covering relations and isolating segments we will provide abstract topological theorems
that can be employed for finding homoclinic orbits
(and certain other types of connecting orbits) for maps and differential equations.
Homoclinic loops to a hyperbolic point are in general a codimension one phenomenon,
hence we need to include some kind of shooting from the parameter space in the formulation of our theorems.

We focus on the case, where the unstable manifold of one equilibrium is one-dimensional 
and the stable manifold of the other equilibrium can be multidimensional, say of dimension $s$. 
If $N$ is the dimension of the equation, then, generically, the ``dimension gap'' has to be patched by $N-s$ parameters. 
As an example, for the homoclinic to zero equilibrium in the FitzHugh-Nagumo system~\eqref{FhnOde}, 
we have $N=3$, $s=2$ and we need one parameter -- we will use the wave speed $\theta$.
By reversing the vector field, one can readily apply our theorems to treat the symmetric case of a one-dimensional stable manifold
and a multidimensional unstable manifold.

Certainly, without much effort similar theorems involving isolating segments can be formulated for some other dimension 
combinations, e.g. the ``stable'' codimension zero scenarios of
transverse connecting orbits (such as a connection from a saddle with two-dimensional unstable manifold
to a saddle with two-dimensional stable manifold in a 3D phase space).
However, formulation of such theorems is postponed to later research, once we find good example applications.
 
Throughout this subsection we assume we only work in the $\max$ norm.
This is to facilitate the exposition as in this norm a product of two balls is a ball.
Given an h-set $X$ we will denote by $\pi_{u(X)}: \rr^{u(X)} \times \rr^{s(X)} \to \rr^{u(X)}$
the projection onto first $u(X)$ coordinates and by $\pi_{s(X)}: \rr^{u(X)} \times \rr^{s(X)} \to \rr^{s(X)}$
the projection onto the last $s(X)$ ones.

\begin{defn}
  Let $X,Y$ be h-sets of not necessarily the same dimensions. We define the h-set $X \times Y$ by setting 
  \begin{itemize}
    \item $u(X \times Y) = u(X) + u(Y)$ and $s(X \times Y) = s(X) + s(Y)$,
    \item $|X \times Y| = |X| \times |Y|$,
    \item $c_{X \times Y} = (\pi_{u(X)} \circ c_X, \pi_{u(Y)} \circ c_Y, \pi_{s(X)} \circ c_X, \pi_{s(Y)} \circ c_Y)$. 
   \end{itemize}
\end{defn}

We now extend the definition of the covering relation to cases where the covering set has a lower entry dimension than the set to be covered,
by adding dummy variables to the domain.
It will be used for coverings by an h-set of parameters, which will have zero entry dimension.
We remark that there exists a more general definition of a covering relation, which allows the covering h-set also to have a higher dimension than the h-set to be covered,
introduced by Wilczak in~\cite{WilczakShilnikov} (Definition 2.2).
Our ad-hoc extension is a special case of the definition given by Wilczak.
We chose against using Wilczak's definition throughout all the thesis to avoid reproving certain theorems which we cited, in particular
Lemma~\ref{covlemma}.

\begin{defn}\label{defn:altcover}
 Let $X,Y$ be h-sets with $u(X) = u(Y)$, $s(X) < s(Y)$ and let $n=u(Y)+s(Y)$ and $s=s(X)-s(Y)$. 
 Let $g: |X| \to \rr^{n}$. We say that $X$ $g$-covers $Y$ with degree $w$, and write 
 \begin{equation}
   X \cover{g,w} Y
 \end{equation}
 iff
 \begin{equation}\label{eq:bs01}
 X \times \overline{B_s(0,1)} \cover{\tilde{g},w} Y,
 \end{equation}
 where $\tilde{g}: |X| \times \overline{B_s(0,1)} \to \rr^n$ is given by
 $\tilde{g}(x,y) = g(x)\ \forall x \in |X|, \ y \in \overline{B_s(0,1)}$ and by $\overline{B_s(0,1)}$ in~\eqref{eq:bs01} we denote (with a slight abuse of notation)
 an h-set given by the quadruple $( \overline{B_s(0,1)},0,s,\id)$.
\end{defn}

We now state a basic topological theorem which can be used to find connecting orbits for maps.

\begin{thm}\label{thm:conmap}
  Let $Z$, and $X_0,\dots,X_k$, $k \geq 0$ be h-sets with $u(Z)=u(X_1) = \dots = u(X_k)=u$, $s(X_1)= \dots = s(X_k)=s$, $s(Z)=0$ and set $n:=u+s$.
  We assume the following:
  \begin{itemize}
    \item $b: \overline{B_s(0,1)} \times |Z| \to |X_k|$ is continuous and for each $z \in |Z|$ the map $b(\cdot,z)$ is a vertical disk in $X_k$;
    \item there exists a map $W: |Z| \to \rr^n$ such that $Z$ $W$-covers $X_0$ with degree $w_0 \in \zz^{*}$; 
    \item there is a sequence of maps $g_i: \Omega_i \times |Z| \to \rr^n,\ \Omega_i \subset \rr^n, \ i = 1,\dots,k$ such that for all $i \in 1,\dots, k$ 
      we have
      \begin{equation}
        X_{i-1} \longlonggencover{g_{i}(\cdot, z), w_i} X_{i}, \ \forall z \in |Z|.
      \end{equation}
  \end{itemize}

  Then, there exists a $\bar{z} \in |Z|$ such that
  \begin{equation}
    \begin{aligned}
      W(\bar{z}) &\in |X_0|, \\ 
      (g_i(\cdot, \bar{z}) \circ \dots \circ g_1(\cdot, \bar{z}))( W(\bar{z}) ) &\in |X_i|,\ i \in \{1, \dots, k\},\\
      (g_k(\cdot, \bar{z}) \circ \dots \circ g_1(\cdot, \bar{z}))( W(\bar{z}) ) &\in b\left(\overline{B_s(0,1)}, \bar{z} \right).
    \end{aligned}
  \end{equation}
\end{thm}

Let us now comment on what is the meaning of objects in the statement of this theorem in the context of application to finding connecting orbits of ODEs.
The set $Z$ will be the set of parameters we need to fix to obtain a connection and the role of $W$ will be to assign to a parameter
the point of intersection of a branch of the unstable manifold of one equilibrium with some transversal section.
The maps $g_i$ will be defined as a sequence Poincar\'e maps that will allow to propagate the unstable manifold up to the last section where we will have control
over the stable manifold of the second equilibrium (equal to the first equilibrium for the case of a homoclinic orbit), given by $b$. 
This theorem is a modified version of Theorem 3.3 in~\cite{WilczakShilnikov}, in particular we allow backcoverings
and we allow the vertical disk to vary with the parameters.

The following Theorem on intersection of horizontal and vertical disks will be used as a lemma when proving Theorem~\ref{thm:conmap}:

\begin{thm}[Theorem 3 in~\cite{WZsymmetry}]\label{thm:horver}
  Let $X_i, \ i \in \{0, \dots, k\}$ be h-sets with $u(X_0) = \dots = u(X_k) = u$, $s(X_0)= \dots = s(X_k) = s$.
  Let $b_0$ be a horizontal disk in $X_0$ and $b_e$ be a vertical disk in $X_k$.
  Assume that we have the following chain of covering relations:
 \begin{equation}
  X_{0} \longgencover{g_{1},w_1} X_{1} \longgencover{g_{2},w_2} X_{2} \longgencover{g_{3},w_3}  \dots \longgencover{g_{k},w_k} X_{k},
  \end{equation}
  for some integers $w_i \in \zz^{*}$.
 Then there exists a point $x \in \inter |X_{0}|$ such that
 \begin{equation}
   \begin{aligned}
     x &= b_0(p), \ \text{for some $p \in B_{u}(0,1)$,}\\
     (g_{i} \circ g_{i-1} \circ \dots \circ g_{1}) (x) &\in \inter |X_{i}|, \ i \in \{ 1, \dots k \},\\ 
     (g_{k} \circ g_{k-1} \circ \dots \circ g_{1}) (x) &= b_e(q), \ \text{ for some $q \in B_{s}(0,1)$.}   
  \end{aligned}
 \end{equation}
\end{thm}

\begin{proof}[Proof of Theorem~\ref{thm:conmap}]
  We will apply Theorem~\ref{thm:horver} to a sequence
  of h-sets $Z  \times Z^T$, $X_0 \times Z^T,\  X_1 \times Z^T,\dots, X_k \times Z^T$ and disks $b_0$, $b_e$, which we will define later.
  Observe that these h-sets have dimension $2u+s$ with $u$ exit and $u+s$ entry directions, except for $Z \times Z^T$, which has
  $u$ exit and $s$ entry directions.

  Let $\delta \in [0,1]$. We define $A_\delta: \rr^u \to \rr^u$ to be a map given by 
  \begin{equation}
    A_\delta(z) = c_Z^{-1}\left( \delta c_Z(z) \right). 
  \end{equation}
  The role of $A_\delta$ for $\delta \in [0,1)$ is to generate an artificial covering between the ``parameter'' h-sets $Z^T$.
  Later in the proof we will pass with $\delta$ to $1$. 
  For the purpose of this proof we will overload our notation and denote by $A_\delta$ 
  also the map $(x,z) \rightarrow A_\delta(z), \ x \in \rr^n,\ z \in \rr^u$.

  Assume that for some $i \in \{ 1,\dots ,k\}$ we have a covering relation $X_{i-1} \cover{g_i(\cdot,z)} X_{i}$ for all $z \in |Z|$.
  We will construct a homotopy that establishes a covering relation
  \begin{equation}\label{eq:xicov}
    X_{i-1} \times Z^T \longlongcover{(g_i, A_\delta), w_i} X_i \times Z^T
  \end{equation}
  for all $\delta \in [0,1)$.  
  The required homotopy $\tilde{h}_i :[0,1] \times \overline{B_{u+s}(0,1)} \times \overline{B_{u}(0,1)} \to \rr^{2u+s}$ will be
  a composition of a deformation retraction of $Z_c$ onto $0$ with the homotopy $h_i=h_i(\xi,x,z)|_{z=0}$ from the definition of covering by $g_i(\cdot,z)$.
  The formula for $\tilde{h}_i$ is given by
  \begin{equation}\label{eq:newhomotopy}
    \tilde{h}_i (\xi, x, z) = \begin{cases} \left(g_c\left(x,A_{1-2\xi}(z)\right), (1-2\xi)\delta z \right), &\text{ for $\xi \in[0,1/2]$}, \\  
                                          \left( h_i(2\xi-1,x,0), 0 \right), &\text{ for $\xi \in[1/2,1]$}.
                            \end{cases}
  \end{equation}

  It is clear that~$\tilde{h}_i$ satisfies condition~\eqref{cover:1a}, \eqref{cover:1b}, \eqref{cover:1c} and \eqref{cover:2},
  since $h_i(\cdot,\cdot,z)$ does for all $z \in |Z|$. 

  By the same argument we obtain a covering relation
  \begin{equation}\label{eq:xicov2}
    Z \times Z^T \longcover{W_\delta, w_0} X_0 \times Z^T.
  \end{equation}
  with $W_\delta : Z \times Z^T \to \rr^{2u+s}$ given by $W_\delta (z_1,z_2) = \left( W(z_1), A_\delta (z_1) \right)$.

  Now assume that $\delta \in (0,1)$ and for some $i$ we have a backcovering $X_{i-1} \longbackcover{g_i(\cdot,z),w_i} X_{i}$ for all $z \in |Z|$.
  We now want to verify a backcovering relation 
  \begin{equation}\label{eq:xibackcov}
    X_{i-1} \times Z^T \longlongbackcover{(g_i, A_\delta),w_i} X_i \times Z^T.
  \end{equation}

  In other words we need to construct a homotopy $\tilde{h}_i$ that establishes a covering relation     
  $Z \times X_{i}^T \longlongcover{(A_\frac{1}{\delta}, g_i^{-1} ), w_i} Z \times X_{i-1}^T$.
  Let $h_i$ be a homotopy for the covering relation $X_{i}^T \longlongcover{g^{-1}_i(\cdot,z),w_i} X_{i-1}^T$.
  It can be easily checked that the homotopy given by $\tilde{h}_i(\xi,x,z) = (\frac{1}{\delta}z, h_i(\xi,x))$ satisfies 
  conditions~\eqref{cover:1a}, \eqref{cover:1b}, \eqref{cover:1c},
  and condition~\eqref{cover:2} follows from the product property of the Brouwer degree (A7).

  From the above considerations we obtain a following chain of covering relations
 \begin{equation}
    Z \times Z^T \longcover{W_\delta,w_0} X_0 \times Z^T \longlonggencover{(g_1, A_\delta),w_1} X_1 \times Z^T 
    \longlonggencover{(g_2, A_\delta),w_2}  
    \dots \longlonggencover{(g_k, A_\delta),w_k} X_{k} \times Z^T.
  \end{equation}
  
  We define a horizontal disk $b_0$ in $Z \times Z^T$ by $b_0(p) = \left(c_Z^{-1}(p), c_{Z^T}^{-1}(0)\right), \ p \in \overline{B_u(0,1)}$
  and a map $b_e: \overline{B_s(0,1)} \times \overline{B_u(0,1)} \to X_k \times Z^T$ by $b_e( q_1, q_2 ) = \left( b\left( q_1, c_{Z}^{-1}(q_2) \right), c_{Z^T}^{-1}(q_2) \right)$.
  The map $b_e$ is a vertical disk and the required homotopy $\hat{h}_e$ is given by 
  \begin{equation}
    \hat{h}_e (\xi, q_1, q_2) = \begin{cases} \left( b_c(q_1, (1-2\xi )q_2), q_2 \right), &\text{ for $\xi \in[0,1/2]$}, \\  
      \left( h_{e}(2\xi -1,q_1,0), q_2  \right), &\text{ for $\xi \in[1/2,1]$},
    \end{cases}  \end{equation}
  where $h_e(\cdot,\cdot,0)$ is the homotopy from the definition of the vertical disk $b\left(q_1,c_{Z}^{-1}(0)\right)$.

  By Theorem~\ref{thm:horver}, for all $\delta \in (0,1)$ there exists a solution 
  \begin{equation}
    \mathbf{x}(\delta) = (z_{-1}, z_{0}, \dots , z_{k}, x_0, \dots, x_{k-1}, x_s  )(\delta) 
  \end{equation}
    in $\inter |Z|^{k+1} \times \inter |X_0| \times \dots \times \inter |X_k| \times B_s(0,1) $
  to the following system
  \begin{equation}\label{eq:longsystem}
    \begin{aligned}
      W( z_{-1} ) - x_0 &= 0,\\
      A_\delta( z_{-1} )  - z_0&=0,\\
      g_1( x_0, z_0 ) - x_1&=0,\\
      A_\delta( z_0 ) -z_1&=0,\\
      g_2( x_1, z_1 ) -x_1&=0,\\
      A_\delta( z_1 ) -z_2&=0,\\  
      &\dots \\
      g_k( x_{k-1}, z_{k-1} ) - b(x_s, z_k) &=0,\\
      A_\delta( z_{k-1} ) - z_k&=0.
    \end{aligned}
  \end{equation}

  Denote the right-hand side of~\eqref{eq:longsystem} by $F_{\delta}$. Clearly $F_\delta$ depends continuously on the coefficients,
  hence for $\delta=1$ there exists a solution $\mathbf{\bar{x}} := \mathbf{x}(1) \in |Z|^{k+1} \times |X_0| \times \dots \times |X_k| \times \overline{B_s(0,1)}$ 
  of the following system
  \begin{equation}
    \begin{aligned}
      W( z_{-1} ) - x_0 &= 0,\\
      g_1( x_0, z_0 ) - x_1&=0,\\
      g_2( x_1, z_1 ) -x_1&=0,\\
      &\dots \\
      g_k( x_{k-1}, z_{k-1} ) - b(x_s, z_k) &=0,\\
      z_{-1} &= z_0 = z_1 = \dots = z_k.
    \end{aligned}
  \end{equation}
  Setting $\bar{z}:=z_{-1}$ proves our assertion.
\end{proof}

Our next theorem is designed for connecting orbits of vector fields and allows for use of isolating segments.

\begin{thm}\label{thm:hom1}
  Let 
  \begin{equation}\label{eq:vfield}
    \dot{x}=f(x,z), \ x \in \rr^{N}, \ z \in \rr^u
  \end{equation} 
  be an ODE given by a smooth parameter-dependent vector field $f: \rr^{N+u} \to \rr^N$, $u < N$. 
  Let $Z$ be an h-set in $\rr^u$ such that $u(Z)=u$, $s(Z)=0$. Let $x_1=x_1(z)$, $x_2=x_2(z)$ be (not necessarily distinct) equilibrium points of~\eqref{eq:vfield}.
  Assume that we are given a sequence $\{ \Sigma_{i} \}_{i=0}^{k},\ k \in \mathbb{N}$ of transversal sections for~\eqref{eq:vfield}
  for all $z \in |Z|$ and a family of h-sets
  \begin{equation}
   \mathcal{X} = \{X_{i}: \ |X_{i}| \subset \Sigma_{i}, \ i = 0,\dots, k \},
  \end{equation}
  with $u(X_0) = \dots = u(X_k) = u$ and $s(X_0) = \dots = s(X_k) = s$. 
  Moreover, we assume that the unstable manifold of $x_1$ has an intersection point with $\Sigma_0$, denoted by $W^u_{x_1,\Sigma_0} (z)$, that varies continuously with $z \in |Z|$
  and $Z$ $W^u_{x_1, \Sigma_0}$-covers $X_0$ with degree $w_0 \in \zz^{*}$;
  and that for each $z \in |Z|$ there is a vertical disk $b(\cdot,z)$ in $X_k$, such that all points in the image of $b(\cdot,z)$ belong to 
  the stable manifold of $x_2$, and $b$ is continuous as a map from $\overline{B_s(0,1)} \times |Z|$ to $|X_k|$.

  Suppose that for each two consecutive h-sets $X_{i-1}$, $X_{i} \in \mathcal{X}$ we have one of the following:
  \begin{itemize}
    \item there exists a (parameter-dependent) Poincar\'e map $P_{i}(\cdot,z) : \Omega_{i-1} \times |Z| \to \Sigma_{i}$ with $\Omega_{i-1} \subset \Sigma_{i-1}$ 
      and an integer $w_i \in \zz^{*}$ such that
        \begin{equation}
         X_{i-1} \longlonggencover{P_{i}(\cdot,z),w_i} X_{i} \quad \forall z \in |Z|.
        \end{equation}
      \item there exists a segment $S_{i}$ between $\Sigma_{i-1}$ and $\Sigma_{i}$, such that $X_{S_{i},\text{in}} = X_{i-1}$ and $X_{S_{i},\text{out}} = X_{i}$
        and $S_i$ is an isolating segment for all $z \in |Z|$.
  \end{itemize}
  Then there exists a $\bar{z} \in |Z|$,
  such that the solution $x(t)$ to~\eqref{eq:vfield} with parameter $z$ set to $\bar{z}$ and initial condition $W^u_{x_1,\Sigma_0}(\bar{z})$ satisfies the following
  \begin{itemize}
    \item  $x(t)$ passes consecutively through the supports of all $X_{i}$'s, 
    \item whenever $X_{i-1}$ and $X_{i}$ are connected by an isolating segment as its front and rear faces, respectively, 
       $x(t)$ passes through $S^{0}_{i}$;
    \item the image of $x(\cdot)$ intersects with the image of $b(\cdot,\bar{z})$,
  \end{itemize}
  In other words, $x(t)$ forms a connecting orbit between $x_1$ and $x_2$ and a homoclinic orbit iff $x_1(\bar{z})=x_2(\bar{z})$.
\end{thm}

\begin{proof}
  The proof is almost identical to the proof of Theorem~\ref{thm:1}.
  We apply Theorem~\ref{is:1}, and for each $z \in |Z|$ we obtain a chain of covering relations
  \begin{equation}\label{eq:longchain}
    X_{0} \longlonggencover{g_{1}(\cdot,z),w_1} X_{1} \longlonggencover{g_{2}(\cdot,z),w_2} X_{2} \longlonggencover{g_{3}(\cdot,z),w_3} \dots \longlonggencover{g_{k}(\cdot,z),w_k} X_{k},
  \end{equation}
  where $g_{i}(\cdot,z)= P_{i}(\cdot,z)$ or $g_{i}(\cdot,z)= R_{i}(\cdot,z)$, $R_{i}$ being the diffeomorphism given by Theorem~\ref{is:1} associated with the segment $S_{i}$.

  Observe that each $P_i$ appearing in the chain is continuous on $\Sigma_{i-1} \times |Z|$ 
  as a Poincar\'e map with parameter dependence as we can add the parameter as an additional variable of zero velocity to the vector field;
  same considerations hold for maps $P_i^{-1}$ and $R_i$ (by the proof of Theorem~\ref{is:1} $R_i$'s are Poincar\'e maps of a vector field).
  In particular, the degrees $w_i$ are independent of $z$, by continuity of the Brouwer degree (A3).
  Theorem~\ref{thm:conmap} can be now applied to the chain~\eqref{eq:longchain}, the family of vertical disks $b(\cdot,z)$ and the map $W^u_{x_1,\Sigma_0}$
  (which takes the role of $W$ in assumptions of Theorem~\ref{thm:conmap}).
  
  From assertion of Theorem~\ref{thm:conmap} we obtain a $\bar{z} \in |Z|$ and a sequence
  $\{ x_{i}: x_{i} \in |X_{i}|,\ i = 0,\dots, k \}$ such that 
  \begin{equation}
    \begin{aligned}
      W^u_{x_1,\Sigma_0}(\bar{z}) &= x_0,\\
      g_{i}( x_{i-1}, \bar{z} ) &= x_{i} \quad \text{for } i \in 1,\dots,k,\\
      x_k &\in b(\overline{B_s(0,1)},\bar{z}).
    \end{aligned}
  \end{equation}
  Now we need to show that $x_0,x_1, \dots, x_k$ are consequent points of a true solution to $\dot{x}=f(x,\bar{z})$.
  We will repeat the same argument as in the proof of Theorem~\ref{thm:1}. 
  Suppose that for certain $i$'s we have $g_{i} = R_{i}$. 
  Since $x_{i-1} \in |X_{i-1}|$ and $x_i \in |X_i|$,
  we obtain that $x_{i-1} \in S_{i-1}^0$, hence, by assertion of Theorem~\ref{is:1} the points $x_{i-1}$ and $x_{i}$ indeed belong to a solution of $\dot{x}=f(x,\bar{z})$.
\end{proof}

Similarly as in the case of Theorem~\ref{thm:1} for periodic orbits, Theorem~\ref{thm:hom1} is stated only for future reference.
We will now state and prove the theorem which is applicable to the FitzHugh-Nagumo equation and uses the fast-slow switch (see Subsection~\ref{subsec:switch}).

\begin{thm}\label{thm:hom2}
  Let 
  \begin{equation}\label{eq:vfield2}
    \dot{x}=f(x,z), \ x \in \rr^{3}, \ z \in \rr
  \end{equation} 
  be an ODE given by a smooth parameter-dependent vector field $f: \rr^{3} \times \rr \to \rr^3$. 
  Let $Z$ be an h-set in $\rr$ such that $u(Z)=1$, $s(Z)=0$ and let $x_1=x_1(z)$, $x_2=x_2(z)$ be (not necessarily distinct) equilibrium points of~\eqref{eq:vfield2}.
  Assume that we are given a sequence $\{ \Sigma_{i} \}_{i=0}^{k},\ k \in \mathbb{N}$ of transversal sections for~\eqref{eq:vfield2}
  for all $z \in |Z|$ and that we have a family of h-sets
  \begin{equation}
   \mathcal{X} = \{X_{i}: \ |X_{i}| \subset \Sigma_{i}, \ i = 0,\dots, k \},
  \end{equation}
  with $u(X_0) = \dots = u(X_k) = 1$ and $s(X_0) = \dots = s(X_k) = 1$. 
  Moreover, we assume that the unstable manifold of $x_1$ has an intersection point with $\Sigma_0$, denoted by $W^u_{x_1,\Sigma_0}(z)$, that varies continuously with $z \in |Z|$
  and $Z$ $W^u_{x_1,\Sigma_0}$-covers $X_0$ with degree $w_0 \in \zz^{*}$;
  and that for each $z \in |Z|$ there is a vertical disk $b(\cdot,z)$ in $X_k$, such that all points in the image of $b(\cdot,z)$ belong to 
  the stable manifold of $x_2$, and $b$ is continuous as a map from $\overline{B_s(0,1)} \times |Z|$ to $|X_k|$.

  Suppose that for each two consecutive h-sets $X_{i-1}$, $X_{i} \in \mathcal{X}$ we have one of the following:
  \begin{itemize}
    \item there exists a (parameter-dependent) Poincar\'e map $P_{i}(\cdot,z) : \Omega_{i-1} \times |Z| \to \Sigma_{i}$ with $\Omega_{i-1} \subset \Sigma_{i-1}$ and
        \begin{equation}
         X_{i-1} \longlonggencover{P_{i}(\cdot,z),w_i} X_{i} \quad \forall z \in |Z|.
        \end{equation}
      \item there exists a segment $S_{i}$ between $\Sigma_{i-1}$ and $\Sigma_{i}$, such that $X_{S_{i},\text{in}} = X_{i-1}$ and $X_{S_{i},\text{out}} = X_{i}$
        and $S_i$ is an isolating segment for all $z \in |Z|$.
      \item there exists a segment $S_{i}$ between $\Sigma_{i-1}$ and $\Sigma_{i}$ such that $X_{S_{i},\text{in}} = X_{i-1}$, the segment $S_i$ is an isolating segment for all $z \in |Z|$,
        and either $X_{S_{i},\text{lu}} = X_{i}$ or $X_{S_{i},\text{ru}} = X_{i}$;
      \item there exists a segment $S_{i}$ between $\Sigma_{i-1}$ and $\Sigma_{i}$ such that $X_{S_{i},\text{out}} = X_{i}$, the segment $S_i$ is an isolating segment for all $z \in |Z|$,
        and either $X_{S_{i},\text{ls}} = X_{i-1}$ or $X_{S_{i},\text{rs}} = X_{i-1}$.
  \end{itemize}
  Then there exists a $\bar{z} \in |Z|$ 
  such that the solution $x(t)$ to~\eqref{eq:vfield2} with parameter $z$ set to $\bar{z}$ and initial condition $W^u_{x_1,\Sigma_0}(\bar{z})$ satisfies the following
  \begin{itemize}
    \item  $x(t)$ passes consecutively through the supports of all $X_{i}$'s, 
    \item whenever $X_{i-1}$ and $X_{i}$ are connected by an isolating segment as its front and rear faces, respectively, 
       $x(t)$ passes through $S^{0}_{i}$;
    \item the image of $x(\cdot)$ intersects with the image of $b(\cdot,\bar{z})$,
  \end{itemize}
  In other words, $x(t)$ forms a connecting orbit between $x_1$ and $x_2$ and a homoclinic orbit iff $x_1(\bar{z})=x_2(\bar{z})$.
\end{thm}

\begin{proof}
  We perform the same replacement procedure as in the first part of the proof of Theorem~\ref{thm:2}, 
  but the new isolating segments and covering relations have to be valid for the whole range of parameter $z \in |Z|$.
  This is possible, since $|Z|$ is compact.

  We apply Theorems \ref{is:1}, \ref{is:3}, \ref{is:4} and obtain a chain of covering relations
  \begin{equation}
   \tilde{X}_{0} \longlonggencover{g_{1}(\cdot,z),w_1} \tilde{X}_{1} \longlonggencover{g_{2}(\cdot,z),w_2} \tilde{X}_{2} \longlonggencover{g_{3}(\cdot,z),w_3} 
   \dots \longlonggencover{g_{k}(\cdot,z),w_k} \tilde{X}_{k},\ \forall z \in |Z|,
  \end{equation}
  where for each $g_{i}$ we have one of the following:
  \begin{itemize}
    \item $g_{i} = P_{i}$,
    \item $g_{i} = R_{i}$, with $R_{i}$ given by Theorem~\ref{is:1},
    \item $g_{i} = E_{S_{i}}$,
    \item $g_{i} = E_{S^{T}_{i}}^{-1}$,
  \end{itemize}
  for each $z \in |Z|$.

  From now the proof continues in the same way as the proof of Theorem~\ref{thm:hom1}. 
  We obtain a $\bar{z} \in |Z|$ and a sequence of points $\{ x_{i}: x_{i} \in \inter X_{i}, i = 1,\dots, k \}$ such that $g_{i}(x_{i-1},\bar{z}) = x_{i}$,
  $W^u_{x_1,\Sigma_0} (\bar{z})=x_0$ and $x_k \in b\left(\overline{B_1(0,1)}, \bar{z} \right)$.
  By the same argument as in Theorem~\ref{thm:hom1} the sequence lies on a true trajectory of the flow, hence it forms a connecting orbit.
\end{proof}

Similarly as with Theorem~\ref{thm:2}, the formulation of Theorem~\ref{thm:hom2} is not aimed at full generality
and one can produce similar theorems for one expanding and arbitrary number of contracting directions or vice versa.

\chapter{Traveling waves in the FitzHugh-Nagumo model}\label{chap:waves}

\section{Local estimates for the stable and the unstable manifold of the zero equilibrium}\label{sec:cc}

\noindent

The purpose of this section is to give proofs of existence and 
rigorous local bounds on the position of the stable manifold and the unstable manifold
of the zero equilibrium in the FitzHugh-Nagumo system~\eqref{FhnOde}. 
For $\epsilon > 0$, $\theta>0$ and other parameters given in~\eqref{eq:parameters} this equilibrium is hyperbolic, with one
repelling and two attracting directions. For such parameter choices $x_0 = (0,0,0)$ is a saddle point
of the fast subsystem, hence the second entry direction is spanned by the slow part of the vector field
(more precisely, by the direction tangent to the slow manifold at the equilibrium, as we will observe in Remark~\ref{slowrem}).

The abstract topological-geometric approach to computation of unstable and stable manifolds of equilibria has been 
summarized in Subsections~\ref{sec:conebasic} and~\ref{sec:blockbasic}.
Eventually we want to apply Theorem~\ref{szczelir} and conclude that the unstable manifold $W^u_{B_u}(x_0)$ and the stable manifold $W^s_{B_s}(x_0)$
are horizontal and vertical disks for some isolating blocks $B_u$, $B_s$.
Observe that by this theorem a single isolating block satisfying the cone conditions gives information about both the stable and the unstable manifold.
However, it is more profitable to construct separate blocks for each of them.
For example, for the unstable manifold it is desirable to have the size of the block in the entry direction small, as this gives tighter bounds.
At the same time we would like the block to be wide in the exit direction, so we can propagate the manifold far from the equilibrium point.
An analogous principle applies to the stable manifold block.

For a hyperbolic equilibrium point of an ODE it is reasonable to expect
that verification of assumptions of Theorem~\ref{szczelir} in interval arithmetics will succeed, see Theorem 26 in~\cite{ZgliczynskiMan}.
However, our parameter range of interest is $\epsilon \in (0,\epsilon_0]$ and a direct evaluation 
of cone conditions and isolation inequalities in interval arithmetics for such range will fail, as for the adjacent parameter value $\epsilon=0$
the equilibrium loses hyperbolicity and the stable manifold degenerates to one dimension.
Therefore we need to prepare the equations and factor out $\epsilon$ from the slow part of the vector field before applying interval verification.
For the isolation inequalities we do this in a similar way as for the isolating segments 
-- we divide the slow velocity by $\epsilon$ before checking isolation on the faces given by fixing the slow direction coordinate.
To verify the cone conditions we use an $\epsilon$-dependent cone field, which amplifies the slow velocities to be of magnitude of the fast ones.

We will formulate the assumptions in a semi-general setting.
Consider a fast-slow system given by
\begin{equation}\label{eq:modelsf}
  \begin{aligned}
    \dot{p} &= f(p,q),\\
    \dot{q} &= \epsilon g(p,q)
  \end{aligned}
\end{equation}
where $p = (p_1,p_2) \in \rr^2$, $q \in \rr$, $f$, $g$ are smooth functions of $(p,q)$ and 
$0<\epsilon \ll 1$ is the small parameter. We will denote the vector field by $F_\epsilon=(f,\epsilon g)$.
We do not assume the dependence of $f$, $g$ on $\epsilon$ as it is not the case in the FitzHugh-Nagumo system
and it would obscure the exposition; however the method can certainly be adapted to the $\epsilon$-dependent case, under some additional assumptions.
Let $(0,0,0)$ be an exact or approximate equilibrium of interest of~\eqref{eq:modelsf} -- this 
is not a restriction as we can always move other equilibria to the origin by a translation of the coordinate system.

Suppose the equilibrium appears to have the following properties: 
\begin{itemize}
  \item[(E1)] $(0,0)$ is a hyperbolic equilibrium of the fast subsystem $\dot{p}=f(p,0)$, in particular $\frac{d f}{d p}(0,0)$ has 
    one positive eigenvalue $\lambda_u$ and one negative eigenvalue $\lambda_s$, 
  \item[(E2)] Let $\vec{T}$ be a vector tangent to the slow manifold at $(0,0)$. To focus our attention let us choose $\vec{T}$ 
    in a way that its last coordinate is equal to $1$.
    Then the directional derivative $\nabla_{\vec{T}} g(0,0) := \langle \nabla g(0,0), \vec{T} \rangle$ is negative. 
\end{itemize}
We will argue that the above conditions imply the existence of a two-dimensional stable and one-dimensional unstable manifold
of the zero equilibrium point for $\epsilon \in (0,\epsilon_0]$, $\epsilon_0$ small enough.
We emphasize that we do not actually verify rigorously neither (E1) nor (E2);
for our proofs we only need to check isolation inequalities and cone conditions for blocks around the equilibrium.
However, throughout this section we will argue that if these two conditons hold,
then verification of isolation inequalities and cone conditions
will succeed for $\epsilon>0$ and block sizes small enough. 
As a consequence, if nonrigorous computations suggest that (E1) and (E2) hold, then we expect 
to succeed with our methods.

In further considerations it will be important that the slow direction remains unchanged in the block coordinates.
Therefore we introduce the concept of an \emph{admissible} linear change of coordinates.

\begin{defn}
  Let $A = \{ a_{i,j} \}_{i,j=1}^3$ be a $3 \times 3$ matrix. We say that $A$ is admissible if $\det{A} \neq 0$ and $a_{3,1} = a_{3,2}=0$.
\end{defn}

Let us observe that in isolating segments we employ a similar, affine condition~\eqref{eq:centralform} for changing coordinates in the central direction.

\begin{rem}\label{rem:invA}
  If $A$ is an admissible matrix then $A^{-1}$ is admissible. 
\end{rem}

Our goal is to define an h-set $B$ (in the $\max$ norm), which will form the isolating block with cones, as follows. 
We will set $u(B)=1$, $s(B)=2$ and require that the change of coordinates $c_{B}$ is linear, admissible and brings the linear part of~\eqref{eq:modelsf} with $\epsilon=0$
to an approximate diagonal form. In other words $c_{B}$ will satisfy
\begin{equation}\label{eq:cbA}
  \begin{aligned}
     c_{B} \circ DF_{0}|_{(p,q)=(0,0)} \circ c_{B}^{-1}  &\approx \left[ \begin{array}{ccc} \lambda_u & 0 & 0 \\ 0 & \lambda_s & 0 \\ 0 & 0 & 0 \end{array} \right].
  \end{aligned}
\end{equation}

\begin{rem}\label{slowrem}
To comply with~\eqref{eq:cbA} the columns of $c_{B}^{-1}$ should be chosen as the (approximate) eigenvectors of $DF_{0}|_{(p,q)=(0,0)}$.
The two eigenvectors corresponding to the two non-zero eigenvalues can be formed by adding zeros to the approximate eigenvectors of $Df|_{p=0}$,
which, via Remark~\ref{rem:invA}, implies the admissibility of $c_B$.
Assuming the equilibrium is at $(0,0)$, the third eigenvalue is $0$ and $\vec{T}$ is the corresponding eigenvector.
This follows from taking the derivative $\frac{d}{dq}$ on both sides of the equality 
\begin{equation}
  (F_0 \circ C_0)(q) = 0,
\end{equation} 
where $C_0(q) = (p(q), q)$ is the local parametrization of the slow manifold near the equilibrium, given by the implicit function theorem.
\end{rem}

Let us emphasize at this point that the h-set $B$, which will form an isolating block is independent of $\epsilon$ and will serve for 
the whole range $\epsilon \in (0,\epsilon_0]$. However, the associated cone field will depend on $\epsilon$.

In further considerations, without loss of generality we will assume that $c_{B}$ and $c_{B}^{-1}$ have the form
\begin{equation}\label{eq:cbform}
  \begin{aligned}
    c_{B} &=  \left[ \begin{array}{ccc} \frac{1}{\delta_u} & 0 & 0 \\ 0 & \frac{1}{\delta_s} & 0 \\ 0 & 0 & \frac{1}{\delta_y} \end{array} \right] 
    \circ \left[ \begin{array}{ccc} \multicolumn{3}{c}{  \multirow{2}{*}{ \dots }  } \\ & & \\ 0 & 0 & 1 \end{array} \right].\\
      c_{B}^{-1} &=  \left[ \begin{array}{ccc} \multicolumn{3}{c}{  \multirow{2}{*}{ \dots }  } \\ & & \\ 0 & 0 & 1 \end{array} \right] 
    \circ  \left[ \begin{array}{ccc} \delta_u & 0 & 0 \\ 0 & \delta_s & 0 \\ 0 & 0 & \delta_y \end{array} \right],
    \end{aligned}
\end{equation}
where $\delta_u, \delta_s, \delta_y \in \rr^+$ are some constants which emphasize that we can adjust the size proportions between the directions of the block.

From now on the new coordinates will be denoted by $[x_u,x_s,y]^T = c_B( [p_1,p_2,q]^T )$ 
and we will label the projections onto $x_u,x_s,y$ by $\pi_u, \pi_s,\pi_y$, respectively.

\subsection{Verification of isolation inequalities}\label{subsec:cc1}

\begin{lem}\label{lem:block}
  Let $\epsilon$ be greater than $0$, $B$ be an h-set with $u(B)=1$, $s(B)=2$ and let $c_{B}$ be admissible.
  Assume the following conditions hold
  \begin{align}
    \pm \pi_u (c_B \circ F_{\epsilon} \circ c_B^{-1})\left( \{\pm 1\} \times [-1,1]^2 \right) &> 0\label{block:xu}\\
    \pm \pi_s (c_B \circ F_{\epsilon} \circ c_B^{-1})\left([-1,1] \times \{\pm 1\} \times [-1,1]\right) &< 0\label{block:xs}\\
    \pm (g \circ c_B^{-1}) \left( [-1,1]^2 \times \{\pm 1\} \right) &< 0.\label{block:xmu}
  \end{align}
  Then the h-set $B$ is an isolating block for~\eqref{eq:modelsf}.
\end{lem}

\begin{proof}
  It is immediately clear that conditions~\eqref{block:xu}, \eqref{block:xs} and the following condition
  \begin{equation}
    \pm \pi_y(c_B \circ F_{\epsilon} \circ c_B^{-1})([-1,1]^2 \times \{\pm 1\}) < 0.\label{block:xueq}
  \end{equation}
  are the isolation inequalities (B1), (B2),
  with $x_u$ serving as the exit variable, and $x_s$, $y$ as the entry variables. Inequalities~\eqref{block:xueq}
  are equivalent to~\eqref{block:xmu} since $c_B$ is admissible and $\epsilon$, $\delta_y$ are greater than 0. 
\end{proof}

\begin{rem}
  For the choice of $c_B$ described in Remark~\ref{slowrem} 
  verification of~\eqref{block:xu} and~\eqref{block:xs} will be easy for small ranges of $\epsilon$, 
  as in new coordinates the system is approximately diagonalized. 
  From admissibility of $c_B$ it follows that
  the condition~\eqref{block:xmu} will hold if $\delta_y$, the ratios $\delta_u / \delta_y$, $\delta_s / \delta_y$ and the (range of) parameter $\epsilon$
  are small enough -- then the first order approximation of the left-hand side of 
  the inequalities for $\epsilon_0$ is $\langle \nabla g(0,0),  \delta_y \vec{T} \rangle < 0$, see Condition (E2).
  Observe that inequality~\eqref{block:xmu} is independent of $\epsilon$, 
  unlike~\eqref{block:xueq} which does not hold for $\epsilon = 0$.
  In consequence, all six inequalities~\eqref{block:xu}, \eqref{block:xs} and~\eqref{block:xmu}
  are robust with respect to $\epsilon$ and possible to verify in interval arithmetics by setting $\epsilon$
  to an interval of the form $[0,\epsilon_0]$.
\end{rem}

\subsection{Verification of cone conditions}\label{subsec:cc2}

\begin{lem}\label{lem:con}
  Let $\epsilon$ be greater than $0$, $B$ be an isolating block with $u(B)=1$, $s(B)=2$, $c_{B}$ be an admissible linear map
  and let $Q_{\epsilon}$ be a quadratic form given by $Q_{\epsilon}(x_u,x_s,y) = x_u^2 - x_s^2 - \frac{1}{\epsilon} y^2$.
  Consider the interval matrix
  \begin{equation}\label{eq:jeps}
    J_\epsilon := Q_{\epsilon} [DF_{\epsilon,c}(B_c)] + ( Q_{\epsilon} [DF_{\epsilon,c}(B_c)] )^T,
  \end{equation}
  where $F_{\epsilon,c} = c_{B} \circ F_\epsilon \circ c_{B}^{-1}$, $B_c = [-1,1]^3$.
   
  Assume that all three principal minors of $J_\epsilon$ are positive.
  Then, there exists a unique $x_0 = (p_0,q_0) \in \inter |B|$ 
  such that $F_\epsilon (x_0) = 0$. Moreover, $W^u_B(x_0)$ is a horizontal disk in $B$ satisfying the cone condition 
  and $W^s_B(x_0)$ is a vertical disk in $B$ satisfying the cone condition.
\end{lem}

\begin{proof}
  This is a direct consequence of Theorem~\ref{szczelir}.
\end{proof}

\begin{rem}\label{rem:block}
  The last row of $Q_{\epsilon} [DF_{\epsilon,c}(B_c)]$
  is independent of $\epsilon$, since $c_B$ is admissible and $g$ is independent of $\epsilon$. 
  This observation is crucial during computation of principal minors of $J_\epsilon$ in interval arithmetics,
  as we can evaluate directly the first two rows of $Q_{\epsilon} [DF_{\epsilon,c}(B_c)]$ for $\epsilon \in [0,\epsilon_0]$
  and the last row is replaced by the last row of (for example) $Q_{1} [DF_{1,c}(B_c)]$.
  Therefore, to check the assumptions of Lemma~\ref{lem:con} for a half-open
  range of the form $\epsilon \in (0,\epsilon_0]$, it is enough to perform computations on a closed range $\epsilon = [0,\epsilon_0]$.
\end{rem}

\begin{rem}
  For the choice of admissible $c_B$ described in Remark~\ref{slowrem}, for small $\epsilon$, $\delta_u$, $\delta_s$, $\delta_y$
  the matrix $J_\epsilon$ can be approximated as follows
  \begin{equation}
    J_{\epsilon} \approx \left[ \begin{array}{ccc} 2\lambda_u & 0 & O(\frac{\delta_u}{\delta_y}) \\ 0 & -2\lambda_s & O(\frac{\delta_s}{\delta_y}) 
        \\ O(\frac{\delta_u}{\delta_y}) &  O(\frac{\delta_s}{\delta_y})
      & -2\nabla g(0,0) \cdot \vec{T} \end{array} \right].
  \end{equation}
  Due to (E2) and for small ratios $\delta_u/\delta_y$, $\delta_s/\delta_y$ we therefore expect all three principal minors to be positive.
\end{rem}

\begin{rem}
  To verify the existence of the homoclinic orbit in the FitzHugh-Nagumo system~\eqref{FhnOde} we will perform a ``shooting'' with parameter $\theta$
  procedure, given by Theorem~\ref{thm:hom2}. For that reason, we will need to verify the existence of stable and unstable manifolds in a single block,
  not only for a half-open range of $\epsilon$, but simultaneously for a (small) range $[ \theta_l, \theta_r ]$. 
  To achieve that, we will build a suitable block based on diagonalization of $F_0$ for just one value of $\theta \in [ \theta_l, \theta_r ]$.
  For such block, and the range $[ \theta_l, \theta_r ]$ small enough, we expect conditions~\eqref{block:xs}, \eqref{block:xu}, \eqref{block:xmu} and~\eqref{eq:jeps} to persist,  
  since they are based only on an approximate diagonalization of the system, and given by strong inequalities. 
\end{rem}

\section{Model examples}\label{sec:model}

The purpose of this section is to discuss model examples of fast-slow systems,
which share some qualitative properties with the FitzHugh-Nagumo equations.
In our examples we will carry out a pen-and-paper construction of the chains of covering relations
and isolating segments necessary to prove periodic and homoclinic orbits.

The contents of this section are by no means necessary to prove the main Theorems~\ref{thm:main1}, \ref{thm:main2}, \ref{thm:main3}, \ref{thm:main4};
the proofs of these theorems are computer assisted and described in Section~\ref{sec:implementation}. 
Instead, our goal is to argue that an approach based on Theorem~\ref{thm:2} and Theorem~\ref{thm:hom2} will succeed,
thanks to the analytical properties of the singularly perturbed system.

\subsection{The periodic orbit}\label{sec:covslowfast}

The model example for the periodic orbit is given by a fast-slow ODE
\begin{equation}\label{fastslow}
  \begin{aligned}
    \dot{x} &= f(x,y,\epsilon), \\
    \dot{y} &= \epsilon g(x,y,\epsilon),
  \end{aligned}
\end{equation}
where $x \in \rr^{2}$, $y \in \rr$, $f,g$ are smooth functions of $(x,y,\epsilon)$ and $0 < \epsilon \ll 1$ is the small parameter.
We will also write $x=(x_{1},x_{2})$ to denote the respective fast coordinates. 
We will denote the projections onto $x_1,x_2, y$ by $\pi_{x_1}, \pi_{x_2}, \pi_y$, respectively.

We will use the notion of the unstable/stable manifold of branches of the slow manifold in~\eqref{fastslow} for $\epsilon=0$;
by that we will mean the union of unstable/stable manifolds of the equilibria forming a given branch.

We make the following assumptions:
\begin{itemize}
  \item[(P1)] we have two branches of the slow manifold $\Lambda_{\pm 1}$,
    that coincide with $\{ 0 \} \times \{ 1 \} \times \rr$ and $\{ 0 \} \times \{ -1 \} \times \rr$
    respectively\footnote{One expects that the branches would actually connect with each other, but we bear in mind that this is a model example and 
    the fold points are of no interest to us. }.
    Both are hyperbolic with one expanding and one contracting direction, and the vector field in their neighborhoods $U_{\pm 1}$ is of the following form:
    \begin{align}
      f(x,y,0) &= A_{\pm 1}(y)(x \mp [0,1]^{T} ) + h_{\pm 1}(x,y), \label{P1:1} \\
      \exists \tilde{\epsilon}_{0} >0: \ 0 &< \delta_{\pm 1}  \leq \pm g(x,y,\epsilon) \leq \delta^{-1}_{\pm 1}, \ (x,y) \in U_{\pm 1}, \ \epsilon \in (0, \tilde{\epsilon}_{0}]. \label{P1:2}
    \end{align}
    The functions $A_{\pm 1},h_{\pm 1}$ are assumed to be smooth and to have the following properties
    \begin{align}
      A_{\pm 1}(y) &= \left[ \begin{array}{cc} \lambda_{u,\pm1}(y) & 0 \\ 0 & \lambda_{s,\pm 1}(y) \end{array} \right], \label{P1:3} \\
        -\delta_{\pm 1}^{-1} \leq \lambda_{s,\pm 1}(y) \leq - \delta_{\pm 1} &< 0 < \delta_{\pm 1} \leq \lambda_{u, \pm 1}(y) \leq \delta_{\pm 1}^{-1}, \label{P1:4} \\
        \frac{h_{\pm 1}(x,y)}{ \vectornorm{ x - (0,\pm 1) } } &\xrightarrow{x \to (0,\pm 1)} 0\quad \forall y. \label{P1:5}
    \end{align}
    The values $\delta_{\pm 1} > 0$ are some constant bounds, which in particular do not depend on neither $\epsilon$ nor $y$.

  \item[(P2)] For the parameterized family of the fast subsystems
    we have two parameters $y_{*},y^{*}$, without loss of generality assumed to be equal to $\mp 1$, for which there exists a transversal heteroclinic connection between
    the equilibria $(0,-1)$ and $(0, 1)$ in the first case, and $(0,1)$ and $(0,-1)$ in the second.
    That means: given any two one-dimensional transversal sections $\Sigma_{f, \pm 1}$ for the fast subsystems for $y=\pm 1$
    which have a nonempty, transversal intersection with the heteroclinic orbits, the maps $\Psi_{\pm 1}$ given by
    \begin{equation}
      \Psi_{\pm 1}: y \to W^u_{\pm 1 ,\Sigma_{f, \pm 1} }(y)- W^s_{\mp 1 ,\Sigma_{f, \pm 1}}(y) \in \rr
    \end{equation}
    have zeroes and a non-zero derivative at $y = \pm 1$.

    Here $W^u_{\pm 1, \Sigma}(y)$ and $W^s_{\pm 1, \Sigma}(y)$ denote the first intersections between the appropriate branches\footnote{To not complicate further the notation,
    we make an implicit assumption that only one pair of branches cross in each of the two subsystems and only refer to them.}
    of the unstable/stable manifolds of the equilibria $(0,\pm 1)$ with a given section $\Sigma$ in the section coordinates.

  \item[(P3)] Denote the points $(0,-1,-1)$, $(0,1,-1)$, $(0,1,1)$, $(0,-1,1)$ by $\Gamma_{\alpha}$, $\alpha \in \mathcal{I} = \{ dl,$  $ul, ur, dr \}$\footnote{
      Index letters in $\mathcal{I}$ stand for up/down and left/right and refer to positions of the points in the $(y,x_2)$ plane, see Figure~\ref{schemeFig}.}, 
      respectively and set $\epsilon=0$.
    For each $\alpha \in \mathcal{I}$ there exists a neighborhood $V_{\alpha}$ of $\Gamma_{\alpha}$,
    such that if $\Lambda_{\pm 1} \cap V_{\alpha}$ is the part of the slow manifold contained
    in $V_{\alpha}$, then the part of its unstable manifold contained in $V_{\alpha}$ coincides with the plane $\rr \times \{\pm 1\} \times \rr$,
    and the part of the stable manifold contained in $V_\alpha$ - with the plane $ \{ 0 \} \times \rr \times \rr$. Without loss of generality we can have
    $\bigcup_{\alpha \in \mathcal{I}} V_{\alpha} \subset ( U_{-1} \cup U_{1} )$.
\end{itemize}

Assumptions that provide us with straightened coordinates are used mostly to simplify the exposition.
It is our impression that the Fenichel theory, and in particular the Fenichel normal form around the slow manifold
are well-suited for verifying such conditions, see~\cite{Fenichel, JonesBook}.

\begin{thm}\label{thm:heuristic}
  Under assumptions (P1)-(P3), there exists an $\epsilon_{0}>0$ and
  six sets forming isolating segments for~\eqref{fastslow} for $\epsilon \in (0,\epsilon_0]$:
  \begin{itemize}
    \item $S_{u}$, $S_{d}$ - two ``long'' isolating segments positioned around the branches $\Lambda_{\pm 1}$ of the slow manifold;
    \item $S_{\alpha},\ \alpha \in \mathcal{I}$ - four short ``corner'' isolating segments, each containing the respective point $\Gamma_{\alpha}$;
  \end{itemize}
  along with the associated transversal sections of the form $\Sigma_{S_{*},\text{in}},\Sigma_{S_{*},\text{out}}$, with
  \begin{equation}
    \begin{aligned}
      u(S_{dl}) &= s(S_{dl}) = u(S_{dr}) = s(S_{dr})
      \\ &= u(S_{ul})=s(S_{ul})= u(S_{ur})=s(S_{ur})
      \\ &=u(S_{u}) = s(S_{u})=u(S_{d})=s(S_{d})
      \\ &= 1.
    \end{aligned}
  \end{equation}
  Moreover, for the h-sets defined by isolating segments we have
  \begin{align}
   X_{S_{u},\text{out}} &= X_{S_{ur},\text{in}}, \label{eq:UR} \\
   X_{S_{dr},\text{out}} &= X_{S_{d},\text{in}}, \\
   X_{S_{d},\text{out}} &= X_{S_{dl},\text{in}}, \\
   X_{S_{ul},\text{out}} &= X_{S_{u},\text{in}}, \label{eq:UL}
  \end{align}
  and the collection
  \begin{equation}
    \begin{aligned}
      \mathcal{X}_{\text{FHN},\text{P}} &= \\
      \{ &X_{S_{u},\text{in}},\ X_{S_{u},\text{out}},\ X_{S_{ur},\text{lu}},
      \ X_{S_{dr},\text{rs}}, \\
      &X_{S_{d}, \text{in}},\ X_{S_{d},\text{out}}, \ X_{S_{dl},\text{ru}},
      \ X_{S_{ul},\text{ls}},\ X_{S_{u}, \text{in}}
      \}
    \end{aligned}
  \end{equation}
  satisfies assumptions of Theorem~\ref{thm:2} for $\epsilon \in (0,\epsilon_{0}]$.
  In particular we have the following covering relations among the h-sets
  not connected by an isolating segment:
  \begin{align}
     X_{S_{dl},\text{ru}} &\cover{P_{L}}  X_{S_{ul},\text{ls}}, \\
     X_{S_{ur},\text{lu}} &\cover{P_{R}}  X_{S_{dr},\text{rs}},
  \end{align}
  where $P_{*}$ are Poincar\'e maps between the respective h-sets and transversal sections containing the next h-set.
  
  As a consequence there exists a periodic solution of the system for these parameter values.
\end{thm}

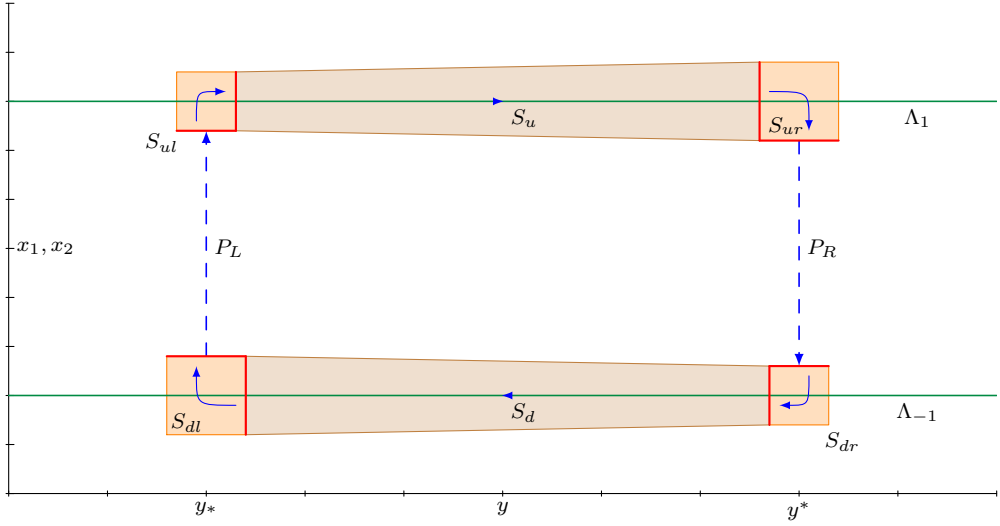
\begin{figure}
  \centering
  \begin{tikzpicture}[line cap=round,line join=round,>=latex,x=13mm,y=6.5mm]
 
  \draw[color=black] (-0.,0.) -- (10,0.);
  \foreach \x in {0,...,10}
  \draw[shift={(\x,0)},color=black] (0pt,1pt) -- (0pt,-1pt) node[below] {};

  \draw[color=black] (0.,0.) -- (0.,10);
  \foreach \y in {0,...,10}
  \draw[shift={(0,\y)},color=black] (2pt,0pt) -- (-1pt,0pt) node[left] {};
 

  \clip(0,-1) rectangle (10,10);

  \filldraw[draw=orange, fill=orange!25] (1.7,7.4) -- (1.7,8.6) -- (2.3,8.6) -- (2.3,7.4) -- (1.7,7.4);
  \filldraw[draw=orange, fill=orange!25] (1.6,1.2) -- (1.6,2.8) -- (2.4,2.8) -- (2.4,1.2) -- (1.6,1.2);
  \filldraw[draw=orange, fill=orange!25] (7.7,1.4) -- (7.7,2.6) -- (8.3,2.6) -- (8.3,1.4) -- (7.7,1.4);
  \filldraw[draw=orange, fill=orange!25] (7.6,7.2) -- (7.6,8.8) -- (8.4,8.8) -- (8.4,7.2) -- (7.6,7.2);

  \filldraw[draw=brown, fill=brown!25] (7.7,1.4) -- (7.7,2.6) -- (2.4,2.8) -- (2.4,1.2) -- (7.7,1.4);
  \filldraw[draw=brown, fill=brown!25] (7.6,7.2) -- (7.6,8.8) -- (2.3,8.6) -- (2.3,7.4) -- (7.6,7.2);

  \draw [color=darkgreen,semithick,domain=0.:10.] plot(\x,{8});
  \draw [color=darkgreen,semithick,domain=0.:10.] plot(\x,{2});

  \draw [->,semithick,dash pattern=on 5pt off 5pt,color=blue] (2.,2.8) -- (2.,7.4);
  \draw [->,semithick,dash pattern=on 5pt off 5pt,color=blue] (8.,7.2) -- (8.,2.6);

  \draw [->,color=blue] (4.99,8.) -- (5.01,8.);
  \draw [->,color=blue] (5.01,2.) -- (4.99,2.);
 
  \draw [color=red,thick] (1.6,2.8) -- (2.4,2.8);
  \draw [color=red,thick] (7.7,2.6) -- (8.3,2.6);
  \draw [color=red,thick] (2.4,2.8) -- (2.4,1.2);
  \draw [color=red,thick] (7.7,2.6) -- (7.7,1.4);
 
  \draw [color=red,thick] (7.6,8.8) -- (7.6,7.2);
  \draw [color=red,thick] (7.6,7.2) -- (8.4,7.2);
  \draw [color=red,thick] (2.3,7.4) -- (2.3,8.6);
  \draw [color=red,thick] (2.3,7.4) -- (1.7,7.4);
 
  \draw [color=blue,<-] (2.3-0.1,8.+0.2) ..controls(2.-0.1,8.+0.2).. (2.-0.1,7.4+0.2);
  \draw [color=blue,<-] (8.4-0.3,8.-0.6) ..controls(8.4-0.3,8.8-0.6).. (8.-0.3,8.8-0.6);
 
  \draw [color=blue,->] (8.+0.1,2.6-0.2) ..controls(8.+0.1,2.-0.2).. (7.7+0.1,2.-0.2); 
  \draw [color=blue,<-] (1.6+0.3,2.+0.6) ..controls(1.6+0.3,1.2+0.6).. (2.+0.3,1.2+0.6);

  \begin{scriptsize}
    \draw (9.2,8.) node[anchor=north] {$\Lambda_{1}$};
    \draw (9.2,2.) node[anchor=north] {$\Lambda_{-1}$};
 
    \draw (2.,5.) node[anchor=west] {$P_{L}$};
    \draw (8.,5.) node[anchor=west] {$P_{R}$};
  
    \draw (5.,8.) node[anchor=north west] {$S_{u}$};
    \draw (5.,2.) node[anchor=north west] {$S_{d}$};   

    \draw (8.14,7.84) node[anchor=north east] {$S_{ur}$};
    \draw (1.3,7.5) node[anchor=north west] {$S_{ul}$};     
    \draw (1.56,1.09) node[anchor=south west] {$S_{dl}$};
    \draw (8.7,0.7) node[anchor=south east] {$S_{dr}$};     
  
    \draw (2.,0.) node[anchor=north] {$y_{*}$};
    \draw (8.,0.) node[anchor=north] {$y^{*}$};

    \draw (5.,0.) node[anchor=north] {$y$};
    \draw (0.,5.) node[anchor=west] {$x_{1}, x_{2}$};
  \end{scriptsize}

\end{tikzpicture}
  \caption{Isolating segments and Poincar\'e maps in the model example for the periodic orbit, the sequence of h-sets plotted in red.}\label{schemeFig}
\end{figure}

The conclusion of the theorem is portrayed in Figure~\ref{schemeFig}.
We break the proof into two parts, first we prove the existence of the corner isolating segments and coverings as a separate lemma.

\begin{lem}\label{lem:heuristic1}
  Consider the system~\eqref{fastslow}.
  For $\epsilon \in (0,\bar{\epsilon}_{0}]$, $\bar{\epsilon}_{0}>0$ small
  there exist two transversal sections of the form
  \begin{equation}
    \begin{aligned}
      \Sigma_{L} &:= \{ (x_{1},x_{2},y): \ x_{2} = 1-\varepsilon_{L} \} \cap \tilde{V}_{ul} \subset V_{ul}, \\
      \Sigma_{R} &:= \{ (x_{1},x_{2},y): \ x_{2} = -1+\varepsilon_{R} \} \cap \tilde{V}_{dr} \subset V_{dr},
    \end{aligned}
  \end{equation}
  $\tilde{V}_{ul}, \tilde{V}_{dr}$ being neighborhoods of $\Gamma_{ul}$ and $\Gamma_{dr}$
  and four isolating segments $S_{dl}$, $S_{ul}$, $S_{ur}$, $S_{dr}$ as specified in Theorem~\ref{thm:heuristic}
  such that
  \begin{equation}
    \begin{aligned}
      |X_{S_{ul},\text{ls}}| &\subset \Sigma_{L}, \\
      |X_{S_{dr},\text{rs}}| &\subset \Sigma_{R},
    \end{aligned}
  \end{equation}
  and there are coverings
  \begin{align}
    X_{S_{dl},\text{ru}} &\cover{P_{L}}  X_{S_{ul},\text{ls}}, \label{eq:PLcover} \\
    X_{S_{ur},\text{lu}} &\cover{P_{R}}  X_{S_{dr},\text{rs}}.
  \end{align}
  Moreover, the sections $\Sigma_{*}$ and the segments $S_{*}$ are $\epsilon$-independent, and 
  given a maximal diameter $\diam_{\text{max}}>0$ they can be chosen
  so that
  \begin{equation}\label{diam}
   \diam(S_{*}) < \diam_{\text{max}}.
  \end{equation}
\end{lem}

\begin{proof}
  We focus first on the ``left'' part of the picture, since all arguments for the ``right'' part are symmetric and independent.
  Without loss of generality we can assume the crossing of the unstable and stable manifolds near the point
  $\Gamma_{ul}$ occurs for $x_{2}-1$ negative and take $\varepsilon_{L}> 0$. For $\varepsilon_{L}$
  and $\epsilon$ small enough, condition (P1) implies that the linear part of the vector field dominates the higher order terms $h_{\pm 1}$, so after having
  set a sufficiently small neighborhood $\tilde{V}_{ul}$ the section $\Sigma_{L}$ is transversal.

  The construction of the isolating segments $S_{ul}, S_{dl}$ is also enabled by (P1).
  Because we already work in straightened coordinates, their supports can be chosen to be of the form:
  \begin{equation}
    \begin{aligned}
      |S_{dl}| &= [-\varepsilon_{L}, \varepsilon_{L}] \times [-1-\delta_{s, dl}, -1+\delta_{s,dl} ] \times [ -1-\delta_{u,dl}, -1+\delta_{u,dl}], \\
      |S_{ul}| &= [-\delta_{u, ul}, \delta_{u,ul} ] \times [1-\varepsilon_{L}, 1+\varepsilon_{L}] \times [ -1-\delta_{s,ul}, -1+\delta_{s,ul}], \\
      u(S_{dl}) &= s(S_{dl}) = u(S_{ul}) = s(S_{dl}) = 1,
    \end{aligned}
  \end{equation}
  where the constants $\delta_{s,dl}, \delta_{u,dl}, \delta_{s,ul}, \delta_{u,ul}$ will be fixed later in the proof.
  The changes of coordinates $c_{S_{dl}}, c_{S_{ul}}$ are defined as a translation of the cuboids to the origin of the coordinate system
  composed with rescaling to $[-1,1]^{2} \times [0,1]$. We label the first coordinate as exit, second as entry, third central.
  Again, if $\varepsilon_{L}$, $\delta_{s,dl}$, $\delta_{u,dl}$, $\delta_{u,ul}$, $\delta_{s,ul}$ are small,
  then the linear part of the vector field dominates the nonlinear part and conditions (S2b), (S3b) are satisfied for $\epsilon=0$
  and for $\epsilon>0$ small.
  Since our change of coordinates is of the form as in~\eqref{eq:centralform}, for $\epsilon>0$ small (S1a) follows from the inequalities~\eqref{P1:2}.

  We can now move on to proving the covering relation~\eqref{eq:PLcover}.
  The supports of the h-sets $X_{S_{dl},\text{ru}}$, $X_{S_{ul}, \text{ls}}$ are of the form:
  \begin{equation}
    \begin{aligned}
      |X_{S_{dl},\text{ru}}| &= \{ \varepsilon_{L} \} \times [-1-\delta_{s, dl}, -1+\delta_{s,dl} ] \times [ -1-\delta_{u,dl}, -1+\delta_{u,dl}], \\
      |X_{S_{ul}, \text{ls}}| &= [-\delta_{u, ul}, \delta_{u,ul} ] \times \{ 1-\varepsilon_{L} \} \times [ -1-\delta_{s,ul}, -1+\delta_{s,ul}].
    \end{aligned}
  \end{equation}
  In $X_{S_{dl},\text{ru}}$ the $x_{2}$ variable takes the role of the entry variable and $y$ takes the role of the exit one;
  in $X_{S_{ul}, \text{ls}}$ the variable $x_{1}$ is exit and $y$ is entry.

  Since covering relations are robust with respect to perturbations of the vector field (see Theorem 13 in \cite{GideaZgliczynski})
  it is enough to show them for $\epsilon=0$.
  From (P2) and (P3) we know that
  \begin{equation}\label{P2:mon1}
    \begin{aligned}
      P_{L}(\varepsilon_{L}, -1, -1) &= (0, 1-\varepsilon_{L}, -1), \\
      \frac{d}{dy} \pi_{x_1} P_{L}(\varepsilon_{L},-1,-1) &\neq 0,
    \end{aligned}
  \end{equation}
  and without loss of generality let us assume that $\pi_{x_1} P_{L}$ is increasing in the neighborhood of the point $(\varepsilon_{L}, -1, -1)$.
  That is already enough to generate two h-sets with a covering relation between them.
  The procedure is as follows:
  \begin{itemize}
    \item fix some $\delta_{s,ul}>0$.
    \item To comply with the covering condition (C1) from Lemma~\ref{covlemma} choose $\delta_{s,dl}>0$ and $\delta_{u,dl}>0$ so that
      \begin{equation}
        \pi_{y}P_{L}(|X_{S_{ul},\text{ls}}|) \subset (-1 -\delta_{s,ul}, -1 + \delta_{s,ul}).
      \end{equation}
      Now, provided $\delta_{s,dl}$ and $\delta_{u,dl}$ were chosen small enough, from~\eqref{P2:mon1}
      there exists $\varepsilon_{ul}>0$ such that
      \begin{equation}
        \begin{aligned}
          \pi_{x_{1}} P_{L}\left( \{ \varepsilon_{L} \} \times  [-1-\delta_{s, dl}, -1+\delta_{s,dl} ] \times \{ -1 -\delta_{u,dl} \} \right) &< \varepsilon_{ul} < 0,  \\
          \pi_{x_{1}} P_{L}\left( \{ \varepsilon_{L} \} \times  [-1-\delta_{s, dl}, -1+\delta_{s,dl} ] \times \{ -1 +\delta_{u,dl} \} \right) &> \varepsilon_{ul} > 0.
        \end{aligned}
      \end{equation}
    \item To fulfill (C2) it is enough to choose $\delta_{u,ul} \leq \varepsilon_{ul}$.
  \end{itemize}
  It is clear that we can choose $\varepsilon_{L}$ small enough and then perform the procedure above with $\delta$'s small
  in a way, that the diameter constriction~\eqref{diam} is satisfied.

  The same procedure is repeated for the isolating segments $S_{ur}$, $S_{dr}$; we will only introduce the notation for these segments,
  as they will be used later in the main part of the proof of Theorem~\ref{thm:heuristic}. Similarly to the left side segments,
  we define them by giving the cuboid supports
  \begin{equation}
    \begin{aligned}
      |S_{ur}| &= [-\varepsilon_{R}, \varepsilon_{R}] \times [1-\delta_{s, ur}, 1+\delta_{s,ur} ] \times [ 1-\delta_{u,ur}, 1+\delta_{u,ur}], \\
      |S_{dr}| &= [-\delta_{u, dr}, \delta_{u,dr} ] \times [-1-\varepsilon_{R}, -1+\varepsilon_{R}] \times [ 1-\delta_{s,dr}, 1+\delta_{s,dr}], \\
      u(S_{ur}) &= s(S_{ur}) = u(S_{dr}) = s(S_{dr}) = 1,
    \end{aligned}
  \end{equation}
  and the coordinate changes $c_{S_{ur}}$, $c_{S_{dr}}$ are again simple translations and rescalings to $[-1,1]^{2} \times [0,1]$, so the first
  variable in the supports is the exit one and the second is entry.

  The supports of h-sets of interest $X_{S_{ur},\text{lu}}$, $X_{S_{ur}, \text{rs}}$ are as follows:
  \begin{equation}
    \begin{aligned}
      |X_{S_{ur},\text{lu}}| &= \{ -\varepsilon_{R} \} \times [1-\delta_{s, ur}, 1+\delta_{s,ur} ] \times [ 1-\delta_{u,ur}, 1+\delta_{u,ur}], \\
      |X_{S_{dr}, \text{rs}}| &= [-\delta_{u, dr}, \delta_{u,dr} ] \times \{ -1+\varepsilon_{R} \} \times [ -1-\delta_{s,dr}, -1+\delta_{s,dr}].
    \end{aligned}
  \end{equation}
  We will not go into details of determining $\delta_{u, ur}$, $\delta_{s,ur}$, $\delta_{u,dr}$, $\delta_{s,dr}$ and $\varepsilon_{R}$
  - the procedure is exactly the same as for the left side segments. The variable $y$ is the exit variable and $x_{2}$ is the entry variable in $X_{S_{ur},\text{lu}}$;
  as for $X_{S_{dr},\text{rs}}$, $x_{1}$ is the exit one and $y$ is entry.

  By taking  the minimum of all upper bounds on $\epsilon$'s throughout this lemma we obtain $\bar{\epsilon}_{0}$ and the proof is complete.
\end{proof}

We can now return to proving Theorem~\ref{thm:heuristic}. We import all the notation from the proof of the Lemma~\ref{lem:heuristic1}
and in particular assume that the isolating segments $S_{dl},S_{ul},S_{ur},S_{dr}$ and the respective h-sets can be chosen to be of the form given therein.

\begin{proof}[Proof of Theorem~\ref{thm:heuristic}]
  From Lemma~\ref{lem:heuristic1} for any given maximal corner segment diameter $\diam_{\text{max}}>0$
  we obtain a bound $\bar{\epsilon}_{0}$ on $\epsilon$'s and four isolating segments $S_{dl}, S_{ul}, S_{ur}, S_{dr}$ containing the respective points $\Gamma_{\alpha}$
  with covering relations between their respective faces.
  We set $\diam_{\text{max}}$ small enough to have
  \begin{align}
    f(x,y,0) \approx A_{1}(y)(x - [0,1]^{T} ),\ (x,y) \in \conv(|S_{ul}| \cup |S_{ur}|), \label{eq:ULin} \\
    f(x,y,0) \approx A_{-1}(y)(x + [0,1]^{T} ),\ (x,y) \in \conv(|S_{dl}| \cup |S_{dr}|),
  \end{align}
  so the higher order terms $h$ can be assumed negligible when checking the isolation inequalities in these neighborhoods.

  Given our four corner isolating segments we are left with construction of
  two isolating segments $S_{u}$ and $S_{d}$ which connect the pairs $S_{ul}$, $S_{ur}$ and $S_{dr}$, $S_{dl}$ respectively.
  We will only construct $S_{u}$, the case of $S_{d}$ is analogous. The strategy is to first connect the pairs by segments,
  then, if necessary, decrease $\bar{\epsilon}_{0}$ to some smaller $\epsilon_{0}$ to obtain isolation.

  We introduce the following notation for rectangular sets around the upper branch of the slow manifold:
  \begin{equation}
    L_{u}(\delta_{u},\delta_{s},y):= [-\delta_{u},\delta_{u}] \times [1-\delta_{s},1+\delta_{s}] \times \{ y \}.
  \end{equation}

  We set
  \begin{equation}
    \begin{aligned}
      a_{u} &:= -1 + \delta_{s,ul}, \\
      b_{u} &:= 1 - \delta_{u,ur},
    \end{aligned}
  \end{equation}
  and we can assume that $a_{u}<b_{u}$.
  Now, we can define $S_{u}$ as a cuboid stretching from $X_{S_{ul,\text{out}}}$ to $X_{S_{ur,\text{in}}}$ as follows.
  For the support we put
  \begin{equation}\label{eq:SuForm}
    \begin{aligned}
    |S_{u}| &:=  \\
    &\bigcup_{\xi \in [0,1]} L_{u}\left(
                         (1-\xi)\delta_{u,ul} + \xi \delta_{u,ur},
                         (1-\xi)\delta_{s,ul} + \xi \delta_{s,ur},
                         (1-\xi)a_{u} + \xi b_{u}
                         \right).
   \end{aligned}
  \end{equation}
  We also set $u(S_{u})=s(S_{u}):=1$.
  There is no need for description of $c_{S_{u}}$ by precise formulas, so we only mention that it is a composition of
  \begin{itemize}
    \item a diffeomorphism which rescales each fiber $L_{u}(\cdot,y)$, given by fixing $y \in [a_{u},b_{u}]$, to $[-1,1] \times [-1,1]$,
    \item a rescaling in the central, $y$ direction from $[a_{u}, b_{u}]$ to $[a_{u}, a_{u}+1]$,
    \item a translation to the origin of the coordinate system.
  \end{itemize}
  As with the corner segments, $x_{1}$ is labeled as the exit direction, $x_{2}$ as entry, and $y$ as the central direction.
  Then one sees that equalities~\eqref{eq:UR} and~\eqref{eq:UL} are true.
  Condition (S1a) is a consequence of inequalities~\eqref{P1:2} for small $\epsilon$ , as the change of variables $c_{S_{u}}$
  in the central direction takes the form~\eqref{eq:centralform}. The upper bound for $\epsilon$'s
  given by $\bar{\epsilon}_{0}$ may need to be decreased at this step.

  It remains to check (S2b) and (S3b) and for that purpose we may need to further reduce $\bar{\epsilon}_{0}$.
  Normals to $S_{u}^{-}$ pointing outward of $|S_{u}|$ are given by
  \begin{equation}\label{eq:n+}
    n_{-}(x,y) = \left(\sgn x_{1}, 0, -\frac{\delta_{u,ur} - \delta_{u,ul}}{b_{u}-a_{u}} \right).
  \end{equation}
  From~\eqref{eq:ULin}, \eqref{eq:n+} and \eqref{P1:3}, \eqref{P1:4} for $(x,y) \in S_{u}^{-}$ we have
  \begin{equation}
    \begin{aligned}
    \langle (f,\epsilon g), n_{-} \rangle (x,y,\epsilon) &\approx \lambda_{u,1}(y) |x_{1}|  - \epsilon g(x,y,\epsilon) \frac{\delta_{u,ur} - \delta_{u,ul}}{b_{u}-a_{u}}  \\
    &> \delta_{1} |x_{1}| - \frac{\epsilon}{\delta_{1}} \frac{|\delta_{u,ur} - \delta_{u,ul}|}{b_{u}-a_{u}}
    \end{aligned}
  \end{equation}
  and the right-hand side is greater than 0 for $\epsilon \in (0, \bar{\epsilon}_{0}]$, $\bar{\epsilon}_{0}$ small enough, see Figure~\ref{slopeFig}.
  This proves (S2b).

  \begin{figure}
    \centering
    \begin{tikzpicture}[line cap=round,line join=round,>=latex,x=2.5mm,y=2.5mm]
 
  \draw[color=black] (-0.,0.) -- (20,0.);
  \foreach \x in {0,...,10}
  \draw[shift={(2*\x,0)},color=black] (0pt,1pt) -- (0pt,-1pt) node[below] {};

  \draw[color=black] (0.,0.) -- (0.,10);
  \foreach \y in {0,...,5}
  \draw[shift={(0,2*\y)},color=black] (2pt,0pt) -- (-1pt,0pt) node[left] {};
 
  \clip(0,-2) rectangle (40,12);

\begin{scope}[spy using outlines=
      {magnification=4, size=90pt, rounded corners, connect spies}]

  \fill [color=brown!25] (2.,4.) -- (2.,6.) -- (18.,8.) -- (18.,2.);
  \draw [color=brown] (2.,4.) -- (18.,2.);
  \draw [color=brown] (2.,6.) -- (18.,8.);
  \draw [color=brown] (2.,4.) -- (2.,6.);
  \draw [color=brown] (18.,8.) -- (18.,2.);

  \draw [->,color=brown] (10.,7.) -- (11.0,9.2);
  \draw [->,color=brown] (6.,6.5) -- (7.0,8.7);
  \draw [->,color=brown] (14.,7.5) -- (15.0,9.7);
 
  \draw [->,color=brown] (10.,3.) -- (11.0,0.8);
  \draw [->,color=brown] (6.,3.5) -- (7.0,1.3);
  \draw [->,color=brown] (14.,2.5) -- (15.0,0.3);
 
  \draw [color=darkgreen,semithick,domain=0.:20.] plot(\x,{5});
  \draw [->,color=brown] (9.99,5.) -- (10.01,5.);

  \spy [black] on (10.5,7.8) in node at (30.,5.);

\end{scope}

  \draw [->,color=black] ( 30.-10.5*4.+10.*4., 5.-7.8*4.+7.*4) -- ( 30.-10.5*4.+10.*4. , 5.-7.8*4.+9.2*4);
  \draw [->,color=black] ( 30.-10.5*4.+10.*4., 5.-7.8*4.+7.*4) -- ( 30.-10.5*4.+11.*4. , 5.-7.8*4.+7.*4);
 
  \draw [color=black,dashed] ( 30.-10.5*4.+10.*4., 5.-7.8*4.+9.2*4) -- ( 30.-10.5*4.+11.*4. , 5.-7.8*4.+9.2*4);
  \draw [color=black,dashed] ( 30.-10.5*4.+11.*4., 5.-7.8*4.+7.*4) -- ( 30.-10.5*4.+11.*4. , 5.-7.8*4.+9.2*4);

  \draw ( 30.-10.5*4.+11.*4. , 5.-7.8*4.+7.*4) node[anchor=north] {\scriptsize{$\displaystyle\frac{dy}{dt}$}};
  \draw ( 30.-10.5*4.+10.*4. , 5.-7.8*4.+9.2*4 - 1) node[anchor=east] {\scriptsize{$\displaystyle\frac{dx_{1}}{dt}$}};

  \draw (10.,0.) node[anchor=north] {\scriptsize{$y$}};
  \draw (0.,9.5) node[anchor=west] {\scriptsize{$x_{1}$}};

\end{tikzpicture}
    \caption{Isolation in segments around the slow manifold for small~$\epsilon$. The fast component of the vector field dominates the slow one
    and offsets the influence of the slope on isolation inequalities.}\label{slopeFig}
  \end{figure}
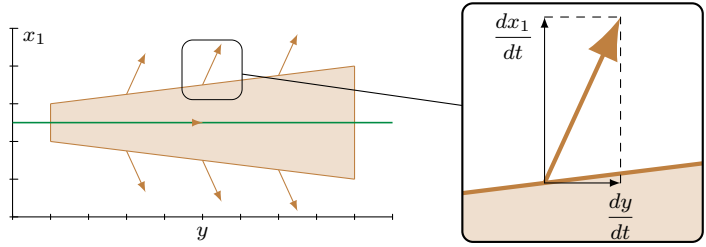

  Verifying (S3b) goes along the same lines, the expression for outward normals is
  \begin{equation}\label{eq:n-}
    n_{+}(x,y) = \left(0, \sgn(x_{2}-1), -\frac{\delta_{s,ur} - \delta_{s,ul}}{b_{u}-a_{u}} \right),
  \end{equation}
  and one readily checks that
  \begin{equation}
    \begin{aligned}
      \langle (f,\epsilon g), n_{+} \rangle (x,y,\epsilon) &\approx \lambda_{s,1}(y) |x_{2}-1|  - \epsilon g(x,y,\epsilon) \frac{\delta_{s,ur} - \delta_{s,ul}}{b_{u}-a_{u}}  \\
      &< -\delta_{1} |x_{2} - 1| + \frac{\epsilon}{\delta_{1}} \frac{|\delta_{s,ur} - \delta_{s,ul}|}{b_{u}-a_{u}}<0, \ \epsilon \in (0,\bar{\epsilon}_{0}],
    \end{aligned}
  \end{equation}
  decreasing $\bar{\epsilon}_{0}$ if necessary.
  
  We remark that our proof effectively relies on the fact that the fibers $L_u(\cdot , y)$ form suitable isolating blocks in the fast subsystem.

  The only difference in the construction of $S_{d}$ is that the recipe for $c_{S_{d}}$ has to include a flip in the $y$ direction
  so we can have $X_{S_{d},\text{in}}=X_{S_{dr},\text{out}}$ and $X_{S_{d},\text{out}}=X_{S_{dr},\text{in}}$.
  By taking minimum of all upper bounds for $\bar{\epsilon}_{0}$ throughout the proof we obtain the desired $\epsilon_{0}$.
\end{proof}

\subsection{The homoclinic orbit}\label{sec:covslowfast2}

We are looking for a codimension one situation, hence we add an additional parameter $\theta \in \rr$.
Our model example is now given by a family of fast-slow systems
\begin{equation}\label{fastslow2}
  \begin{aligned}
    \dot{x} &= f(x,y,\theta,\epsilon), \\
    \dot{y} &= \epsilon g(x,y,\theta,\epsilon).
  \end{aligned}
\end{equation}
where $x \in \rr^{2}$, $y \in \rr$, $f,g$ are smooth functions of $(x,y,\theta,\epsilon)$ and $0 < \epsilon \ll 1$ is the small parameter.
As in the previous example we will write $x=(x_{1},x_{2})$ to denote the respective fast coordinates and
use the symbols $\pi_{x_1}$, $\pi_{x_2}$, $\pi_y$ to denote the respective projections. 

As in the case of the periodic orbit, by the unstable/stable manifold of a branch of the slow manifold we will mean the union of the unstable/stable manifolds
of the equilibria forming it.

The counterparts of conditions (P1)-(P3) for
the existence of a homoclinic orbit of~\eqref{fastslow2} (in the vicinity of $\theta=0$ in the parameter space) are as follows:
\begin{itemize}
  \item[(H1)] we have two branches of the slow manifold $\Lambda_{\pm 1}$,
    that coincide with $\{ 0 \} \times \{ 1 \} \times \rr$ and $\{ 0 \} \times \{ -1 \} \times \rr$,
    respectively, and their position does not depend on $\theta$.
    Both are hyperbolic with one expanding and one contracting direction, and the vector field in their neighborhoods $U_{\pm 1}$ is of the following form:
    \begin{align}
      f(x,y,\theta,0) &= A_{\pm 1}(y,\theta)(x \mp [0,1]^{T} ) + h_{\pm 1}(x,y,\theta), \label{h1:1} \\
      \exists \tilde{\epsilon}_{0}, \tilde{\theta}_0 >0: \quad 0 &< \delta_{1}  \leq  g(x,y,\theta,\epsilon) \leq \delta^{-1}_{1}, \ (x,y) \in U_{1},  \label{h1:2} \\
      -\delta_{-1}^{-1} &<  g(x,y,\theta,\epsilon) < \delta_{-1} < 0, \ (x,y) \in U_{-1}, \label{h1:extra} \\
      \begin{split}
      -\delta_{-1}^{-1} &<\frac{\partial g(x,y,\theta,\epsilon)}{\partial y} < -\delta_{-1} < 0, \ (x,y) \in U_{-1},  \\
      &\forall \ \epsilon \in (0, \tilde{\epsilon}_{0}],\ \theta \in [-\tilde{\theta}_0, \tilde{\theta}_0]. \label{h1:3}
    \end{split}
    \end{align}
    The functions $A_{\pm 1},h_{\pm 1}$ are assumed to be smooth and to have the following properties
    \begin{align}
      A_{\pm 1}(y,\theta) &= \left[ \begin{array}{cc} \lambda_{u,\pm 1}(y,\theta) & O(\theta) \\ O(\theta) & \lambda_{s,\pm 1}(y,\theta) \end{array} \right], \label{H1:3} \\
        -\delta_{\pm 1}^{-1} \leq \lambda_{s,\pm 1}(y,\theta) \leq - \delta_{\pm 1} &< 0 < \delta_{\pm 1} \leq \lambda_{u, \pm 1}(y,\theta) \leq \delta_{\pm 1}^{-1}, \label{H1:4} \\
        \frac{h_{\pm 1}(x,y,\theta)}{ \vectornorm{ x - (0,\pm 1) } } &\xrightarrow{x \to (0,\pm 1)} 0,\ \forall y, \theta \label{H1:5}
    \end{align}
    The values $\delta_{\pm 1} > 0$ are some constant bounds, which in particular do not depend on neither $\epsilon$, $\theta$, nor $y$.

  \item[(H2)] For the parameterized family of fast subsystems,
    for $\theta = 0$ and $y=\mp 1$ there exist heteroclinic connections between
    the equilibria $(0,-1)$ and $(0, 1)$, in the first case, and $(0,1)$ and $(0,-1)$ in the second.
    The connection between $(0,-1)$ and $(0, 1)$ for $y=-1$ is assumed to be transversal with respect to $\theta$,
    and the connection between $(0,1)$ and $(0,-1)$ is assumed to be transversal with respect to $y$.
    That means: given any two one-dimensional transversal sections $\Sigma_{f, \pm 1}$ for the fast subsystems for $y=\pm 1$
    which have a nonempty, transversal intersection with the heteroclinic orbits, the maps $\Psi_{\pm 1}$ given by
    \begin{equation}
      \begin{aligned}
      \Psi_{1}: (y,\theta) &\to W^u_{1 ,\Sigma_{f,  1} }(y,\theta)- W^s_{-1,\Sigma_{f, 1}}(y,\theta) \in \rr,\\
      \Psi_{-1}: (y,\theta) &\to W^u_{-1 ,\Sigma_{f,  -1} }(y,\theta)- W^s_{1 ,\Sigma_{f, -1}}(y,\theta) \in \rr
    \end{aligned}
    \end{equation}
    satisfy the following:
    \begin{align}
       \Psi_{1} (1,0) &= 0, \\
       \frac{\partial{\Psi_{1}}}{\partial y} (1,0) &> 0, \\
       \Psi_{-1} (-1,0) &= 0, \\
       \frac{\partial{\Psi_{-1}}}{\partial \theta} (-1,0) &> 0.
    \end{align}

    Here $W^u_{\pm 1, \Sigma_{*}}(y,\theta)$ and $W^s_{\pm 1, \Sigma_{*}}(y,\theta)$ denote the first intersections between the appropriate branches
    of the unstable/stable manifolds of the equilibria $(0,\pm 1)$ with a given section $\Sigma$ in the section coordinates.

  \item[(H3)] Denote the points $(0,-1,-1)$, $(0,1,-1)$, $(0,1,1)$, $(0,-1,1)$ by $\Gamma_{\alpha}$, $\alpha \in \mathcal{I} = \{ dl, $ $ul, ur, dr \}$, 
    respectively. 
    For $\epsilon=\theta=0$, for each $\alpha \in \mathcal{I}$ there exist respective neighborhoods $V_{\alpha}$ of $\Gamma_{\alpha}$,
    such that if $\Lambda_{\pm 1} \cap V_{\alpha}$ is the part of the slow manifold contained
    in $V_{\alpha}$, then the part of its unstable manifold contained in $V_{\alpha}$ coincides with the plane $\rr \times \{\pm 1\} \times \rr$,
    and the part of the stable manifold contained in $V_\alpha$ - with the plane $ \{ 0 \} \times \rr \times \rr$. Without loss of generality we can have
    $\bigcup_{\alpha \in \mathcal{I}} V_{\alpha} \subset ( U_{-1} \cup U_{1} )$.

  \item[(H4)] The point $\Gamma_{dl}$ is an equilibrium of the system~\eqref{fastslow2} for 
    all $\theta \in \rr$ and $\epsilon \in [0,1]$\footnote{It corresponds to the $(0,0,0)$ equilibrium of the FitzHugh-Nagumo system~\eqref{FhnOde}.}.
    Moreover, its unstable manifold $W^u_{\Gamma_{dl}}$ varies continuously with parameters $\epsilon$ and $\theta$.
    We also assume, that there exists a compact set $B$ containing $\Gamma_{dl}$ in its interior, 
    which is of the form of a rectangular cuboid $[-\delta_{x, dl}, \delta_{x,dl}] \times [-\delta_{x, dl}, \delta_{x,dl}] \times [-\delta_{y, dl}, \delta_{y,dl}]$\footnote{We note,
      that the witdth in $x_1$ direction is assumed to be the same as in $x_2$ -- this will be important later in view of Remark~\ref{rem:ratio}.}, 
    and satisfies the following:
    \begin{itemize}
      \item given any maximal diameter $\diam_{\text{max}}>0$, the set $B$ can be chosen to satisfy $\diam(B) < \diam_{\text{max}}$.
      \item After fixing its size, $B$ forms an isolating block for~\eqref{fastslow2} with $u(B)=1$, $s(B)=2$, for all $\epsilon+|\theta|$ small enough and $\epsilon$ positive. 
          Its exit direction is spanned by $x_1$ and its entry directions are spanned by $x_2,y$ (so the coordinate change $c_B$ is given by a translation and rescaling).    
      \item $W^s_B (\Gamma_{dl})$ is a vertical disk in $B$ varying continuously with $\epsilon$ and $\theta$.
    \end{itemize}
\end{itemize}

Out of these conditions, only the second part of (H4) may seem artificial. 
In fact, such conditions can be deduced by equipping $B$ with an $\epsilon$-dependent cone field
from more qualitative assumptions (E1) and (E2) given in Section~\ref{sec:cc}.
For an $\epsilon$-independent $f$ and $g$ the construction of a required block with cones of the form $(B,Q_\epsilon)$
in such scenario is presented in Subsections~\ref{subsec:cc1} and~\ref{subsec:cc2}.

We will denote by $X_B$ the two-dimensional rectangular h-set with $u(X_B)=1$, $s(X_B)=1$ formed on the face of the block given by 
$[-\delta_{x, dl}, \delta_{x,dl}] \times [-\delta_{x, dl}, \delta_{x,dl}] \times \{ \delta_{y,dl} \}$.
By Lemma~\ref{lem:sidedisk} applied to the vertical disk $W^s_B (\Gamma_{dl}) \cap |B|$ one obtains that $W^s_B (\Gamma_{dl}) \cap |X_B|$
is a vertical disk in $X_B$.

For two real numbers $\theta_l, \theta_r$ with $\theta_l < 0 < \theta_r$, 
let $Z_{[\theta_l, \theta_r]}$ be a parameter h-set given by $u(Z_{[\theta_l, \theta_r]})=1$, $s(Z_{[\theta_l, \theta_r]})=0$, $|Z_{[\theta_l, \theta_r]}| = [\theta_l, \theta_r]$ 
(the coordinate change $c_{Z_{[\theta_l, \theta_r]}}$ is given by a rescaling to $[-1,1]$).

\begin{thm}\label{thm:heuristic2}
  Under assumptions (H1)-(H4), there are values of $\epsilon_{0}>0$, $ \theta_l < 0 < \theta_r$, such that
  for all $\epsilon \in (0,\epsilon_0]$ and $\theta \in [\theta_l,\theta_r]$ the system~\eqref{fastslow2} contains 
  the following ($\epsilon$ and $\theta$-independent) objects:
  \begin{itemize}
    \item an isolating block $B$ given by (H4),
    \item $S_{u}$, $S_{d}$ - two ``long'' isolating segments positioned around the branches $\Lambda_{\pm 1}$ of the slow manifold;
    \item $S_{\alpha},\ \alpha \in \mathcal{I} \backslash \{dl\}$ - three short ``corner'' isolating segments, each containing the respective point $\Gamma_{\alpha}$;
  \end{itemize}
  along with the associated transversal sections of the form $\Sigma_{S_{*},\text{in}},\Sigma_{S_{*},\text{out}}$, with
  \begin{equation}
    \begin{aligned}
      u(S_{dr}) &= s(S_{dr})
      \\ &= u(S_{ul})=s(S_{ul})= u(S_{ur})=s(S_{ur})
      \\ &=u(S_{u}) = s(S_{u})=u(S_{d})=s(S_{d})
      \\ &= 1.
    \end{aligned}
  \end{equation}

  We have the following relations among the h-sets on faces of the respective isolating segments/blocks:
  \begin{align}
   X_{S_{ul},\text{out}} &= X_{S_{u},\text{in}}, \\
   X_{S_{u},\text{out}} &= X_{S_{ur},\text{in}}, \\
   X_{S_{dr},\text{out}} &= X_{S_{d},\text{in}}, \\
   X_{S_{d},\text{out}} &= X_B.
  \end{align}
  The collection of h-sets
  \begin{equation}
    \begin{aligned}
      \mathcal{X}_{\text{FHN},\text{H}} &= \\
      \{ &Z_{[\theta_l, \theta_r]}, \ X_{S_{ul},\text{ls}},\ X_{S_{u}, \text{in}},\\
        &X_{S_{u},\text{out}},\ X_{S_{ur},\text{lu}},\ X_{S_{dr},\text{rs}}, \\
      &X_{S_{d}, \text{in}},\ X_{S_{d},\text{out}}, \ X_{B}
      \}
    \end{aligned}
  \end{equation}
  together with the vertical disk  satisfies assumptions of Theorem~\ref{thm:2} for $\epsilon \in (0,\epsilon_{0}]$.
  In particular we have the following covering relations among the h-sets
  not connected by an isolating segment:
  \begin{align}
    Z_{[\theta_l, \theta_r]} &\cover{W_L^{u} }  X_{S_{ul},\text{ls}} \quad \forall \epsilon \in (0,\epsilon_0],\label{eq:Zcover} \\
    X_{S_{ur},\text{lu}} &\cover{P_{R}}  X_{S_{dr},\text{rs}} \quad \forall  \epsilon \in (0,\epsilon_0],\ \theta \in [\theta_l, \theta_r],
  \end{align}
  where 
  \begin{itemize}
    \item $W_L^{u} : |Z_{[\theta_l, \theta_r]}| \to \Sigma_{L}$ is a map of argument $\theta$, which (for fixed $\epsilon$) 
      assigns the intersection point of the appropriate branch of the unstable manifold of $\Gamma_{dl}$\footnote{That means the same branch as in (H2) for $\epsilon=0$.}
      with a section $\Sigma_{L}$ containing $X_{S_{ul},ls}$
    \item $P_R$ is a Poincar\'e map from $X_{S_{ur},\text{lu}}$ to a transversal section containing $X_{S_{dr},\text{rs}}$.
  \end{itemize}
  
  As a consequence, there exists a homoclinic loop to $\Gamma_{dl}$ for these parameter values.
\end{thm}

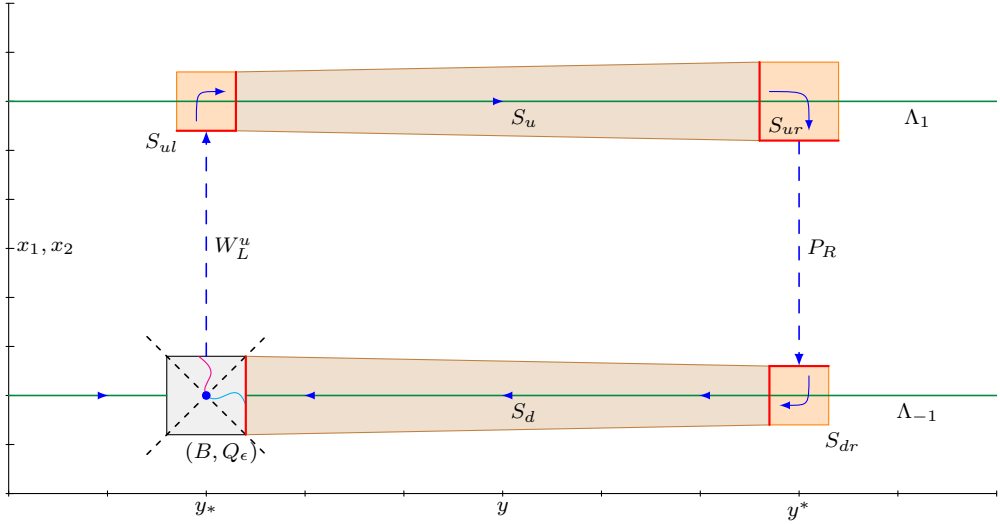
\begin{figure}
  \centering
  \begin{tikzpicture}[line cap=round,line join=round,>=latex,x=13mm,y=6.5mm]
 
  \draw[color=black] (-0.,0.) -- (10,0.);
  \foreach \x in {0,...,10}
  \draw[shift={(\x,0)},color=black] (0pt,1pt) -- (0pt,-1pt) node[below] {};

  \draw[color=black] (0.,0.) -- (0.,10);
  \foreach \y in {0,...,10}
  \draw[shift={(0,\y)},color=black] (2pt,0pt) -- (-1pt,0pt) node[left] {};
 

  \clip(0,-1) rectangle (10,10);

  \filldraw[draw=orange, fill=orange!25] (1.7,7.4) -- (1.7,8.6) -- (2.3,8.6) -- (2.3,7.4) -- (1.7,7.4);
  \filldraw[draw=black, fill=shadecolor!10] (1.6,1.2) -- (1.6,2.8) -- (2.4,2.8) -- (2.4,1.2) -- (1.6,1.2);
  \filldraw[draw=orange, fill=orange!25] (7.7,1.4) -- (7.7,2.6) -- (8.3,2.6) -- (8.3,1.4) -- (7.7,1.4);
  \filldraw[draw=orange, fill=orange!25] (7.6,7.2) -- (7.6,8.8) -- (8.4,8.8) -- (8.4,7.2) -- (7.6,7.2);

  \filldraw[draw=brown, fill=brown!25] (7.7,1.4) -- (7.7,2.6) -- (2.4,2.8) -- (2.4,1.2) -- (7.7,1.4);
  \filldraw[draw=brown, fill=brown!25] (7.6,7.2) -- (7.6,8.8) -- (2.3,8.6) -- (2.3,7.4) -- (7.6,7.2);

  \draw [color=darkgreen,semithick,domain=0.:10.] plot(\x,{8});
  \draw [color=darkgreen,semithick,domain=0.:1.6] plot(\x,{2});
  \draw [color=darkgreen,semithick,domain=2.4:10.] plot(\x,{2});

  
  \draw [semithick,dash pattern=on 2pt off 3pt,color=black] (1.4,3.2) -- (2.6,0.8);
  \draw [semithick,dash pattern=on 2pt off 3pt,color=black] (1.4,0.8) -- (2.6,3.2);

  \draw [color=cyan!233] (2,2) .. controls (2.15,1.73) and (2.3,2.4) .. (2.4,1.8);
  \draw [color=magenta] (2,2) .. controls (1.92,2.3) and (2.2,2.4) .. (1.92,2.8);


  \draw [->,semithick,dash pattern=on 5pt off 5pt,color=blue] (2.,2.8) -- (2.,7.4);
  \draw [->,semithick,dash pattern=on 5pt off 5pt,color=blue] (8.,7.2) -- (8.,2.6);
 
  \node[draw,circle,inner sep=1pt,fill, color=blue] at (2,2){};

  \draw [->,color=blue] (4.99,8.) -- (5.01,8.);
  \draw [->,color=blue] (5.01,2.) -- (4.99,2.);
  \draw [->,color=blue] (0.99,2.) -- (1.01,2.);
  \draw [->,color=blue] (3.01,2.) -- (2.99,2.);
  \draw [->,color=blue] (7.01,2.) -- (6.99,2.);
 
  \draw [color=red,thick] (7.7,2.6) -- (8.3,2.6);
  \draw [color=red,thick] (2.4,2.8) -- (2.4,1.2);
  \draw [color=red,thick] (7.7,2.6) -- (7.7,1.4);
 
  \draw [color=red,thick] (7.6,8.8) -- (7.6,7.2);
  \draw [color=red,thick] (7.6,7.2) -- (8.4,7.2);
  \draw [color=red,thick] (2.3,7.4) -- (2.3,8.6);
  \draw [color=red,thick] (2.3,7.4) -- (1.7,7.4);
 
  \draw [color=blue,<-] (2.3-0.1,8.+0.2) ..controls(2.-0.1,8.+0.2).. (2.-0.1,7.4+0.2);
  \draw [color=blue,<-] (8.4-0.3,8.-0.6) ..controls(8.4-0.3,8.8-0.6).. (8.-0.3,8.8-0.6);
  \draw [color=blue,->] (8.+0.1,2.6-0.2) ..controls(8.+0.1,2.-0.2).. (7.7+0.1,2.-0.2);

  \begin{scriptsize}
    \draw (9.2,8.) node[anchor=north] {$\Lambda_{1}$};
    \draw (9.2,2.) node[anchor=north] {$\Lambda_{-1}$};
 
    \draw (2.,5.) node[anchor=west] {$W_L^{u}$};
    \draw (8.,5.) node[anchor=west] {$P_{R}$};
  
    \draw (5.,8.) node[anchor=north west] {$S_{u}$};
    \draw (5.,2.) node[anchor=north west] {$S_{d}$};   

    \draw (8.14,7.84) node[anchor=north east] {$S_{ur}$};
    \draw (1.3,7.5) node[anchor=north west] {$S_{ul}$};     
    \draw (2.61,1.22) node[anchor=north east] {$(B,Q_{\epsilon})$};
    \draw (8.7,0.7) node[anchor=south east] {$S_{dr}$};     
  
    \draw (2.,0.) node[anchor=north] {$y_{*}$};
    \draw (8.,0.) node[anchor=north] {$y^{*}$};

    \draw (5.,0.) node[anchor=north] {$y$};
    \draw (0.,5.) node[anchor=west] {$x_{1}, x_{2}$};
  \end{scriptsize}

\end{tikzpicture}
  \caption{A schematic drawing for the model example for the homoclinic orbit. 
    An isolating block with cones $(B,Q_{\epsilon})$ gives bounds on the unstable and stable manifold of the zero equilibrium. 
    The manifolds are connected by a sequence isolating segments and Poincar\'e maps,
    which forces the existence of the homoclinic orbit; the h-sets forming the sequence are plotted in red.}\label{schemeFig2}
\end{figure}

The proof is significantly more involved than the proof of Theorem~\ref{thm:heuristic}, due to $\theta$-dependent off-diagonal terms in the
diagonalization of the fast subsystem $A_{\pm 1}(y,\theta)$ assumed in~\eqref{H1:3}. We need to include these terms, as our model system is assumed to 
reflect the qualitative properties of a suitable diagonalization of the FitzHugh-Nagumo system at a particular value of $\theta$. 
We note, that we have to be careful with decreasing the range $[\theta_l, \theta_r]$, 
as it has to be wide enough to generate the covering~\eqref{eq:Zcover}.

Due to these difficulties, and because this theorem is not necessary to prove any of the main theorems in this thesis,
we will only sketch its proof without being very formal. 
We will focus on these parts, where it significantly differs from the proof of Theorem~\ref{thm:heuristic}.
Therefore we suggest to the reader to get acquainted with the proof of Theorem~\ref{thm:heuristic} first.

We start with the following technical remark.

\begin{rem}\label{rem:ratio}
  In the proof of Theorem~\ref{thm:heuristic} isolating segments are essentialy given by unions of isolating blocks
  for the fast subsystem $\dot{x}=f(x,y,0)$, over the slow variable $y$. The isolation inequalities (B1), (B2) (implying (S2b), (S3b) for the segment)
  are guaranteed to be satisfied 
  for any rectangular blocks with supports of the form\footnote{The change of variables is given by translation and rescaling, first coordinate exit, second entry.}
  \begin{equation}\label{eq:blockform}
    (0,\pm 1 ) + \left( [-a(y),a(y)] \times [-b(y),b(y)] \right)
  \end{equation}
  on compact ranges of $y$, with $a(y),b(y)$ small, due to the diagonalized form of $f$.

  In our current model example the diagonalization of $f$ at $\Lambda_{\pm 1}$ given by $A_{\pm 1}$ in~\eqref{H1:3} contains off-diagonal
  elements, which are small for $\theta$ small enough.
  However, later in the proofs we will need to decrease the sizes of the blocks after having already fixed the range $\theta \in [\theta_l, \theta_r]$.
  Whether the isolation inequalities persist will then depend on the ratio $a(y) / b(y)$;
  for example when the block is too thin in the exit direction, then (B1) may be easily violated by the off-diagonal $O(\theta)$ term.
  To prevent that, we a priori restrict ourselves to ranges $\theta \in [\theta_l, \theta_r]$ small enough such that
  the blocks of the \emph{almost square} form
    \begin{equation}
      (0,\pm 1 ) + \left( [-a,a] \times [-b,b] \right) , \quad b \approx a \quad \text{e.g. } |b-a| \leq a/4
   \end{equation}
   are isolating blocks for $a,b$ small for the fast subsystem
   \begin{equation}\label{eq:fasttemp}
     \dot{x} = f(x,y,\theta,0)
   \end{equation}
   for a (wide enough for our further constructions) compact range of $y \in \mathbb{Y}$ and all $\theta \in [\theta_l, \theta_r]$.
   In fact, given compact ranges $\mathbb{A} \subset \rr^+$, $\mathbb{Y} \subset \rr$, 
   we can always find values $\theta_l<0$, $\theta_r>0$,
   such that for all $\theta \in [\theta_l,\theta_r]$, all $y \in \mathbb{Y}$ and all $(a(y),b(y))$ sufficiently small,
   with 
   \begin{equation}
     a(y)/b(y) \in \mathbb{A},
   \end{equation}
   the sets~\eqref{eq:blockform} form isolating blocks for~\eqref{eq:fasttemp}, as described above.
\end{rem}

First we prove the existence of the corner isolating segments and the proposed covering relations as a separate lemma.

\begin{lem}\label{lem:heuristic2}
  For $\epsilon \in (0,\bar{\epsilon}_{0}]$, $\bar{\epsilon}_{0}>0$ small,
  there exist real numbers $\theta_l$, $\theta_r$ with 
  $\theta_l<0 < \theta_r$ such that for $\theta \in [\theta_l,\theta_r]$ the system~\eqref{fastslow2} possesses:
  \begin{itemize}
    \item two transversal sections of the form
  \begin{equation}
    \begin{aligned}
      \Sigma_{L} &:= \{ (x_{1},x_{2},y): \ x_{2} = 1-\varepsilon_{L} \} \cap \tilde{V}_{ul} \subset V_{ul}, \\
      \Sigma_{R} &:= \{ (x_{1},x_{2},y): \ x_{2} = -1+\varepsilon_{R} \} \cap \tilde{V}_{dr} \subset V_{dr},
    \end{aligned}
  \end{equation}
  $\tilde{V}_{ul}, \tilde{V}_{dr}$ being neighborhoods of $\Gamma_{ul}$ and $\Gamma_{dr}$;
  \item three isolating segments $S_{ul}$, $S_{ur}$, $S_{dr}$, as specified in Theorem~\ref{thm:heuristic2}
  such that
  \begin{equation}
    \begin{aligned}
      |X_{S_{ul},\text{ls}}| &\subset \Sigma_{L}, \\
      |X_{S_{dr},\text{rs}}| &\subset \Sigma_{R},
    \end{aligned}
  \end{equation}
  and the following covering relations hold
  \begin{align}
    Z_{[\theta_l, \theta_r]} &\cover{W_L^{u} }  X_{S_{ul},\text{ls}} \quad \forall \epsilon \in (0,\bar{\epsilon}_0], \label{eq:zthetacov}\\
    X_{S_{ur},\text{lu}} &\cover{P_{R}}  X_{S_{dr},\text{rs}} \quad  \forall  \epsilon \in (0,\bar{\epsilon}_0],\ \theta \in [\theta_l, \theta_r].\label{eq:rightcovering}
  \end{align}  
  \end{itemize}
  
  Moreover, the sections $\Sigma_{*}$ and the segments $S_{*}$ are $\epsilon$ and $\theta$-independent and 
  given a maximal diameter $\diam_{\text{max}}>0$, they can be chosen
  so that
  \begin{equation}
   \diam(S_{*}) < \diam_{\text{max}},
  \end{equation}
  and the segment $S_{ul}$ is formed by almost square isolating blocks of the fast subsystem, as described in Remark~\ref{rem:ratio}.
  In addition, the numbers $\theta_l$, $\theta_r$ can be chosen to be arbitrarily small in absolute value,
  and can be fixed after defining the segments $S_{ur}$, $S_{dr}$.
\end{lem}

\begin{proof}
  To create the segments $S_{ur}$, $S_{dr}$ and generate the covering~\eqref{eq:rightcovering} we repeat the same construction as for the periodic orbit, for $\theta=0$
  and some small range $\epsilon \in (0,\bar{\epsilon}_0]$. 
  Covering relations and isolating segments will persist for $|\theta|$ small enough, hence
  from now on whenever we consider a small range of $\theta$'s containing $0$, we implicitly assume, that it is taken to be small enough for the above to hold.

  Now we will roughly sketch how to construct $S_{ul}$ and generate the covering~\eqref{eq:zthetacov}.
  Let us consider the $\theta$-dependent fast subsystem 
  \begin{equation}
    \dot{x}=f(x,-1,\theta,0).
  \end{equation}
  
  From (H1) it follows, that for $|\theta|$ small enough the set given by 
  \begin{equation}
    \tilde{\Sigma}_{L} := \{ (x_{1},x_{2}): \ x_{2} = 1-\tilde{\varepsilon}_{L} \} \cap \hat{V}_{ul},   
  \end{equation}
  is a transversal section for some small neighborhood $\hat{V}_{ul}$ of $(0,1)$, for $\tilde{\varepsilon}_L$ small. 
  Without loss of generality we can assume the crossing of the unstable and stable manifolds near the point
  $\Gamma_{ul}$ occurs for $x_{2}-1$ negative and take $\tilde{\varepsilon}_{L}> 0$.
  Moreover, we take $\tilde{\varepsilon}_L$ small enough, so for all $|\theta|$ small 
  and all $x$ in the set $\tilde{B}_{ul}:= (0,1) + ([- \tilde{\varepsilon}_L, \tilde{\varepsilon}_L]^2)$ the following two hold
  \begin{itemize}
   \item the forward trajectory of $x$ either belongs to the stable manifold of $(0,1)$ or escapes $\tilde{B}_{ul}$
  via the set $\tilde{B}_{ul}^- := (0,1) + (\{- \tilde{\varepsilon}_L,\tilde{\varepsilon}_L\} \times [- \tilde{\varepsilon}_L, \tilde{\varepsilon}_L])$.
    \item the backward trajectory of $x$ either belongs to the unstable manifold of $(0,1)$ or escapes $\tilde{B}_{ul}$
      via the set $\tilde{B}_{ul}^+ := (0,1) + ([- \tilde{\varepsilon}_L,\tilde{\varepsilon}_L ] \times \{- \tilde{\varepsilon}_L, \tilde{\varepsilon}_L \} )$
  \end{itemize}

  In fact, $\tilde{B}_{ul}$ forms an isolating block.
  The proof that fulfillment of these properties is possible for a small range of $\theta$ follows e.g. from arguments in~\cite{ZgliczynskiMan},
  by equipping $\tilde{B}_{ul}$ with a suitable cone field.
  In particular, a detailed proof of these properties would use the characterization
  of the unstable and stable manifold of a hyperbolic equilibrium as the forward invariant set of an isolating block with cones, given therein (Lemma 9). 

  In addition, for $|\theta|$ small enough we may ensure that the unstable and stable manifolds of $(0,1)$
  do not intersect the diagonals $\{ (0,1) + (\pm \lambda,\lambda) : \ \lambda \in \rr  \} \cap \tilde{B}_{ul}$.

  We now proceed to construct a range $\theta \in [\theta_l, \theta_r]$ and a smaller isolating block $B_{ul}$ such that
  \begin{itemize}
    \item the isolation inequalities for the block hold for all $\theta$'s in the range,
        so the block $B_{ul}$ will later be extended to the desired isolating segment $S_{ul}$; 
    \item the range is wide enough so the $\theta$-dependent unstable manifold $W^u(0,-1)$ covers
      (as a map of $\theta \in [\theta_l,\theta_r]=|Z_{[\theta_l,\theta_r]}|$) the bottom boundary h-set of $B_{ul}$ for $\theta=\theta_l$
      -- this will generate the covering relation~\eqref{eq:zthetacov}.
  \end{itemize}
  Our reasoning is depicted in Figure~\ref{shootingFig}. 

  First, let us consider the function $\Psi_{-1}(-1,\cdot)$ given by (H2) for the section $\tilde{\Sigma}_L$ .
  Without loss of generality we may assume that there exist arbitrarily small in absolute value parameters $\tilde{\theta}_l<0$ and $\tilde{\theta}_r>0$
  such that $\Psi_{-1}(-1,\cdot)$ is increasing for $\theta \in [\tilde{\theta}_l, \tilde{\theta}_r]$.
  Therefore $W^u(0,-1) \cap \tilde{B}_{ul}$ is in the bottom left quadrant of $\tilde{B}_{ul}$ for $\theta=\theta_l$
  and in the bottom right quadrant for $\theta=\theta_r$, where the quadrants are the four connected components of the set 
  $\tilde{B}_{ul} \setminus (W^u_{\tilde{B}_{ul}}(0,1) \cup W^s_{\tilde{B}_{ul}}(0,1) )$, see Figure~\ref{shootingFig}.

  We choose two points $D_{\pm}$ from the two half-diagonals $\{ (0,1) + (\pm \lambda, -\lambda), \ \lambda \in \rr^+\} \cap \tilde{B}_{ul}$
  defined by setting $\lambda:=\lambda_{ul}$, with $\lambda_{ul}>0$ taken to be small enough such that both of these points lie closer to $(0,1)$
  than the intersection points of $W^u(0,-1)$ with the half-diagonals for $\theta=\tilde{\theta}_l, \tilde{\theta}_r$.

  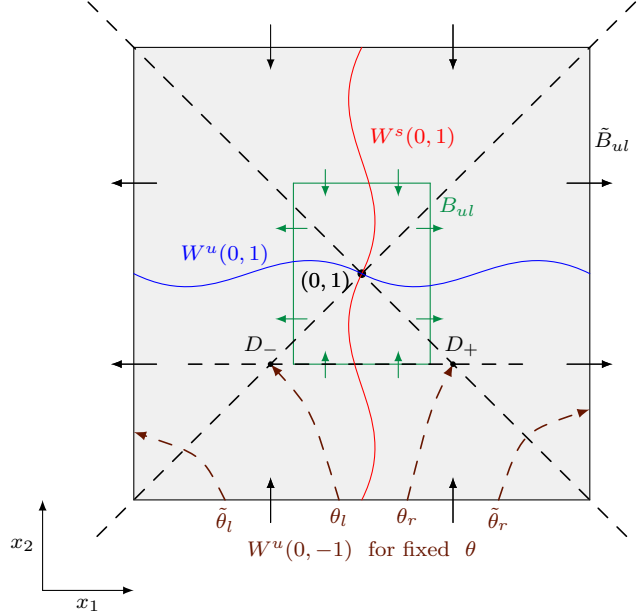
\begin{figure}
    \centering
    \begin{tikzpicture}[line cap=round,line join=round,>=latex,x=6.0mm,y=6.0mm]

  \fill [color=shadecolor!10] (-0.,-0.) -- (-0.,10.) -- (10.,10.) -- (10.,-0.); 
  \draw[color=black] (-0.,0.) -- (10,0.);
  \draw[color=black] (0.,0.) -- (0.,10);
  \draw[color=black] (0.,10.) -- (10.,10);
  \draw[color=black] (10.,10.) -- (10.,0.);

  \draw[color=darkgreen] (3.5,3.) -- (3.5,7.);
  \draw[color=darkgreen] (6.5,3.) -- (6.5,7.);
  \draw[color=darkgreen] (3.5,3.) -- (6.5,3.);
  \draw[color=darkgreen] (3.5,7.) -- (6.5,7.);
 
  \draw[color=darkgreen,->] (3.8,6.0) -- (3.1,6.0);
  \draw[color=darkgreen,->] (3.8,4.0) -- (3.1,4.0);
  \draw[color=darkgreen,->] (6.2,6.0) -- (6.8,6.0);
  \draw[color=darkgreen,->] (6.2,4.0) -- (6.8,4.0);
  \draw[color=darkgreen,->] (4.2,2.7) -- (4.2,3.3);
  \draw[color=darkgreen,->] (5.8,2.7) -- (5.8,3.3);
  \draw[color=darkgreen,->] (4.2,7.3) -- (4.2,6.7);
  \draw[color=darkgreen,->] (5.8,7.3) -- (5.8,6.7);

  \draw[color=black,->] (7.,10.5) -- (7.,9.5);
  \draw[color=black,->] (3.,-0.5) -- (3.,0.5);
  \draw[color=black,->] (7.,-0.5) -- (7.,0.5);
  \draw[color=black,->] (9.5,3.) -- (10.5,3.);
  \draw[color=black,->] (9.5,7.) -- (10.5,7.);
  \draw[color=black,->] (0.5,7.) -- (-0.5,7.);
  \draw[color=black,->] (0.5,3.) -- (-0.5,3.);
  
  \draw [semithick,dash pattern=on 5pt off 5pt,color=black] (1.2,3.) -- (8.8,3.);

  \draw[color=black,->] (3.,10.5) -- (3.,9.5);
  \draw[color=black,->] (7.,10.5) -- (7.,9.5);
  \draw[color=black,->] (3.,-0.5) -- (3.,0.5);
  \draw[color=black,->] (7.,-0.5) -- (7.,0.5);
  \draw[color=black,->] (9.5,3.) -- (10.5,3.);
  \draw[color=black,->] (9.5,7.) -- (10.5,7.);
  \draw[color=black,->] (0.5,7.) -- (-0.5,7.);
  \draw[color=black,->] (0.5,3.) -- (-0.5,3.);
 
  \draw [semithick,dash pattern=on 5pt off 5pt, color=Sepia, ->] (6,0) .. controls (6.5,2) .. (7,3);
  \draw [semithick,dash pattern=on 5pt off 5pt, color=Sepia, ->] (4.5,0) .. controls (3.9,2) .. (3,3);
  \draw [semithick,dash pattern=on 5pt off 5pt, color=Sepia, ->] (8,0) .. controls (8.5,1.5) .. (10,2);
  \draw [semithick,dash pattern=on 5pt off 5pt, color=Sepia, ->] (2.0,0) .. controls (1.5,1.0) .. (0,1.5);

 \node[draw,circle,inner sep=1pt,fill, color=black] at (5,5){};
 
 \node[draw,circle,inner sep=0.6pt,fill, color=black] at (3,3){};
 \node[draw,circle,inner sep=0.6pt,fill, color=black] at (7,3){};

  \draw [semithick,dash pattern=on 5pt off 5pt,color=black] (-1,11) -- (11,-1);
  \draw [semithick,dash pattern=on 5pt off 5pt,color=black] (11,11) -- (-1,-1);

  \draw [color=red] (5,5) .. controls (6,7) and (4,8) .. (5,10);
  \draw [color=red] (5,5) .. controls (4,3) and (6,2) .. (5,0);
  
  \draw [color=blue] (5,5) .. controls (7,4) and (8,6) .. (10,5);
  \draw [color=blue] (5,5) .. controls (3,6) and (2,4) .. (0,5);

  \clip(-3,-3) rectangle (13,13);
  
  \begin{scriptsize}
  \draw[color=black,->] (-2.,-2.) -- (-0,-2.);
  \draw[color=black,->] (-2.,-2.) -- (-2,-0.);

  \draw[color=red] (5.,8.) node[anchor=west] {$W^s(0,1)$};
  \draw[color=blue] (2.,5.) node[anchor=south] {$W^u(0,1)$};
  \draw[color=black] (4.2,4.4) node[anchor=south] {$(0,1)$};

  \draw[color=Sepia](4.5,0.0) node[anchor=north] {$\theta_l$}; 
  \draw[color=Sepia](6,0.0) node[anchor=north] {$\theta_r$}; 
  \draw[color=Sepia](2.0,0.0) node[anchor=north] {$\tilde{\theta}_l$}; 
  \draw[color=Sepia](8,0.0) node[anchor=north] {$\tilde{\theta}_r$}; 
 
  \draw[color=Sepia] (5.,-0.7) node[anchor=north] {$W^u(0,-1)$ \text{ for fixed } $\theta$};

  \draw[color=black] (4.2,4.4) node[anchor=south] {$(0,1)$};

  \draw[color=black] (2.8,3) node[anchor=south] {$D_{-}$};
  \draw[color=black] (7.2,3) node[anchor=south] {$D_{+}$};

  \draw[color=darkgreen]  (7.1,6.1) node[anchor=south] {$B_{ul}$};
  \draw[color=black]  (10.5,8.4) node[anchor=north] {$\tilde{B}_{ul}$};

  \draw (-2.,-1) node[anchor=east] {\scriptsize{$x_{2}$}};
  \draw (-1.,-2) node[anchor=north] {\scriptsize{$x_{1}$}};
\end{scriptsize}

\end{tikzpicture}
    \caption{Shooting with $W^u(0,-1)$ to the diagonal of $\tilde{B}_{ul}$.}\label{shootingFig}
  \end{figure}
  Let us now consider the bottom left quadrant.
  The first intersection point of $W^u(0,-1)$ with the section $\tilde{\Sigma}_L$ varies continuously with $\theta \in [\tilde{\theta_l},0]$,
  as does the intersection of the backward trajectory of $D_{-}$ with $\tilde{\Sigma}_L$ (which exists from our previous considerations
  about dynamics in $\tilde{B}_{ul}$).
  However, as $\theta$ varies from $\tilde{\theta}_l$ to $0$ these two points have to meet,
  as for $\theta=0$ the unstable manifold $W^u(0,-1)$ intersects $W^s(0,1)$.
  Therefore, there must exist a $\theta_l$ with $\tilde{\theta}_l < \theta_l<0$, such that $W^{u}(0,-1)$ passes through $D_{-}$.
  By the same argument there must exist a $\theta_r$ with $0 < \theta_r< \tilde{\theta}_r$ such that $W^u(0,-1)$ hits the half-diagonals at $D_{+}$.

  Now we are ready to define an h-set $B_{ul}$,
  which will form an isolating block for all $\theta \in [\theta_l,\theta_r]$ and give rise to the segment $S_{ul}$, by
  setting
  \begin{equation}
    \begin{aligned}
      |B_{ul}| &:= (0,1) + [-\lambda_{ul} + \varepsilon_{ul},\lambda_{ul} - \varepsilon_{ul}] \times [-\lambda_{ul},\lambda_{ul}],\\
      u(B_{ul})&:=s(B_{ul}) := 1, \\
      \varepsilon_{ul} &>0 \quad \text{small,}
    \end{aligned}
  \end{equation}
  with $x_1$ serving as the exit variable, $x_2$ as entry and the change of coordinates given by a translation to the origin and rescaling.
  Indeed, $B_{ul}$ is an isolating block for the whole range of $\theta$, as it can be chosen to be in an almost square form, 
  as described in Remark~\ref{rem:ratio}.
  
  We now add the variable $y$ and consider the full system for $\epsilon=0$, in which the fast subsystem is embedded.
  Let us observe, that for a given $\theta$ the previously considered unstable manifold $W^u(0,-1)$ from the fast subsystem at $y=-1$ 
  coincides now with the unstable manifold $W^u(\Gamma_{dl})$.

  We define the section $\Sigma_{L}$ by setting $x_2:=1-\lambda_{ul}$ (i.e. $\varepsilon_L:=\lambda_{ul}$).
  Since $B_{ul}$ forms an isolating block, this will indeed be a transversal section, when intersected with 
  some small neighborhood $\tilde{V}_{ul}$ of $\Gamma_{ul}$ containing $[-\lambda_{ul} + \varepsilon_{ul},\lambda_{ul} - \varepsilon_{ul}] \times \{ 1-\lambda_{ul} \} \times \{-1\}$.

  By Definition~\ref{defn:altcover}, 
  Lemma~\ref{covlemma} and previous considerations about the behavior of $W^u(0,-1)$, the h-set $Z_{[\theta_l,\theta_r]}$ $W^u_L$-covers the h-set formed from the set
  \begin{equation}
    \begin{aligned}
    & [-\lambda_{ul} + \varepsilon_{ul},\lambda_{ul} - \varepsilon_{ul}] \times \{ 1-\lambda_{ul} \} \times [-1-\delta_{u,ul}, -1+\delta_{u,ul}] \subset \Sigma_L, \\
    & \delta_{u,ul}>0 \quad \text{small},
  \end{aligned}
  \end{equation}
  by setting the first variable as exit, third as entry, and the change of coordinates as translation and rescaling.
  If we define a segment $S_{ul}$ by 
  \begin{equation}
    \begin{aligned}
    |S_{ul}|&:=  |B_{ul}| \times [-1-\delta_{u,ul}, -1+\delta_{u,ul}], \\
    u(S_{ul})&:=s(S_{ul}):=1
  \end{aligned}
  \end{equation}
  and the change of coordinates $c_{S_{ul}}$ as a translation and rescaling (so $x_1$ is the exit direction, $x_2$ entry and $y$ is the central one),
  then the covered h-set is precisely the h-set $X_{S_{ul},\text{ls}}$.
  
  Since $B_{ul}$ is an isolating block for $\theta \in [\theta_l, \theta_r]$, 
  the segment $S_{ul}$ will form an isolating segment
  for any $\epsilon \in (0,\bar{\epsilon}_0]$, $\bar{\epsilon}_0$ small enough and all $\theta \in [\theta_l, \theta_r]$, 
  and the covering~\eqref{eq:zthetacov} will persist.
\end{proof}

\begin{proof}[Proof of Theorem~\ref{thm:heuristic2}]
  The proof is analogous to the proof of Theorem~\ref{thm:heuristic}.
  We use Lemma~\ref{lem:heuristic2} to perform the following steps: 
  \begin{enumerate}
    \item First we construct two isolating segments $S_{ur}$ and $S_{dr}$ and the isolating block $B$
      (for $\epsilon>0$ small enough, $|\theta|$ small enough), as specified in the statement of the theorem. 
      The diameters of these sets have to be small enough, 
      so the higher order terms of the vector field can be assumed to be negligible for checking the isolation inequalities 
      in convex hulls of 
      \begin{itemize}
        \item the segment $S_{dr}$ and the h-set $B$,
        \item the segment $S_{ur}$ 
          and any choice of a small enough segment $S_{ul}$ given by a union of almost square isolating blocks for the fast subsystem, as in Remark~\ref{rem:ratio}.
      \end{itemize}
    \item Then, we fix the range $[\theta_l,\theta_r]$ such that 
      \begin{itemize}
        \item the segment $S_d$ connecting $S_{dr}$ and $B$ and constructed as a union of isolating blocks, like in the proof of Theorem~\ref{thm:heuristic} 
          is an isolating segment for $\theta \in [\theta_l,\theta_r]$ and $\epsilon>0$ small enough;
        \item  for all choices of the isolating segment $S_{ul}$ small enough in diameter, and given by a union of almost square isolating blocks of the fast subsystem, 
        the segment $S_u$ constructed as a union of isolating blocks (as in the proof of Theorem~\ref{thm:heuristic}), connecting $S_{ul}$ and $S_{ur}$ 
        would satisfy the isolation inequalities (S2b) and (S3b) for $\theta \in [\theta_l,\theta_r]$ and $\epsilon =0$
        -- see Remark~\ref{rem:ratio}.
    \end{itemize}
      \item By possibly further decreasing the range $[\theta_l,\theta_r]$, we construct $S_{ul}$ as specified above
        such that $S_{ul}$ forms an isolating segment for $\epsilon>0$ small, $\theta \in [\theta_l, \theta_r]$, 
        and the covering relation~\eqref{eq:Zcover} holds.
      \item We connect $S_{ul}$ and $S_{ur}$ by $S_{u}$, as in the proof Theorem~\ref{thm:heuristic}; by
        previous considerations $S_u$ will form an isolating segment for $\theta \in [\theta_l,\theta_r]$ and $\epsilon > 0$ small enough.
      \item Finally, we take the minimum of all upper bounds $\bar{\epsilon}_0$ on $\epsilon$ to get the desired range $\epsilon \in (0,\epsilon_0]$.
  \end{enumerate}

\end{proof}

\section{Descriptions of the computer assisted proofs}\label{sec:implementation}

Most of the numerical values in this section  are given as approximations with 8 significant digits.
An exception to that are the equation parameters, which are exact.
Therefore, computations that are described below are not actually rigorous, but the programs
execute rigorous computations for values close to the ones provided.
Actual values in the program used for rigorous computations are intervals with double precision endpoints -- we decided
that writing their binary representations would obscure the exposition.
If needed, exact values can always be retrieved by the reader from the programs.
If interval is very narrow and used to represent only one particular value,
such as a coordinate of a point, we just write a single value instead.

Rigorous and nonrigorous integration, computation of enclosures of Poincar\'e maps defined between affine sections
and their derivatives, linear interval algebra and interval arithmetics is handled
by routines from the CAPD library~\cite{CAPD} and we do not discuss it here.
For rigorous integration we used the Taylor integrator provided in CAPD.

The source code executing the proofs is available at the author's webpage~\cite{Czechowski}.
Our exposition loosely follows what is performed by our programs.
The best way to examine the proof in detail is to look into the source code files.  
For most objects we use the same notation in the description as in the source code,
however occasionally these two differ. 
In such cases identifying the appropriate variables should be easy from the context and from the comments
left in the source code files.

For a given vector object $\texttt{x}$, by $\texttt{x[i-1]}$ we denote the i-th coordinate of $\texttt{x}$.
We will denote the right-hand side of~\eqref{FhnOde} by $F$.

We recall that the fast subsystem of~\eqref{FhnOde} is given by: 
\begin{equation}\label{eq:fastsub}
     \begin{aligned}
       u'&=v, \\
       v'&=0.2(\theta v - u(u-0.1)(1-u) + w).
     \end{aligned}
\end{equation}

Unless otherwise specified, the half-open parameter intervals $\epsilon \in (0,\epsilon_0]$ is treated 
in computations by enclosing it in a closed interval $[0,\epsilon_0]$.
The assumption $\epsilon \neq 0$ is utilized only in verification of condition (S1a) for isolating segments (see Subsection~\ref{subsubsec:segments})
and in the verification of existence of isolating blocks satisfying the cone conditions around the zero equilibrium (see Section~\ref{sec:cc}).

\subsection{General remarks}

\subsubsection{H-sets and covering relations}\label{subsec:covrel}

  Almost every h-set $X$ appearing in our program 
  is two-dimensional with $u(X)=s(X):=1$ and can be identified with a parallelogram lying 
  within some affine section. 
  The only exceptions to this rule are three isolating blocks~\texttt{BU}, \texttt{BUext}, \texttt{BS} and the parameter h-set~\texttt{theta}, all 
  defined in Subsection~\ref{subsec:thm4}.

  Verification of covering relations is performed exclusively by means of Lemmas~\ref{covlemma}, \ref{backcovlemma}
  (see also Remark~\ref{rem:covlemma}). 
  The procedure is relatively straightforward and has been described in detail in several papers, see for example~\cite{WilczakZgliczynski2},
  therefore we do not repeat it here. We only mention that, if needed, the procedure may include subdivision of h-sets.
  This reduces the wrapping effect, but greatly increases runtimes
  (note that wrapping is already significantly reduced by use of the Lohner algorithm within the CAPD integration routines).
  Given an h-set $X$ we want to integrate with subdivision, we introduce an integer parameter $\texttt{div}$. It indicates into how many  
  equal intervals we divide the set in each direction. For example, setting $\texttt{div}=20$ means that we 
  integrate 20 pieces of $X^{-,l}, X^{+,r}$ and 400 pieces of $|X|$ to evaluate the image of the Poincar\'e map.
  In the outlines of our proofs we will indicate the values of $\texttt{div}$ to emphasize which parts of the proof involved time-consuming computations.

  \subsubsection{Segments}\label{subsubsec:segments}

  Our segments are cuboids placed along the slow manifold $C_{0}$ so that a part of it belonging to the singular orbit is enclosed by them. 
  For each segment $S$ we have $u(S)=s(S)=1$.
  All of the segments have the property~\eqref{eq:centralform}, with 
  the slow variable $w$ serving as the central variable. Therefore, to establish (S1a) it is enough to show~\eqref{eq:sgna},
  which is equivalent to verifying either $u>w$ or $u<w$ for all points of the segment. 
  This in particular allows us to handle half-open ranges $\epsilon \in (0,\epsilon_0]$ computationally,
  as at this point we effectively factor out the small parameter.
  
  Confirming (S2b) and (S3b) is simple, as all of the faces lie in affine subspaces.
  As the exit/entry directions we take the approximate directions of the unstable/stable bundles of $C_{0}$.
  Similarly as for verification of covering relations, we subdivide the sets $S^{-}$, $S^{+}$ before evaluating isolation inequalities. 
  The normals are constant within a face, the actual benefit is in reduction of wrapping in evaluation of the right-hand side of the vector field 
  over a face.

  Our segments are rigid and the stable and unstable bundles of $C_{0}$ actually slightly revolve as we travel along the manifold branches.
  By using a single segment to cover a long piece of the branch we could not expect conditions (S2b), (S3b) to hold anymore.
  Therefore we use sequences of short segments, the position of each is well-aligned with the unstable/stable bundles of $C_{0}$ - we call them \emph{chains of segments}.
  They are simply sequences of short segments placed one after another, so a longer piece of the slow manifold can be covered. 
  We require that each segment $S_{i}$ from a chain is an isolating segment and that
  for each two consecutive segments $S_{i}$, $S_{i+1}$ in the chain the transversal section $\Sigma_{i,\text{out}}$ containing the face $S_{i,\text{out}}$
  coincides with the section $\Sigma_{i+1,\text{in}}$ containing $S_{i+1,\text{in}}$ and there is a covering relation by the identity map
  \begin{equation}\label{eq:chain}
    X_{S_{i},\text{out}} \longlongcover{\id|_{\Sigma_{i+1,\text{in}}}}  X_{S_{i+1},\text{in}}.
  \end{equation}  
  In other words, the covering relation is realized purely by the change of coordinates $c_{X_{S_{i+1},\text{in}}}\circ c_{X_{S_{i},\text{out}}}^{-1}$.
  For purposes of checking the assumptions of Theorem~\ref{thm:2}, we treat the identity map as a special case of a Poincar\'e map, see Subsection~\ref{subsec:poinc}.

  A topic we think is worth exploring, is whether chains of segments are a viable alternative to numerical integration
  in computer assisted proofs for differential inclusions
  arising from evolution PDEs; or of stiff systems where one has a good guess for the orbit
  from a nonrigorous stiff integrator. In future we plan to conduct numerical simulations
  to get more insight on that matter.
  
  \paragraph{Representation of segments}\label{par:segments}
  
  Each segment \texttt{S} in our programs can be represented by 
  \begin{itemize}
    \item two points $\texttt{Front},\texttt{Rear} \in \rr^{3}$, serving as approximations of points on $C_{0}$,
    \item a 2x2 real matrix $\texttt{P}$ representing the rotation of the segment around the slow manifold (this does not need to be a rotation matrix) --
      it will contain approximate eigenvectors of the linearization of the fast subsystem~\eqref{eq:fastsub} at a selected point from $C_{0} \cap \texttt{S}$,
    \item four positive numbers $\texttt{a},\texttt{b},\texttt{c},\texttt{d}>0$ -- the pair $(\texttt{a},\texttt{b})$ describes how to stretch or narrow the exit and the entry widths of 
      the front face of the segment, respectively, and the pair $(\texttt{c},\texttt{d})$ does the same for the rear face.
  \end{itemize}

  For a pair of points $(a,b)$ and a 2x2 matrix $A$ we define an auxiliary linear map $\Pi_{a,b,A}: \rr^{2} \to \rr^{3}$ by
  \begin{equation}
    \Pi_{a,b,A}(x_{u},x_{s}) = \left[ {\begin{array}{c} A \\ 0 \end{array} } \right] \left[ {\begin{array}{c} ax_{u} \\ bx_{s} \end{array} } \right] .
  \end{equation}
  Our segment is then defined by
  \begin{equation}
    \begin{aligned}
      c_{\texttt{S}}^{-1}(x_{u},x_{s},x_{\mu})  &= (1-x_{\mu})  ( \texttt{Front} + \Pi_{\texttt{a},\texttt{b},\texttt{P}}(x_{u},x_{s}) )
      \\ &+ x_{\mu}  ( \texttt{Rear} + \Pi_{\texttt{c},\texttt{d},\texttt{P}}(x_{u},x_{s}) ).
   \end{aligned}
  \end{equation}
  
  For such segments one can define their front \& rear faces and the left/right entrance/exit faces $X_{\texttt{S},\text{in}}, X_{\texttt{S},\text{out}},
  X_{\texttt{S},\text{ls}}, X_{\texttt{S},\text{rs}}, X_{\texttt{S},\text{lu}}, X_{\texttt{S},\text{ru}}$ as in Section~\ref{sec:segments}.

  \paragraph{Construction of chains of segments}\label{subsec:chains}
  Our recipe for creating a chain of segments  \({\bf S} \)=$\{\texttt{S}_{\texttt{i}}\}_{\texttt{i} \in \{\texttt{1}, \dots, \texttt{N} \} }$ 
  along a branch of the slow manifold is as follows.
  We assume we are given two disjoint segments $\texttt{S}_{\texttt{0}}$, $\texttt{S}_{\texttt{N+1}}$ 
  positioned along the slow manifold $C_{0}$ that we would like to connect by a chain.
  Without loss of generality we may assume that we are on the upper branch of $C_{0}$, so
  $|\texttt{S}_{\texttt{0}}|$ is to the left of $|\texttt{S}_{\texttt{N+1}}|$ in terms of variable $w$.
  
  For each segment $\texttt{S}_{\texttt{i}}$ we will use 
  its representation 
  \begin{equation}
    (\texttt{Front}_{\texttt{i}}, \texttt{Rear}_{\texttt{i}}, \texttt{P}_{\texttt{i}}, \texttt{a}_{\texttt{i}}, \texttt{b}_{\texttt{i}}, 
    \texttt{c}_{\texttt{i}}, \texttt{d}_{\texttt{i}})
  \end{equation}
  given in Paragraph~\ref{par:segments}. Wherever we mention an identity map $\id$
  between two h-sets, we mean the identity map restricted to the common transversal section.

  Our chain will connect the segments $\texttt{S}_{\texttt{0}}$, $\texttt{S}_{\texttt{N+1}}$ 
  in the sense that
  \begin{align}
    X_{\texttt{S}_{\texttt{0}}} &\cover{\id} X_{\texttt{S}_{\texttt{1}},\text{in}},\label{eq:firstcover} \\
    X_{\texttt{S}_{\texttt{N}},\text{out}} &= X_{\texttt{S}_{\texttt{N+1}},\text{in}}.\label{eq:lastcover}
  \end{align}
  We remark that we connect the faces of two segments as this is what we later do in the proof of Theorem~\ref{thm:main1}, but with little changes these 
  could as well be any two parallelogram h-sets placed on sections crossing $C_{0}$.

  Creating a chain is a sequential process akin to rigorous integration with a fixed time step; to construct the segment $\texttt{S}_{\texttt{i}}$ we
  need to know the representation of the segment $\texttt{S}_{\texttt{i-1}}$,
  If $1 \leq i<N$ we define the segment $\texttt{S}_{\texttt{i}}$ as follows
  \begin{itemize}
    \item we set $\texttt{Front}_{\texttt{i}}:= \texttt{Rear}_{\texttt{i-1}}$.
    \item The point $\texttt{Rear}_{\texttt{i}}$ is constructed by locating an (approximate) equilibrium of the fast subsystem~\eqref{eq:fastsub} with Newton's method
      for 
      \begin{equation}
        w := \texttt{Front}_{\texttt{i}}\texttt{[2]} + \texttt{1}/\texttt{N},
      \end{equation}
      and then embedding it into the 3D space by adding the value of $w$ as the third coordinate.
    \item Columns of the matrix $\texttt{P}_{\texttt{i}}$ are set as approximate eigenvectors of linearization of~\eqref{eq:fastsub} at $\texttt{Rear}_{\texttt{i}}$.
    \item For $(\texttt{a}_{\texttt{i}},\texttt{b}_{\texttt{i}})$ we put
      \begin{equation}\label{eq:shrinkAndExpand}
        \begin{aligned}
          \texttt{a}_{\texttt{i}} &:= \texttt{c}_{\texttt{i-1}}/\texttt{factor},\\
          \texttt{b}_{\texttt{i}} &:= \texttt{factor} \times \texttt{d}_{\texttt{i-1}}
        \end{aligned}
      \end{equation}
      where $\texttt{factor}$ is a real number greater than 1. In our case hardcoding $\texttt{factor}:=1.05$ gave good results.
    \item For $(\texttt{c}_{\texttt{i}},\texttt{d}_{\texttt{i}})$ we put
      \begin{equation}
        \begin{aligned}
          \texttt{c}_{\texttt{i}} &:= \frac{\texttt{i}}{\texttt{N}} \texttt{a}_{\texttt{N+1}} + \frac{\texttt{N-i}}{\texttt{N}} \texttt{c}_{\texttt{0}}, \\
          \texttt{d}_{\texttt{i}} &:= \frac{\texttt{i}}{\texttt{N}} \texttt{b}_{\texttt{N+1}} + \frac{\texttt{N-i}}{\texttt{N}} \texttt{d}_{\texttt{0}}.
        \end{aligned}
      \end{equation}
  \end{itemize}
  For the segment $\texttt{S}_{\texttt{N}}$ we proceed by the same rules with the exception that
  we set 
  \begin{equation}
    \begin{aligned}
    \texttt{Rear}_{\texttt{N}}&:= \texttt{Front}_{\texttt{N+1}}, \\
    \texttt{P}_{\texttt{N}}&:=\texttt{P}_{\texttt{N+1}},
  \end{aligned}
  \end{equation}
  to comply with~\eqref{eq:lastcover}.
  
  For such $\texttt{S}_{\texttt{i}}$ we check the conditions (S1a), (S2b), (S3b) and the covering relation
  $X_{\texttt{S}_{\texttt{i-1}},\text{out}} \cover{\id} X_{\texttt{S}_{\texttt{i+1}},\text{in}}$.
  Then, we proceed to the next segment.

  For $\texttt{N}$ large it is easy to satisfy (S2b), (S3b) for each short segment $\texttt{S}_{\texttt{i}}$,
  as each $\texttt{P}_{\texttt{i}}$ approximates the directions of the unstable and stable bundle of $C_{0}$.
  Moreover, because $\texttt{Rear}_{\texttt{i-1}}$, $\texttt{Rear}_{\texttt{i}}$ are close, for each $\texttt{i} \in \{1,\dots,\texttt{N}\}$ we have
  \begin{equation}
    \texttt{P}_{\texttt{i}} \circ \texttt{P}_{\texttt{i-1}} \approx \id .
  \end{equation}
  Thus, for the identity map in the h-sets variables we get
  \begin{equation}
    \id_{s} = c_{X_{\texttt{S}_{\texttt{i}}, \text{in}} }  \circ c_{X_{\texttt{S}_{\texttt{i-1}}, \text{out}} }^{-1}
    \approx \left[ {\begin{array}{cc} \texttt{factor} & 0 \\ 0 & \frac{1}{\texttt{factor}} \end{array} } \right],
  \end{equation}
  and there are good odds that by use of Lemma~\ref{covlemma} we can succeed in satisfying the conditions~\eqref{eq:firstcover} and~\eqref{eq:chain}.

\subsection{Proof of Theorem~\ref{thm:main1}}\label{subsec:thm11}
To deduce the existence of a periodic orbit we check the assumptions of Theorem~\ref{thm:2}.
Our strategy resembles the one given for the model example in Subsection~\ref{sec:covslowfast}, which was portrayed in Figure~\ref{schemeFig}.
The main modifications are due to numerical reasons:
\begin{itemize}
  \item we introduce two additional sections on the trajectories of the fast subsystem heteroclinics, in some distance from the corner segments,
  \item instead of the ``long'' segments $S_{u}$, $S_{d}$ we place two chains of segments along the slow manifold connecting the corner segments
    -- see Paragraph~\ref{subsec:chains}.
\end{itemize}

We divide the parameter range $\epsilon \in (0,1.5 \times 10^{-4}]$ into 
two subranges $(0,10^{-4}]$ and $[10^{-4}, 1.5 \times 10^{-4}]$. The procedure is virtually the same for both ranges
and the only reason for subdivision is that the proof would not succeed for the whole range $\epsilon \in (0,1.5 \times 10^{-4}]$ in one go, due to an accumulation of overestimates.
Following steps are executed by the program for both ranges:

\begin{enumerate}[leftmargin=*]
  \item  First, we compute four ``corner points'' 
    \begin{equation}
      \begin{aligned}
        \texttt{GammaDL} &= (-0.10841296,0,0.025044220) \approx (\Lambda_d(w_{*}),w_{*}), \\
        \texttt{GammaUL} &= (0.97034558,0,0.025044220) \approx (\Lambda_u(w_{*}),w_{*}), \\
        \texttt{GammaUR} &= (0.84174629,0,0.098807631) \approx (\Lambda_u(w^{*}),w^{*}), \\
        \texttt{GammaDR} &= (-0.23701225,0,0.098807631) \approx (\Lambda_d(w^{*}),w^{*}).
      \end{aligned}
    \end{equation}
    This computation is nonrigorous; in short we perform a shooting with $w$ procedure for the fast subsystem~\eqref{eq:fastsub}
    from first-order approximations of stable and unstable manifolds of the equilibria to an intermediate section;
    this is an approach like in~\cite{GuckenheimerKuehn}.
    The matrices given by the approximate eigenvectors of the linearization of the fast subsystem at points $\texttt{GammaDL}$,
    $\texttt{GammaUR}$, $\texttt{GammaUL}$, $\texttt{GammaDR}$ are 
    \begin{equation}
      \begin{aligned}
        \texttt{PDL} &= \texttt{PUR} = \left[ {\begin{array}{cc} 1 & 1 \\ 0.34113340 & -0.21913340 \end{array} } \right], \\
        \texttt{PUL} &= \texttt{PDR} = \left[ {\begin{array}{cc} 1 & 1 \\ 0.46313340 & -0.34113340 \end{array} } \right],
      \end{aligned}
    \end{equation}
    respectively.
  \item  We initialize four ``corner segments'' \texttt{DLSegment}, \texttt{ULSegment}, \texttt{URSegment} and \texttt{DRSegment} with data
    from Table~\ref{table:corseg} as described in Paragraph~\ref{par:segments} and check that they are isolating segments. 
    For checking the isolation formulas (S2b), (S3b) we subdivide enclosures of each of the respective faces of the exit and the entrance set into $150^{2}$ equal pieces. 
    \begin{table}[ht]
      \begin{center}
      \begin{tabular}{ | l | l | l | l |}
        \hline
        Segment &  \texttt{Front}, \texttt{Rear} & \texttt{P} & $(\texttt{a},\texttt{b}) = (\texttt{c},\texttt{d})$  \\ \hline
        \texttt{DLSegment} & $\texttt{GammaDL} \pm (0,0,0.005)$   & \texttt{PDL} & $(0.015, 0.012)$ \\ \hline
        \texttt{ULSegment} & $\texttt{GammaUL} \mp (0,0,0.005)$  & \texttt{PUL} & $(0.01, 0.015)$ \\ \hline
        \texttt{URSegment} & $\texttt{GammaUR} \mp (0,0,0.005)$   & \texttt{PUR} & $(0.029, 0.019)$ \\ \hline
        \texttt{DRSegment} & $\texttt{GammaDR} \pm (0,0,0.005)$  & \texttt{PDR} & $(0.007, 0.03)$ \\ \hline
      \end{tabular}\caption{Initialization data for the four corner segments. The pair $(\texttt{a},\texttt{b})$
      determines the exit/entry direction widths of the segments and the difference $|\texttt{Front[2]} - \texttt{Rear[2]}|$ the central
        direction width.}\label{table:corseg}
      \end{center}
    \end{table}
  \item Unlike in the model example -- Lemma~\ref{lem:heuristic1}, we do not
    place the transversal sections we would integrate to as supersets of the left/right exit faces of the corner segments.
    Instead, two sections  $\texttt{leftSection}$ and $\texttt{rightSection}$ are positioned in some distance from the corner segments.
    We move away from the segments because rigorous integration too close to slow manifolds poses a numerical problem - 
    the vector field slows too much and the routines for verifying transversality fail.

    The section $\texttt{leftSection}$ is placed on the integration path between the segments \texttt{DLSegment}, \texttt{ULSegment}
    and the section $\texttt{rightSection}$ on the path between \texttt{URSegment} and \texttt{DRSegment}. 
    We define the following Poincar\'e maps:
    \begin{itemize}
      \item $\texttt{pmDL}$ is the Poincar\'e map from $X_{\texttt{DLSegment}, \text{ru}}$ to $\texttt{leftSection}$,
      \item $\texttt{pmUL}$ is the Poincar\'e map from a subset of $\texttt{leftSection}$ to the affine section containing $X_{\texttt{ULSegment}, \text{ls}}$,
      \item $\texttt{pmUR}$ is the Poincar\'e map from $X_{\texttt{URSegment},\text{lu}}$ to $\texttt{rightSection}$,
      \item $\texttt{pmDR}$ is the Poincar\'e map from a subset of $\texttt{rightSection}$ to the affine section containing $X_{\texttt{DRSegment}, \text{rs}}$.
    \end{itemize}

    Let now us briefly describe what covering relations we verify.

    We integrate the h-set $X_{\texttt{DLSegment}, \text{ru}}$ to $\texttt{leftSection}$ and create an h-set
    $\texttt{midLeftSet} \subset \texttt{leftSection}$ so that it is $\texttt{pmUL}$-covered by a small margin by $X_{\texttt{DLSegment}, \text{ru}}$,
    see Lemma \ref{covlemma}. 
    Then, we integrate the h-set $X_{\texttt{ULSegment}, \text{ls}}$ backward in time to $\texttt{leftSection}$ and 
    verify that $\texttt{midLeftSet}$ $\texttt{pmUL}$-backcovers $X_{\texttt{ULSegment}, \text{ls}}$.

    The h-set $X_{\texttt{URSegment}, \text{lu}}$ is integrated to $\texttt{rightSection}$, and, as in the previous case, we define
    an h-set $\texttt{midRightSet} \subset \texttt{rightSection}$, such that it is $\texttt{pmUR}$-covered 
    by $X_{\texttt{URSegment}, \text{lu}}$.
    Then, we integrate the h-set $X_{\texttt{DRSegment}, \text{rs}}$ backward in time to $\texttt{rightSection}$ and 
    verify that $\texttt{midRightSet}$ $\texttt{pmDR}$-backcovers $X_{\texttt{DRSegment}, \text{rs}}$.

    Altogether, we have the following covering relations:
    \begin{equation}
      \begin{aligned}
        X_{\texttt{DLSegment},\text{ru}} &\cover{\texttt{pmDL}} \texttt{midLeftSet} \backcover{\texttt{pmUL}} X_{\texttt{ULSegment}, \text{ls}}, \\
        X_{\texttt{URSegment},\text{lu}} &\cover{\texttt{pmUR}} \texttt{midRightSet} \backcover{\texttt{pmDR}} X_{\texttt{DRSegment}, \text{rs}}.
      \end{aligned}
    \end{equation}
    Parameter $\texttt{div}$ describing partitioning of h-sets for the rigorous integration was set to 20.

  \item To close the loop, we connect the h-sets $X_{\texttt{ULSegment},\text{out}}$ and $X_{\texttt{URSegment},\text{in}}$ by a chain of segments
    \({\bf UpSegment} \) and $X_{\texttt{DRSegment},\text{out}}$ and $X_{\texttt{DLSegment},\text{in}}$ by a chain of segments \({\bf DownSegment} \)
    as described in Paragraph~\ref{subsec:chains}. 
    The number of isolating segments in each chain $\texttt{N}$ is set to 80.  
    For verification of the isolation conditions (S2b), (S3b) in each chain we partition the enclosures of each of the
    faces of their exit and entrance sets into $110^{2}$ equal pieces.
\end{enumerate}

Many choices of program parameters were arbitrary; of most importance are the exit/entry/central direction widths of the corner segments
given in Table~\ref{table:corseg}. 
For very small $\epsilon$ ranges (such as $\epsilon \in (0,10^{-8}]$, $\epsilon \in (0,10^{-7}]$)
various reasonable guesses would yield successful proofs, due to the eminent fast-slow structure of the equations (cf. Section~\ref{sec:covslowfast}). 
However,  the range of possibilities would diminish as the upper bound on $\epsilon$ was increased,
and finding values for our final $\epsilon$ ranges was a long trial-and-error process. 
This can be explained as follows.
For large $\epsilon$'s the periodic orbit moves away from the singular orbit,
around which we position our sequence of segments and h-sets. Moreover, the hyperbolicity of the slow manifold, which plays a vital role in the creation
of the periodic orbit near the singular limit , decreases as $\epsilon$ increases.
Each time a value of a program parameter was adjusted in an attempt to succeed with a particular part of the proof,
it was possible that another part would fail. For example, increasing the central direction
widths of the corner segments facilitated the verification of covering relations for the Poincar\'e maps; but too much of an increase 
made isolation checks for the corner segments fail; increasing the exit direction widths of $\texttt{ULSegment}$, $\texttt{DRSegment}$
made the exit direction isolation checks (S2b) in segments of \({\bf UpSegment} \), \({\bf DownSegment} \)
easier to satisfy but had a negative effect on the covering relations; etc. 
It was particularly difficult to simultaneously obtain both isolation for the corner segments and covering relations in the fast regime.

By repeating the process of
\begin{itemize}
  \item trying to slightly increase the $\epsilon$ range,
  \item executing the program with given parameters,
  \item should the proof fail, changing the parameters in favor of the inequalities which were not fulfilled,
    at the cost of the ones where we still had some freedom,
\end{itemize}
we obtained a relatively large range of $\epsilon \in (0,1.5 \times 10^{-4}]$, for which the inequalities needed in our assumptions
hold by a very small margin. In particular, the right bound $1.5 \times 10^{-4}$ was large enough 
to include it in a continuation-type proof of Theorem~\ref{thm:main2}, performed in reasonable time and without using multiple precision.

It would certainly be helpful to have that procedure automated.
As one can see, we are effectively dealing with a~\emph{constraint satisfaction problem} (see~\cite{constraint})
where variables, given by the program parameters have to be chosen to satisfy constraints given by inequalities coming from covering relations
and isolation conditions.
In addition, verification of whether constraints are satisfied requires execution of the program and is fairly expensive computationally.
A suitable algorithm for adjustment of parameters to satisfy the constraints would allow to extend the range
of the small parameter $\epsilon \in (0,\epsilon_0]$ even further.
We remark that obtaining a large value of $\epsilon_0$ in this proof is crucial for achieving this parameter value 
with further validated continuation algorithms (like the one in Theorem~\ref{thm:main2})
This is due to the fact that the period of this unstable orbit is roughly proportional to $\frac{1}{\epsilon}$ 
(see Table~\ref{table:continuation}) which makes it virtually impossible to track the orbit by numerical integration methods
for very small $\epsilon$.

\subsection{Proof of Theorem~\ref{thm:main2}}\label{subsec:continuation}
Our strategy is to check the assumptions of Theorem~\ref{cov:sequence}
for a sequence of h-sets placed along a numerical approximation of an actual periodic orbit (not the singular orbit).
This can succeed for a very small range of $\epsilon$, then we need to recompute our approximation,
ending up with a continuation procedure.

We start by generating a numerical approximation vector of 212 points from the periodic orbit for $\epsilon = 0.001$ 
obtained from a nonrigorous continuation with MATCONT~\cite{MATCONT}.
From there we perform two continuation procedures, down to $\epsilon = 1.5 \times 10^{-4}$ and up to $\epsilon = 0.0015$.
Each step of the continuation consists of a routine $\texttt{proveExistenceOfOrbit}$ performed on equation~\eqref{FhnOde}
with an interval $\texttt{currentEpsRange}$ of width $\texttt{incrementSize}$ substituted for $\epsilon$.
It can be described by the following steps.

\begin{enumerate}[leftmargin=*]
  \item Given an approximation vector \( \bf initialGuess \) of $\texttt{pm\_count}$ points of the periodic orbit
    obtained from the previous continuation step (in the first step this is the MATCONT-precomputed approximation),
    we initialize a Poincar\'e section $\texttt{section}_{\texttt{i}}$ for each of the points $\texttt{initialGuess}_{\texttt{i}}$ by setting
    the origin of the section as the given point and its normal vector as the vector as the difference between
    the current and the next point of the approximation.
    Then, we refine the approximation by a nonrigorous $C^{1}$
    computation of Poincar\'e maps and their derivatives and application of Newton's method to the system of the form~\eqref{eq:problemform}.
    Note that we set the normal vector to be the difference between the current and the next point on the orbit rather than the direction
    of the vector field, as the latter can be misleading close to the strongly hyperbolic slow manifold.
    Let us denote by \( \bf correctedGuess \) the Newton-corrected approximation.
  \item Each $\texttt{section}_{\texttt{i}}$ is equipped with a coordinate system used for
    the purposes of covering by h-sets as described in Subsection~\ref{subsec:covrel}. The first column corresponding
    to the exit direction is obtained by a nonrigorous $C^{1}$ integration of any non-zero normalized vector by the variational equation of~\eqref{FhnOde}
    along the approximated orbit until it stabilizes; and then propagating it for each $\texttt{i}$ by one additional integration loop.
    Similarly, the second column (corresponding to the entry direction) is computed by backward integration of any non-zero normalized vector until it stabilizes and
    further propagation by inverse Poincar\'e maps for each $\texttt{i}$. 
    Then, we project these columns onto the orthogonal complement of $\texttt{normal}_{\texttt{i}}$.
  \item Let $\texttt{pm}_{\texttt{i}}$ be the Poincar\'e map between (a subset of) 
    $\texttt{section}_{\texttt{i}}$ and $\texttt{section}_{\texttt{i+1}\bmod \texttt{pm\_count}}$.
    We initialize a sequence of h-sets $X_{i}$ on sections $\texttt{section}_{\texttt{i}}$
    by specifying $X_{0}$ and sequentially generating the sets $X_{1},\dots,X_{\texttt{pm\_count}-1}$,
    so the covering relations $X_{i} \cover{\texttt{pm}_{\texttt{i}}} X_{i+1}$, $\texttt{i} \in \{0, \dots, \texttt{pm\_count} - 2\}$ hold by a small margin.     
    The periodic orbit is strongly hyperbolic and the h-sets will quickly grow in the exit direction.
    Therefore we put an additional upper bound on the growth of h-sets in that direction
    to prevent overestimates coming from integrating too large h-sets.
    For rigorous integration of h-sets the parameter $\texttt{div}$ was set to 5.
  \item We check that the following covering relation holds
    \begin{equation}
      X_{\texttt{pm\_count} - 1} \longlongcover{\texttt{pm}_{\texttt{pm\_count} - 1}} X_{0}.
    \end{equation}
    This implies the existence of the periodic orbit of for $\epsilon \in \texttt{currentEpsRange}$, by Theorem~\ref{cov:sequence}.
  \item We produce a new \( \bf initialGuess \) for the next step of continuation by removing the points from the approximate orbit
    where the integration time between the respective sections is too short and adding them where it is too long.
    This way we can adapt the number of sections to the period of the orbit.
  \item We move the interval $\texttt{currentEpsRange}$ and proceed to the next step of the continuation.
\end{enumerate}

The continuation starts with $\texttt{incrementSize}=10^{-6}$ and the size (diameter) of the h-set $X_{0}$ of order $10^{-6}$ and both of these parameters vary throughout the proof.
If any step of $\texttt{proveExistenceOfOrbit}$ fails - for example Newton's method does not converge or there is no covering
between the h-sets, the algorithm will try to redo all the steps for a decreased $\texttt{incrementSize}$ and proportionally decrease the size of the initial h-set.
If the algorithm keeps succeeding, the program will try to increase $\texttt{incrementSize}$ and the diameter to speed up the continuation procedure.
The theorem is proved when bounds of $\texttt{currentEpsRange}$ pass the bounds of $\epsilon$ we intended to reach.
Values of $\texttt{incrementSize}$ for several different $\texttt{currentEpsRange}$ can be found in Table~\ref{table:continuation} along with
periods of the periodic orbit and amounts of sections given by $\texttt{pm\_count}$.

\begin{table}[ht]
  \begin{center}
  \scalebox{0.90}{
      \begin{tabular}{ | l | l | l | l |}
        \hline
        \texttt{currentEpsRange} & \texttt{incrementSize} &  \texttt{period} & $\texttt{pm\_count}$  \\ \hline
        $[0.0014933550, 0.001499146]$ & $5.7918161 \times 10^{-6}$ &    [201.35884, 207.17313] & 179 \\ \hline
        $[0.001, 0.001001]$ & $10^{-6}$ &  [283.37351, 292.02862] & 212 \\ \hline
        $[4.9947443, 5.0200138] \times 10^{-4}$ & $2.5269501 \times 10^{-6}$ &  [521.07987, 557.55718] & 301 \\ \hline
        $[1.5057754, 1.5132376] \times 10^{-4}$ & $7.4621539 \times 10^{-7}$ &    [1593.3303, 1846.4787] & 671  \\ \hline
      \end{tabular}
    }\caption{Sample values from the validated continuation proof of Theorem~\ref{thm:main2}. As one can see,
    the period increases significantly as $\epsilon \to 0$, making it necessary to introduce more sections and lengthening the computations. }
    \label{table:continuation}
  \end{center}
\end{table}

\subsubsection{Further continuation}\label{subsec:further}

We have decided to stop the validated continuation at $\epsilon=0.0015$.
Above that value our continuation algorithm encountered
difficulties in its nonrigorous part, and needed many manual readjustments of the continuation parameters.
As we later checked with MATCONT, this seemed not to have been caused by any bifurcation,
so, most likely, it was just a defect of our ad-hoc method of continuing approximations of the periodic orbit
by computation of Poincar\'e maps between sections.
Nonrigorous continuation methods implemented in continuation packages such as MATCONT
are based on approximation of the orbit curve by Legendre polynomials and seem more reliable than our approach.
Such a good nonrigorous approximation with a large number of collocation points
would be enough to have a rigorous part of the continuation based on Poincar\'e maps succeed,
making further continuation only a matter of computation time.
We did not implement it though, as we have decided that we are satisfied with
how wide our $\epsilon$ range is. By Theorem~\ref{thm:main3} we have already reached
the value where the standard interval Newton-Moore method applied to a sequence of Poincar\'e maps succeeds, and we think it is clear
that a proof for higher values of $\epsilon$ will pose no significant theoretical or computational challenges.

\subsection{Proof of Theorem~\ref{thm:main3}}\label{subsec:thm13}
Recall the interval Newton-Moore method for finding zeroes of a smooth map $\mathcal{F}: \rr^{n} \to \rr^{n}$:
\begin{thm}[The interval Newton-Moore method~\cite{Alefeld, Neumaier, Moore}]\label{thm:intervalNewton}
  Let $X = \Pi_{i=1}^{n} [a_{i}, b_{i}]$, $\mathcal{F}: \rr^{n} \to \rr^{n}$ be of class $C^{1}$ and let $x_{0} \in X$.
  Assume the interval enclosure of $D\mathcal{F}(X)$, denoted by $[D \mathcal{F} (X)]$ is invertible. We denote by
  \begin{equation}
    \mathcal{N}(x_{0},X):=-[D \mathcal{F} (X)]^{-1} \mathcal{F} (x_{0}) + x_{0}
  \end{equation}
  the interval Newton operator. Then
  \begin{itemize}
    \item if $\mathcal{N}(x_{0},X) \subset \inter X$, then the map $\mathcal{F}$ has a unique zero $x_{*} \in X$.
      Moreover $x_{*} \in \mathcal{N}(x_{0},X)$.
    \item If $\mathcal{N}(x_{0},X) \cap X = \emptyset$, then $\mathcal{F}(x) \neq 0$ for all $x \in X$.
  \end{itemize}
\end{thm}

We applied the interval Newton-Moore method to a
problem of the form~\eqref{eq:problemform} given by the
sequence of 179 Poincar\'e maps obtained from the last step of the continuation procedure described in Subsection~\ref{subsec:continuation},
i.e. the step, where $\texttt{currentEpsRange}$ contains 0.0015.
Let $B_{\max}(0,r)$ denote an open ball of radius $r$ centered at 0 in maximum norm.
We obtained the following inclusion
\begin{equation}
  \mathcal{N}\left( 0, \overline{B_{\max}\left(0, 10^{-6}\right)} \right) \subset B_{\max}\left(0, 4.7926638 \times 10^{-14}\right),
\end{equation}
which, by Theorem~\ref{thm:intervalNewton}, implies the existence and local uniqueness of the periodic orbit.

\begin{rem}
 We report that we have succeeded with a verified continuation based on the interval Newton-Moore method
 for the whole parameter range of Theorem~\ref{thm:main2}, that is $\epsilon \in [1.5 \times 10^{-4}, 0.0015]$.
 Although we got a little extra information on the local uniqueness of the solution of the problem~\eqref{eq:problemform}, 
 we have decided to discard this result,
 as it was vastly outperformed in terms of computation time by the method of covering relations\footnote{Substituting
 the interval Krawczyk operator for the interval Newton operator did not resolve this issue, i.e. did not allow for greater widths
 in the parameter steps.}.
 It seems that the sequential covering process in the method of covering relations
 benefits more from the strong hyperbolicity than the interval Newton operator, hence allowing to make wider steps in the parameter range for such type of problems.
 However, for ranges of higher values of $\epsilon$ the interval Newton-Moore method
 was only several times slower than the one of covering relations (e.g. $\approx 7$ times in the range $[0.001, 0.0015]$),
 so we decided to state Theorem~\ref{thm:main3} in its current form to show that
 we have achieved a parameter value where the more widespread tool is already adequate to the task.
\end{rem}

\subsection{Proof of Theorem~\ref{thm:main4}}\label{subsec:thm4}

Our main tool in deducing the existence of the homoclinic orbit will be Theorem~\ref{thm:hom2}, together with 
Lemmas~\ref{lem:block},~\ref{lem:con} and~\ref{lem:sidedisk}.
The strategy loosely resembles the one given for the model example in Subsection~\ref{sec:covslowfast2}, which was portrayed in Figure~\ref{schemeFig2}.
However, the involved methods of shooting to diagonals, 
employed in the construction of the segment $S_{ul}$ in Lemma~\ref{lem:heuristic2} are not necessary,
and not used in practice. 

Similarly, as in the proof of Theorem~\ref{thm:main1} we introduce additional sections crossing the fast subsystem heteroclinics
and we use chains of segments instead of ``long'' segments $S_u$, $S_d$.
An additional modification we need to apply, is to construct a set alike an isolating segment, which allows to propagate
the unstable manifold of the zero equilibrium, starting from one of the faces of the isolating block containing the equilibrium. 
This is because rigorous integration was not accurate enough 
to propagate this manifold in the close proximity of the equilibrium.

We now proceed to describe the proof in details.
Contrary to the case of the periodic orbit, 
we consider the whole range  $\epsilon \in (0, 5 \times 10^{-5}]$ without subdivision, 
as attempting subdivisions did not allow us to significantly increase its upper bound.
The following steps are executed after fixing the $\epsilon$ range, to verify the assumptions of Theorem~\ref{thm:hom2}:

  \begin{enumerate}[leftmargin=*]
    \item  first, we perform a nonrigorous shooting with $\theta$ procedure in the fast subsystem~\eqref{eq:fastsub},
      alike the one in Subsection~\ref{subsec:thm11}, for $w:=0$,
      to find the value of the parameter $\theta$ equal to 
      \begin{equation}
        \texttt{thetaC} = 1.26491106 \approx \theta_{*} ,
      \end{equation} 
      near which there exists a suitable connecting orbit.
      The two ``corner points'' on the left can be computed directly and are given by  
      \begin{equation}
        \begin{aligned}
          \texttt{GammaDL} &= (0,0,0) = ( \Lambda_d(0),0), \\
          \texttt{GammaUL} &= (1,0,0) = ( \Lambda_u(0),0).
        \end{aligned}
      \end{equation}
      The point $\texttt{GammaDL}$ is the stationary point for the homoclinic orbit of the full system.
      
      After having set $\theta:=\texttt{thetaC}$
      we compute the other two points by shooting with $w$, as in Subsection~\ref{subsec:thm11}: 
      \begin{equation}
        \begin{aligned}
           \texttt{GammaUR} &= (0.73333334,0,0.12385185) \approx ( \Lambda_u(w^{*}),w^{*}), \\
          \texttt{GammaDR} &= (-0.26666667,0,0.12385185) \approx ( \Lambda_d(w^{*}),w^{*}).
         \end{aligned}
      \end{equation}
    
      The matrices given by the approximate eigenvectors of the linearization of the fast subsystem at points  
      $\texttt{GammaUR}$, $\texttt{GammaUL}$, $\texttt{GammaDR}$ are given by
      \begin{equation}
        \begin{aligned}
          \texttt{PUR} &= \left[ {\begin{array}{cc} 1 & 1 \\ 0.31622777 & -0.063245553 \end{array} } \right], \\
          \texttt{PUL} = \texttt{PDR} &= \left[ {\begin{array}{cc} 1 & 1 \\ 0.56920998 & -0.31622777 \end{array} } \right],
        \end{aligned}
      \end{equation}
      respectively.
    \item We set the range of parameter $\theta$ to
      \begin{equation}
        \begin{aligned}
        \texttt{theta}&:= \texttt{thetaC} + [-0.0025,  0.0025]\\ &=[1.26241106400572, 1.26741106400572].
      \end{aligned}
      \end{equation}
      The interval $\texttt{theta}$ forms a one-dimensional h-set with one exit direction and the change of coordinates
      given by a translation and rescaling to $[-1,1]$. It serves as the h-set $Z$ in Theorem~\ref{thm:hom2}.
    \item  We initialize three ``corner segments'' \texttt{ULSegment}, \texttt{URSegment} and \texttt{DRSegment} with data
      from Table~\ref{table:corseg2} as described in Paragraph~\ref{par:segments} and check that they are isolating segments for all $\theta \in \texttt{theta}$. 
      For checking the isolation formulas (S2b), (S3b) we subdivide enclosures of each of the respective faces of the exit and the entrance set into $150^{2}$ equal pieces. 
      \begin{table}[ht]
        \begin{center}
        \begin{tabular}{ | l | l | l | l |}
          \hline
          Segment &  \texttt{Front}, \texttt{Rear} & \texttt{P} & $(\texttt{a},\texttt{b}) = (\texttt{c},\texttt{d})$  \\ \hline
          \texttt{ULSegment} & $\texttt{GammaUL} \mp (0,0,0.001)$  & \texttt{PUL} & $( 1.8 \times 10^{-4}, 0.0021)$ \\ \hline
          \texttt{URSegment} & $\texttt{GammaUR} \mp (0,0,7 \times 10^{-4})$   & \texttt{PUR} & $(0.003, 0.005)$ \\ \hline
          \texttt{DRSegment} & $\texttt{GammaDR} \pm (0,0,0.002)$  & \texttt{PDR} & $(8 \times 10^{-4}, 0.013)$ \\ \hline
        \end{tabular}\caption{Initialization data for the three corner segments. The pair $(\texttt{a},\texttt{b})$
        determines the exit/entry direction widths of the segments and the difference $|\texttt{Front[2]} - \texttt{Rear[2]}|$ the central
          direction width.}\label{table:corseg2}
        \end{center}
      \end{table}
    \item We initialize two h-sets $\texttt{BU}$ and $\texttt{BS}$, which will form isolating blocks for $\texttt{F}$, by setting
      \begin{equation}
      \begin{aligned}
        u(\texttt{BU}) &:= u(\texttt{BS}) := 1,\\
        s(\texttt{BU}) &:= s(\texttt{BS}) := 2,\\
        c_{\texttt{BU}}^{-1}&:= \left[ {\begin{array}{ccc} 
             2.4 \times 10^{-5} & 8 \times 10^{-6} &-10^{-4} \\ 
             7.5794685 \times 10^{-6} & -5.0663182 \times 10^{-7} & 0 \\
            0 & 0 & 10^{-5}
          \end{array} } \right], \\
        c_{\texttt{BS}}^{-1}&:= \left[ {\begin{array}{ccc} 
             2 \times 10^{-4} & 2 \times 10^{-4} & -0.0013 \\ 
            6.3162238 \times 10^{-5} & -1.2665795 \times 10^{-5} & 0 \\
             0 & 0 & 1.3 \times 10^{-4}
          \end{array} } \right].
         \end{aligned}
      \end{equation}
      The first two columns of each of the matrices are formed by suitably rescaled eigenvectors of the fast subsystem,
      and the last one is the suitably rescaled tangent vector to the slow manifold -- see Remark~\ref{slowrem}.
      The norm of each column gives an indication of the size of the block in each direction,
      and is one of the program parameters that can be adjusted in order to complete the proof.

      We verify assumptions of Lemmas~\ref{lem:block},~\ref{lem:con} for $\theta \in \texttt{theta}$ to
      conclude that $\texttt{BU}$ and $\texttt{BS}$ are isolating blocks satisfying the cone condition, and in particular
      $W^u_{\texttt{BU}} (\texttt{GammaDL})$ and $W^s_{\texttt{BS}} (\texttt{GammaDL})$
      are, respectively, a horizontal and a vertical disk satisfying the cone condition, for all $\theta \in \texttt{theta}$. 
      In view of Theorem~\ref{szczelirPar} we can assert that these manifolds vary continuously with the parameter $\theta$.

      From Lemma~\ref{lem:sidedisk} we obtain that the intersection $W^s_{\texttt{BS}} (\texttt{GammaDL}) \cap |X_{\texttt{BS},3}|$
      forms a vertical disk in the boundary h-set $X_{\texttt{BS},3}$ (see Definition~\ref{defn:bdset}), which varies continuously with $\theta$.

      From the definition of the horizontal disk it follows that $W^u_{\texttt{BU}}(\texttt{GammaDL})$ has an intersection
      point with the first boundary h-set $|X_{\texttt{BU},1}|$, which we will denote by $W^u_1(\theta)$. 
      The point $W^u_1(\theta)$ is still too close to the equilibrium to reliably compute its trajectory by rigorous integration.
      To patch this numerical problem, we perform a simple phase space analysis to propagate it further from $\texttt{GammaDL}$.
      Namely, we construct an h-set $\texttt{BUext}$ such that
      \begin{equation}
      \begin{aligned}
        u(\texttt{BUext}) &:= 1,\\
        s(\texttt{BUext}) &:=2,\\
        c_{\texttt{BUext}}&:= \left[ {\begin{array}{ccc} 
             0.3 & 0 & 0 \\ 
            0 & 1 & 0 \\
             0 & 0 & 1 
           \end{array} } \right] \cdot c_{\texttt{BU}},
         \end{aligned}
      \end{equation}
      and verify that it is an isolating block. The block $\texttt{BUext}$ is simply an extension of the block $\texttt{BU}$ in the exit direction
      by a factor of $10/3$. Then, we verify the following Lyapunov-like condition:
      \begin{equation}\label{eq:patch}
         \langle \pi_u (c_{\texttt{BUext}}(x)), F(x) \rangle > 0 
        \quad \text{ for } x \in c_{\texttt{BUext}}^{-1}( [0.3,1] \times [-1,1]^2 ),
      \end{equation}
      where $\varphi$ is the local flow induced by $\dot{x} = F(x)$ and $\pi_u$ is the projection onto the first (exit) variable.
      We observe that the left-hand side of the inequality~\eqref{eq:patch} is equal to $\frac{d}{dt} (\pi_u \circ c_{\texttt{BUext}}) ( \varphi(t,x) )$.

      The condition~\eqref{eq:patch} implies that all trajectories starting in $c_{\texttt{BUext}}^{-1}( [0.3,1] \times [-1,1]^2 )$
      leave it in finite time. In particular, since $\texttt{BUext}$ is an isolating block, they can only leave the aforementioned set via the following 
      face of the boundary of $\texttt{BUext}$:
      \begin{equation}
        |X_{\texttt{BUext},1}|=c_{\texttt{BUext}}^{-1}( \{1\} \times [-1,1]^2 ).
      \end{equation}
      This in turn implies, that the forward trajectory of the point $W^u_1(\theta) \in |X_{\texttt{BU},1}| = c_{\texttt{BUext}}^{-1}( \{0.3\} \times [-1,1]^2 )$
      has to intersect $|X_{\texttt{BUext},1}|$. Consequently, the unstable manifold $W^u(\texttt{GammaDL})$ 
      has an intersection point with $|X_{\texttt{BUext},1}|$, that varies continuously with $\theta$, and which we will denote by $W^u_2(\theta)$.

      The reasoning above could be replaced by construction of a suitable isolating segment with support given by $c_{\texttt{BUext}}^{-1}( [0.3,1] \times [-1,1]^2 )$,
      two entry directions given by the entry directions of $\texttt{BUext}$ and the central direction given by the exit direction of $\texttt{BUext}$.
      In theory we could avoid this analysis by verifying that $\texttt{BUext}$ satisfies the cone condition, but in practice the set was too large
      and this evaluation failed.

    \item As in the case of the periodic orbit, 
      two intermediate sections  $\texttt{leftSection}$ and $\texttt{rightSection}$ are positioned in some distance from the slow manifolds.
     
      The section $\texttt{leftSection}$ is placed on the integration path between the block \texttt{BUext} and the segment \texttt{ULSegment}
      and the section $\texttt{rightSection}$ on the path between \texttt{URSegment} and \texttt{DRSegment}. 
      We define the Poincar\'e maps $\texttt{pmUL}$, $\texttt{pmUR}$ and $\texttt{pmDR}$ as in the proof for the periodic orbit:
      \begin{itemize}
        \item $\texttt{pmUL}$ is the Poincar\'e map from a subset of $\texttt{leftSection}$ to the affine section containing $X_{\texttt{ULSegment}, \text{ls}}$,
        \item $\texttt{pmUR}$ is the Poincar\'e map from $X_{\texttt{URSegment},\text{lu}}$ to $\texttt{rightSection}$,
        \item $\texttt{pmDR}$ is the Poincar\'e map from a subset of $\texttt{rightSection}$ to the affine section containing $X_{\texttt{DRSegment}, \text{rs}}$.
      \end{itemize}
      We will now briefly describe which covering relations are verified.

      We integrate the face $|X_{\texttt{BUext},1}|$ to $\texttt{leftSection}$ to enclose the image of $\texttt{WuL}$, defined as the map 
      that assigns to the value of the parameter $\theta$ the first intersection point of $W^u(\texttt{GammaDL})$ with $\texttt{leftSection}$.
      In particular, such map is continuous as a composition of a Poincar\'e map with the map $W^u_1 = W^u_1(\theta)$.
      The integration is performed three times:
      \begin{itemize}
        \item for $\theta$ set to the whole interval $\texttt{theta}$,
        \item for $\theta$ set to the left bound of interval \texttt{theta} (that is $ \theta:= \texttt{thetaC} - 0.0025$),
        \item for $\theta$ set to the right bound of interval \texttt{theta} (that is $ \theta := \texttt{thetaC} + 0.0025$).
      \end{itemize}
      
      Based on this calculation, 
      we create an h-set $\texttt{midLeftSet} \subset \texttt{leftSection}$ so that it is $\texttt{WuL}$-covered by a small margin by the h-set $\texttt{theta}$
      (see Definition~\ref{defn:altcover} and Lemma~\ref{covlemma}). 
      We reset the parameter $\theta$ to the whole interval $\texttt{theta}$ and proceed.

      We integrate the h-set $X_{\texttt{ULSegment}, \text{ls}}$ backward in time to $\texttt{leftSection}$ and 
      verify that $\texttt{midLeftSet}$ $\texttt{pmUL}$-backcovers $X_{\texttt{ULSegment}, \text{ls}}$.
 
      The h-set $X_{\texttt{URSegment}, \text{lu}}$ is integrated to $\texttt{rightSection}$, and, as in the case of the periodic orbit, we define
      an h-set $\texttt{midRightSet} \subset \texttt{rightSection}$, such that it is $\texttt{pmUR}$-covered 
      by $X_{\texttt{URSegment}, \text{lu}}$.
      Then, we integrate the h-set $X_{\texttt{DRSegment}, \text{rs}}$ backward in time to $\texttt{rightSection}$ and 
      verify that $\texttt{midRightSet}$ $\texttt{pmDR}$-backcovers $X_{\texttt{DRSegment}, \text{rs}}$.

      Altogether, we have the following covering relations:
      \begin{equation}
        \begin{aligned}
          \texttt{theta} &\cover{\texttt{WuL}} \texttt{midLeftSet} \backcover{\texttt{pmUL}} X_{\texttt{ULSegment}, \text{ls}}, \\
          X_{\texttt{URSegment},\text{lu}} &\cover{\texttt{pmUR}} \texttt{midRightSet} \backcover{\texttt{pmDR}} X_{\texttt{DRSegment}, \text{rs}},
        \end{aligned}
      \end{equation}
      where the coverings by \texttt{pmUL}, \texttt{pmUR}, \texttt{pmDR} hold for all $\theta \in \texttt{theta}$.

      Parameter $\texttt{div}$ describing partitioning of h-sets for the rigorous integration was set to 25,
      except for the h-set \texttt{theta}, for which no subdivision was needed.

    \item To close the loop, we connect the h-sets $X_{\texttt{ULSegment},\text{out}}$ and $X_{\texttt{URSegment},\text{in}}$ by a chain of segments
      \({\bf UpSegment} \) and $X_{\texttt{DRSegment},\text{out}}$ and $X_{\texttt{BS},3}$ by a chain of segments \({\bf DownSegment} \)
      as described in Subsection~\ref{subsec:chains}. 
      The number of isolating segments in \({\bf UpSegment} \) is set to 200 and in \({\bf DownSegment} \) to 400.  
      For verification of the isolation conditions (S2b), (S3b) in each chain we partition the enclosures of each of the
      faces of their exit and entrance sets into $110^{2}$ equal pieces.
      As with blocks and corner segments, this verification is performed for all $\theta \in \texttt{theta}$.
  \end{enumerate}

Same remarks about the choices of set sizes hold as for the computer assisted proof of Theorem~\ref{thm:main1}.
The wider the $\epsilon$ range is, the harder it is to verify the assumptions, 
since they follow from properties of the singular limit equation at $\epsilon=0$ (see Subsection~\ref{sec:covslowfast2}).
In addition $\texttt{theta}$ has to be chosen wide enough to generate the covering relation of $\texttt{WuL}$,
but narrow enough so all the other covering relations and all of the isolation conditions persist.

Our upper bound on $\epsilon$, equal to $5 \times 10^{-5}$, is three times smaller than the upper bound for existence of the periodic orbit
in Theorem~\ref{thm:main1}, mainly due to the fact, that we have to include a small range of $\theta$'s in most computations.
We did not attempt further continuation of the homoclinic orbit alike the one for the periodic orbit in Theorem~\ref{thm:main2}.
We think that a ``brute force'' continuation using covering relations only (i.e. by means of Theorem~\ref{thm:conmap}), 
would still be possible, perhaps with use of multiple precision interval arithmetic. 
However, we would like to find a more elegant continuation method, that would bypass the numerical instability of the problem.

\subsection{Technical data and computation times}

All computations were performed on a laptop equipped with Intel Core i7 CPU, 1.80 GHz processor, 4GB RAM
and a Linux operating system with gcc-5.2.0. We used the 568th Subversion revision of the CAPD library.
The programs were not parallelized.

Verification of assumptions of Theorem~\ref{thm:main1} took 236 seconds. Over 95\% of the processor time
was taken by verification of isolation for the chains of isolating segments.

Proofs of Theorems~\ref{thm:main2} and~\ref{thm:main3} were executed by the same program.
The validated continuation in Theorem~\ref{thm:main2} was the most time consuming part -- it took 4153 seconds.
Theorem~\ref{thm:main3} is formulated for a single parameter value; the proof here was instantaneous -- it finished within 2 seconds.

Verification of assumptions of Theorem~\ref{thm:main4} took 443 seconds. This time, over 98\% of the processor time
was taken by verification of isolation for the chains of isolating segments, mainly because
it was necessary to include a high number of 400 segments in the lower chain to obtain isolation -- see Subsection~\ref{subsec:thm4}.

We remark that the successful attempt to check the assumptions of Theorem~\ref{thm:main1} also for the range $\epsilon \in [10^{-4}, 1.5 \times 10^{-4}]$ (119 seconds)
saved us a lot of computation time. In theory we could have tried to use a validated continuation approach like in Theorem~\ref{thm:main2}
for this range. We tried it later for a subrange \nopagebreak $\epsilon \in [1.1 \times 10^{-4}, 1.5 \times 10^{-4}]$ (for the whole range 
execution of Newton's method
for the problem~\eqref{eq:problemform} within the nonrigorous part of the continuation algorithm failed due to enormous sizes of matrices to invert) 
-- it took 2571 seconds, that is over 20 times longer.
This indicates that construction of isolating segments around slow manifolds can be a valuable tool for proofs for
``regular'' parameter ranges (i.e. not including the singular perturbation parameter value) in systems with a very large separation of time scales.

\chapter{Concluding remarks}\label{chap:conclusions}

In this thesis we proved the existence of a periodic orbit and a homoclinic orbit in a fast-slow system
for an explicit range of the small parameter of the form $\epsilon \in (0,\epsilon_0]$,
where previous results in literature were given only for $\epsilon_0$ ``small enough''.
For the periodic orbit we also showed that the range is wide enough to succeed
with a validated continuation based on topological or $C^{1}$ methods at its upper bound.
Even though we restricted ourselves to analysis of the FitzHugh-Nagumo equations, 
our methods are general and should be applicable to other fast-slow systems of similar structure.

An intrinsic advantage of our techniques is that, once the topological theory is in place, it relies on
simple quantitative assumptions, such as enclosures of the vector field
or estimates on solutions over some compact sets, which can be verified on a computer without much effort.
Contrary to ours, the methods of 
GSPT require qualitative, sometimes very involved arguments for tracking suitable invariant manifolds
and their transversal intersections.
We admit that the analytical methods of GSPT give more insight into the nature of a given problem,
however it may be easier to deal with unproved hypotheses using a quantitative approach.

Results presented in this thesis can also be of interest to
researchers working in rigorous numerics, 
as by successfully adapting isolating segments into a computational framework
we managed to deal with a stiff, structurally unstable problem.
We have some hopes that isolating segments 
can replace rigorous integration for certain systems with strong expansion of error bounds,
such as ill-posed PDEs (see Subsection~\ref{subsec:Boussinesq}).

Below we propose several other problems 
from multiple time scale dynamics, that perhaps could be possible to tackle
in such explicit ranges of the small parameter, by extensions of methods developed in this thesis.
\begin{enumerate}[label=\emph{\arabic*.}] 
  \item \emph{Higher-dimensional slow manifolds.} So far we have dealt with the case of one-dimensional slow manifolds.
      By a suitable generalization of the concept of an isolating segment one could attempt to find closed orbits in
      equations with several slow variables such as the fast-slow predator-prey system~\cite{Mischaikow2} or the Koper model~\cite{Koper}.
      This would involve rigorous shadowing of orbits of the slow subsystem computed with a nonrigorous integrator.
    \item \emph{Loss of normal hyperbolicity.} We suspect that it would be possible to give
      topological arguments for a transition between the slow and the fast dynamics at fold points,
      where the slow manifold loses its normal hyperbolicity.  
      For such setting, blow-up techniques were used to prove that FitzHugh-Nagumo system exhibits pulses with oscillatory tails 
      for $\epsilon>0$ small~\cite{Carter}.
    \item \emph{Chaotic dynamics.} Another open question is to show the existence of horseshoes
      for explicit ranges of the small parameter. For $\epsilon>0$ small
      such dynamics was demonstrated in e.g. 
      the periodically forced van der Pol system~\cite{Haiduc} and in the fast-slow predator-prey system~\cite{Mischaikow3}.
      We expect that inclusion of isolating segments in proofs of existence of horseshoes based on covering relations 
      (like the one given in~\cite{GaliasZgliczynski}) should be relatively straightforward.
    \item \emph{Uniqueness and stability.} The questions of uniqueness and stability remain,
      both as orbits of ODEs, and as waves in the respective PDEs. Local
      uniqueness, stability and some bounds on stable and unstable manifolds of orbits
      can probably be achieved by the $C^1$ method of cone conditions~\cite{Capinski2} adapted to the fast-slow structure of the system.
      For stability of waves a computational version of the Evans function approach at the singular limit would most likely be necessary 
      (cf. Subsection~\ref{subsec:ArioliKoch}).
\end{enumerate}
Let us add to this list the numerical issues of further validated continuation
to some macroscopic values of $\epsilon$; in this thesis
we have performed it only for the periodic orbit.
We think that our upper bound on $\epsilon$ for the homoclinic orbit should also be achievable in near future.
Since computer assisted methods for ``regular'' ODEs are already well developed,
such continuation problems do not pose new theoretical challenges;
the main concern is how to design suitable continuation algorithms efficiently. 
To this end, one could try do adapt algorithms from nonrigorous continuation packages (see Subsection~\ref{subsec:further}),
which usually reduce the dynamical question to one or several two-point boundary value problems and solve them on a suitable mesh.
Such methods seem more reliable for tracking and nonrigorous continuation of unstable orbits, 
than an approach based on integration.

}

\clearpage
\phantomsection

\addcontentsline{toc}{chapter}{Bibliography}

\bibliography{thesis_bib}
\bibliographystyle{abbrv}


\end{document}